\documentclass[12pt]{amsart}
\usepackage{epsfig}
\usepackage{amssymb, amsmath}
\usepackage{amsfonts,graphics,epsfig}
\usepackage{mathrsfs}
\usepackage{pb-diagram, pb-xy}
\usepackage[all]{xy}
\usepackage{comment}
\usepackage{tikz}
\usepackage{setspace}
\usepackage{url}
\usepackage{soul}
\usepackage[left=1.5in,right=1.5in,top=1.5in,bottom=1.5in]{geometry}
\usepackage{color}   %May be necessary if you want to color links
\usepackage{hyperref}
\hypersetup{
    colorlinks=true, %set true if you want colored links
    linktoc=all,     %set to all if you want both sections and subsections linked
    linkcolor=red,%choose some color if you want links to stand out
        citecolor=blue,  
}
\usepackage{tcolorbox}
\usepackage{enumitem}

\newcommand{\mbR}{\mathbb{R}}
\newcommand{\mbC}{\mathbb{C}}
\newcommand{\mbZ}{\mathbb{Z}}
\newcommand{\mbQ}{\mathbb{Q}}

\def\mbN{\mathbb{N}}
\def\mbP{\mathbb{P}}

\newcommand{\<}{\leq}
\def\>{\geq}

\def\ve{\varepsilon}
\def\vphi{\varphi}

\def\subset{\subseteq}

\newcommand{\lrd}{\lfloor}
\newcommand{\rrd}{\rfloor}
\newcommand{\lru}{\lceil}
\newcommand{\rru}{\rceil}
\newcommand{\num}{\equiv}
\newcommand{\lin}{\sim}
\newcommand{\bir}{\dashrightarrow}
\newcommand{\D}{\Delta}
\def\mcA{\mathcal{A}}
\def\mcO{\mathcal{O}}

\def\mcC{\mathcal{C}}

\def\mcK{\mathcal{K}}

\def\mcP{\mathcal{P}}

\def\msI{\mathscr{I}}
\def\msL{\mathscr{L}}

\def\injective{\hookrightarrow}

\newtheorem{theorem}{Theorem}[section]
\newtheorem{lemma}[theorem]{Lemma}
\newtheorem{proposition}[theorem]{Proposition}
\newtheorem{corollary}[theorem]{Corollary}

\theoremstyle{remark}
\newtheorem{remark}[theorem]{Remark}
\theoremstyle{definition}
\newtheorem{definition}[theorem]{Definition}
\theoremstyle{definition}

\numberwithin{equation}{section}
\theoremstyle{definition}
\newtheorem{claim}[theorem]{Claim}
\newtheorem{assumption}[theorem]{Assumption}

\def\div{\operatorname{div}}

\def\Supp{\operatorname{Supp}}
\def\dim{\operatorname{dim}}
\def\codim{\operatorname{codim}}

\def\Ex{\operatorname{Ex}}

\def\NS{\operatorname{NS}}
\def\NE{\overline{\operatorname{NE}}}

\def\max{\operatorname{max}}

\def\Nef{\operatorname{Nef}}
\def\NA{\operatorname{\overline{NA}}}

\def\pt{\operatorname{pt}}

\def\hor{\operatorname{hor}}
\def\ver{\operatorname{ver}}
\def\Bs{\operatorname{Bs}}
\def\SBs{\operatorname{SBs}}
\def\mult{\operatorname{mult}}

\def\sing{\operatorname{\textsubscript{sing}}}
\def\red{\operatorname{red}}

\def\ex{\operatorname{ex}}
\def\nex{\operatorname{nex}}
\def\Center{\operatorname{Center}}
\def\gcd{\operatorname{gcd}}
\def\discrep{\operatorname{discrep}}
\def\pr{\operatorname{pr}}
\def\sm{\operatorname{\textsubscript{sm}}}
\def\sing{\operatorname{\textsubscript{sing}}}

\author{Omprokash Das}
\address{School of Mathematics\\
Tata Institute of Fundamental Research\\
Homi Bhabha Road, Navy nagar\\
Colaba, Mumbai 400005, India}
\email{omprokash@gmail.com}
\email{omdas@math.tifr.res.in}
\thanks{Omprokash Das is supported by the Start--Up Research Grant(SRG), Grant No. \# SRG/2020/000348 of the Science and Engineering Research Board (SERB), Govt. Of India. }

\author{Wenhao Ou}
\address{Institute of Mathematics\\
 Academy of Mathematics and Systems Science\\
  Chinese Academy of Sciences\\
   Beijing 100190, CHINA}
\email{wenhaoou@amss.ac.cn}

\date{}

\begin{document}
\title{On the Log Abundance for Compact K\"ahler $3$-folds}
\maketitle
\begin{abstract}
        In this article we show that if $(X, \Delta)$ is a log canonical compact K\"ahler $3$-fold  such that $K_X+\Delta$ is nef and the numerical dimension $\nu(K_X+\Delta)\neq 2$, then $K_X+\Delta$ is semi-ample.
\end{abstract}

\tableofcontents

\part{Introduction }

\section{Introduction}
The Minimal Model Program is one of the fundamental tools for birational classification of algebraic varieties. For a smooth projective variety $X$, it predicts that after finitely many elementary birational contractions we can find a variety $X'$ with mild singularities such that it is birational to $X$ and satisfies exactly one of the following two properties: 
\begin{enumerate}
	\item $K_{X'}$ is nef,
	\item there is a fibration $g:X'\to Z$ such that $\dim Z<\dim X'$ and the general fibers of $g$ are (log) Fano varieties.
\end{enumerate}
Observe that $K_{X'}$ being nef is a numerical property, so it doesn't reveal much geometric information about $X'$. The Abundance Conjecture says that when $K_{X'}$ is nef, it is in fact semi-ample, i.e. there is a $m\in\mbN$ such that line bundle bundle $\mcO_{X'}(mK_{X'})$ is generated by its global sections. This leads to a fibration of $X'$ whose general fibers are (log) Calabi-Yau varieties. So one can hope to classify $X'$ by studying these fibers.\\
For complex projective varietes, the abundance conjecture was proved in dimension $3$ by Miyaoka and Kawamata \cite{Miy87, Miy88, Miy88b}, \cite{Kaw92a}. The more general case of the log abundance was settled in \cite{KMM94} for $3$-dimensional complex projective varieties.\\
Between $2015$ and $2016$, Campana, Peternell and H\"oring established the fundamental results of the minimal model program for K\"ahler $3$-folds in \cite{HP15, HP16, CHP16}. More specifically, in \cite{CHP16} they showed that if $X$ is a compact K\"ahler $3$-fold with terminal singularities and $K_X$ is nef, then $K_X$ semi-ample. Building upon their results, in this article we prove the following case of the log abundance conjeture.
\begin{theorem}
\label{thm:lc-log-abundance}
Let $(X,\Delta)$ be a compact K\"ahler $3$-fold log canonical pair. Assume that $K_X+\Delta$ is nef. Then the following hold:
\begin{enumerate}
	\item (Non-Vanishing) There is a sufficiently divisible positive integer $m\in\mbN$ such that $H^0(X, \mcO_X(m(K_X+\Delta)))\neq \{0\}$.
	\item If the Kodaira dimension $\kappa(X, K_X+\Delta)>0$, then $K_X+\Delta$ is semi-ample.
	\item If the numerical dimension $\nu(X, K_X+\Delta)=0, 1 \text{ or } 3$, then $K_X+\Delta$ is semi-ample.
\end{enumerate}
\end{theorem}
Our strategy is similar to the proof of the log abundance theorem for projective $3$-folds as presented in \cite{Kol92} and \cite{KMM94}. Note that the proofs in \cite{Kol92, KMM94} make rigorous use of the log minimal model program for various reduction steps. Unfortunately, the log MMP for compact K\"ahler $3$-folds is not established in full generality in \cite{HP15, HP16, CHP16}. More specifically, if $(X, \D)$ is a $\mbQ$-factorial compact K\"ahler $3$-dimensional klt pair such that $X$ is uniruled and $K_X+\Delta$ is pseudo-effective, then the cone theorem and the existence of log minimal model for $(X, \D)$ is not established in the papers of Campana, Peternell and H\"oring. This is one of first obstacles we overcome in this article. Of course our proof of these results are inspired from and quite similar to those in \cite{HP15, HP16, CHP16}. There are numerous other technical difficulties we encounter while trying to follow the strategy as in \cite{Kol92} and \cite{KMM94}, many of these have to do with the fact that we work with varieties which are not necessarily projective.\\

\begin{remark}\label{rmk:missing-case}
	Note that in our main Theorem \ref{thm:lc-log-abundance} above we excluded the case $\nu(K_X+\D)=2$. For projective varieties this case is established in \cite{Kaw92a}, \cite{Kol92} and \cite{KMM94, KMM94c}. When $X$ is a compact K\"ahler $3$-fold with $\mbQ$-factorial terminal singularities and $K_X$ is nef and $\nu(K_X)=2$, it is claimed in \cite[Theorem 8.2, page 1013]{CHP16} that $K_X$ is semi-ample. However, we were not able to follow the proof in \cite[Theorem 8.2]{CHP16}.  The issue arises in the `Step 1' of the proof of \cite[Theorem 8.2]{CHP16} on page 1016, where a result of Kawamata, namely \cite[Theorem 9.6]{Kaw88} is incorrectly applied.\\
	 We note that the authors of \cite{CHP16} recently informed us that there is a more direct proof of \cite[Theorem 8.2]{CHP16} for the following cases (see \cite{CHP21}): (i) the algebraic dimension $a(X)\>1$, and (ii) $X$ is covered by a family of curves. So the only remaining case is when $a(X)=0$ and there are no subvarieties through very general points of $X$. In this case $X$ is conjectured to be bimeromorphic to a quotient of a torus.\\
	 It appears possible that the original proof of Kawamata in \cite{Kaw92a} could still lead to a proof in the K\"ahler case. One of the main tools in Kawamata's original proof (and also the one in \cite[Chapter 14]{Kol92}) is the use of Orbifold Chern classes on klt $3$-folds and an analog of the Bogomolov-Miyaoka-Yau (BMY) inequality for these Chern classes, see \cite[Lemma 2.7]{Kaw92a} and \cite[Theorem 10.13]{Kol92}. The proof of this new BMY inequality uses ample divisors in an essential way, namely, cutting down the dimension of the ambient variety to reduce all the computations to a surface. Obviously such a trick does not work for compact K\"ahler spaces. At the moment we do not know whether this inequality can be proved more directly by some analytic methods on a compact K\"ahler $3$-fold with klt singularities.	 
	 	
\end{remark}~\\

{\bf Acknowledgement.} We would like to thank Burt Totaro, Christopher Hacon, Andreas H\"oring and Fr\'ed\'eric Campana for many fruitful conversations. We would also like to thank the referee for carefully reading the manuscript, pointing out numerous typos and suggesting various improvements.

\section{Preliminaries}
An \emph{analytic variety} or simply a \emph{variety} is a reduced and irreducible complex space. A \emph{morphism} $f:X\to Y$ between two complex spaces is a \emph{holomorphic map} between complex spaces. An open set $U\subset X$ of a complex space $X$ is called \emph{Zariski open} if the complement $X\setminus U$ is a closed \emph{analytic} subset of $X$, i.e. there is a sheaf of ideals $\msI\subset\mcO_X$ such that $X\setminus U=\Supp(\mcO_X/\msI)$. Let $Z\subset X$ be a closed analytic subset of $X$. We say that an open set $U\subset X$ contains the \emph{general points} of an irreducible analytic subset $Z\subset X$ if $\emptyset\neq U\cap Z$ contains a Zariski open subset of $Z$. Let $f:X\to Y $ be a morphism between two irreducible complex spaces and $\mcP$ is a property. We say that the \emph{general} fibers  of $f$ satisfy $\mcP$ if the set $Z:=\{y\in Y: X_y=f^{-1}(y)\mbox{ does not  satisfy } \mcP\}\subset Y$ is contained in a proper closed analytic subset of $Y$ or equivalently, there is a Zariski open dense subset $U\subset Y$ such that all the fibers of $f$ over $U$ satisfy $\mcP$. Similarly, we say that  the \emph{very general} fibers of $f$ satisfy $\mcP$, if the set $Z$ is contained in an at most countable union of proper closed analytic subsets of $Y$.\\

A proper morphism $f:X\to Y$ between two irreducible complex spaces is called \emph{generically finite} if there exists a Zariski open dense subset $V\subset Y$ such that $f|_{f^{-1}(V)}:f^{-1}(V)\to V$ is a finite morphism. Let $f:X\to Y$ be a proper morphism from a normal variety $X$ and $D$ is a prime Weil divisor on $X$. We say that $D$ is \emph{horizontal} over $Y$ if $f(D)=Y$, otherwise $D$ is called \emph{vertical} over $Y$. For any $\mbR$-Cartier divisor $\Delta$ on $X$, we can uniquely decompose $\Delta$ into horizontal and vertical parts over $Y$ as: $\Delta=\Delta_{\hor}+\Delta_{\ver}$, where $\Delta_{\hor}$ corresponds to the horizontal part of $\D$ over $Y$ and $\Delta_{\ver}$ corresponds to the vertical part.\\

Let $f:X\to Y$ be a morphism of normal varieties and $D$ is a $\mbQ$-Cartier $\mbQ$-divisor on $X$. We will write $D\sim_{\mbQ, f}0$ or $D\sim_{\mbQ, Y}0$ to mean that there is a $\mbQ$-Cartier divisor $L$ on $Y$ such that $D\sim_{\mbQ} f^*L$.\\

{\bf Canonical Divisor.} Let $X$ be a normal analytic variety. Then the canonical sheaf $\omega_X$ is defined as $\omega_X:=(\wedge^{\dim X}\Omega_X^1)^{**}$. Note that, in general $\omega_X$ does not correspond to an Weil divisor on $X$. However, for simplicity and to keep notations as similar as possible to the projective MMP, we will abuse notation in this article and will denote the canonical sheaf by the standard divisorial notation $K_X$. We will often write $\mcO_X(K_X)$ to mean $\omega_X$. We note that this notational abuse does not create any complication in this article, except one case, see the proof of Claim \ref{clm:cbf}. This is the only place where we had to recognize that $K_X$ is truly not represented by an Weil divisor and thus we can not push it forward by a morphism as a cycle.\\

 An $\mbR$-divisor $D$ on $X$ is called a \emph{boundary} (resp. \emph{pure boundary}, resp. \emph{sub-boundary}, resp. \emph{pure sub-boundary}) if the coefficients of $D$ belong to the interval $[0, 1]$ (resp. $[0, 1)$, resp. $(-\infty, 1])$, resp. $(-\infty, 1)$).\\
 
 Let $X$ be a normal variety and $\Delta$  a $\mbQ$-divisor on $X$ such that $K_X+\D$ is $\mbQ$-Cartier. We say that the pair $(X, \D)$ has \emph{terminal} (resp. \emph{canonical}, resp. \emph{klt}, resp. \emph{plt}, resp. \emph{lc}) singularities if $\Delta$ is \emph{effective} and the pair satisfies the defintion as in \cite[Defintion 2.8]{Kol13}. Moreover, if $\Delta$ is not necessarily effective, then we will refer the corresponding singularities as \emph{sub-terminal} (resp. \emph{sub-canonical}, resp. \emph{sub-klt}, resp. \emph{sub-plt}, resp. \emph{sub-lc}). We say that $(X, \Delta)$ has \emph{dlt} singularities if $\Delta$ is effective and the pair satisfy the defintion as in \cite[Def. 2.8]{Kol13}.

In the following we collect some important technical definitions. For a more detailed discussion, we encourage the reader to look at \cite{HP16, HP15, CHP16} and the references therein.
\begin{definition}\label{def:lots-of-definitions}
	
\begin{enumerate}[label=(\roman*)]
	\item An analytic variety $X$ is  called \emph{K\"ahler} if there exists a closed positive $(1, 1)$-form $\omega\in\mcA_{\mbR}^{1, 1}(X)$ such that the following holds: for every point $x\in X$ there exists an open neighborhood $x\in U$ and a closed embedding $\iota_U:U\to V$ into an open set $V\subset\mbC^N$, and a strictly plurisubharmonic $C^\infty$ function $f:V\to\mbR$ such that $\omega|_{U\cap X_{\sm}}=(i\partial\bar{\partial}f)|_{U\cap X_{\sm}}$. Here $X_{\sm}$ is the smooth locus of $X$.
	
	\item\label{item:c-manifolds} A compact analytic variety $X$ is said to belong to \emph{Fujiki's class} $\mcC$ if one of the following equivalent conditions are satisfied:
	\begin{enumerate}
		\item $X$ is a meromorphic image of a compact K\"ahler variety $Y$, i.e., there exists a dominant meromorphic map $f:Y\bir X$ from a compact K\"ahler variety $Y$  (see \cite[4.3, page 34]{Fuj78}).
		\item $X$ is a holomorphic image of compact K\"ahler manifold, i.e., there is a surjective morphism $f:Y\to X$ from a compact K\"ahler manifold $Y$ (see \cite[Lemma 4.6]{Fuj78}).
		\item $X$ is bimeromorphic to a compact K\"ahler manifold (see \cite[Theorem 3.2, page 51]{Var89}).\\
	\end{enumerate}
	
	\item On a normal compact analytic variety $X$ we replace the use of N\'eron-Severi group $\NS(X)_{\mbR}$ by $H^{1, 1}_{\rm BC}(X)$, the Bott-Chern cohomology of real closed $(1, 1)$-forms with local potentials or equivalently, the closed bi-degree $(1, 1)$-currents with local potentials. See \cite[Definition 3.1 and 3.6]{HP16} for more details. More specifically, we define
\[ N^1(X):=H^{1,1}_{\rm BC}(X).\]

	\item If $X$ is in Fujiki's class $\mcC$ and has \emph{rational singularities}, then from \cite[Eqn. (3)]{HP16} we know that $N^1(X)=H^{1, 1}_{\rm BC}(X)\subset H^2(X,\mathbb R)$. In particular, the intersection product can be defined in $N^1(X)$ via the cup product of $H^2(X, \mbR)$.
	
	\item Let $X$ be a normal compact analytic variety contained in Fujiki's class $\mcC$. We define $N_1(X)$ to be the vector space of real closed currents of bi-dimension $(1, 1)$ modulo the following equivalence relation: $T_1\num T_2$ if and only is 
	\[
		T_1(\eta)=T_2(\eta)
	\]
for all real closed $(1, 1)$-form with local potentials.

	\item We define $\NA(X)\subset N_1(X)$ to be the closed cone generated by the classes of positive closed currents. The Mori cone $\NE(X)\subset\NA(X)$ is defined as the closure of  the cone of currents of integration $T_C$, where $C\subset X$ is an irreducible curve. 
	
	\item Let $X$ be a normal compact analytic variety and $u\in H^{1,1}_{\rm BC}(X)$. Then $u$ is called \emph{pseudo-effective} if it can be represented by a bi-degree $(1, 1)$-current $T\in \mathcal D ^{1,1}(X)$ which is locally of the form $\partial \bar \partial f$ for some psh function $f$. It is called \emph{nef} if it can be represented by a form $\alpha$ with local potentials such that for some K\"ahler form $\omega $ on $X$ and for every $\epsilon >0$, there exists a $\mathcal C ^\infty$ function $f_\epsilon\in\mcA^0(X)$ such that $\alpha + i\partial \bar \partial f_\epsilon \geq -\epsilon \omega$. 
	
	\item The {\it nef cone} ${\rm Nef}(X)\subset N^1(X)$ is the cone generated by nef cohomology classes. Let $\mcK$ be the open cone in $N^1(X)$ generated by the classes of K\"ahler forms. By \cite[Proposition 6.1.(iii)]{Dem92} (also see \cite[Remark 3.12]{HP16}) it follows that the nef cone $\Nef(X)$ is the closure of $\mcK$, i.e. $\Nef(X)=\overline{\mcK}$.
	
\item We say that a normal variety $X$  is $\mbQ$-factorial if for every prime Weil divisor $D\subset X$, there is a positive integer $k>0$ such that $kD$ is a Cartier divisor, and for the canonical sheaf $\omega_X:=\wedge^{\dim X}\Omega^1_X$, there is a positive integer $m>0$ such that $(\omega _X^{\otimes m})^{**}$ is a line bundle. It is well known that if $X$ is a $\mbQ$-factorial 3-fold and $X\dasharrow X'$ is a flip or divisorial contraction, then $X'$ is also $\mbQ$-factorial.

\item A $\mbQ$-Cartier divisor $D$ is called \emph{$\mbQ$-effective} if   $mD$ is a Cartier divisor and the linear system $|mD|$ is not empty for some $m\in\mbN$. 

\item Let $f:X\to Y$ be a projective fibration  between two normal K\"ahler varieties. Then $f$ is called a $K_X$-Mori fiber space if the following conditions are satisfied:
                \begin{enumerate}
                        \item $X$ and $Y$ are $\mbQ$-factorial varieties and $X$ has terminal singularities,
                        \item $-K_X$ is $f$-ample, and
                        \item the relative Picard number $\rho(X/Y)=1$. 
                \end{enumerate}

\item Let $(X, B)$ be a log canonical pair and $f:X\to Y$ a contraction of a $(K_X+B)$-negative extremal ray $R$ of $\NA(X)$. Then we say that $f:X\to Y$ is a $(K_X+B)$-divisorial contraction (resp. $(K_X+B)$-flipping contraction) if $f$ is bimeromorphic and $\codim_X\Ex(f)=1$ (resp. bimeromorphic and $\codim_X \Ex(f)\>2$).  
\end{enumerate}

\end{definition}~\\

The following are some easy but useful results which will be used throughout the paper.  

\begin{lemma}[Negativity Lemma]\label{lem:negativity}\cite[Lemma 1.3]{Wan21}
	Let $f:X\to Y$ be a proper bimeromorphic morphism between two normal analytic varieties. Let $B$ be a $\mbQ$-Cartier $\mbQ$-divisor on $X$ such that $-B$ is $f$-nef. Then $B$ is effective if and only if $f_*B$ is effective.
\end{lemma}

\begin{lemma}\label{lem:finite-pullback}
        Let $f:X\to Y$ be a finite surjective   morphism between two normal varieties. Let $D\>0$ be an effective integral Weil divisor on $Y$. Then $f^*D$ is a well defined effective integral Weil divisor on $X$ such that $\Supp (f^*D)=f^{-1}(\Supp D)$.
\end{lemma}

\begin{proof}
        First we explain how to take pullback of a Weil divisor under a finite map. Since $f$ is a finite surjective proper morphism and $X$ and $Y$ are both normal, there exists a dense Zariski open set $V$ contained in the smooth locus of $Y$ such that $f^{-1}(V)$ is also contained in the smooth locus of $X$, and that $\codim_Y (Y\backslash V)\>2$ and $\codim_X(X \backslash f^{-1}(V))\>2$. Then by restricting $D$ to $V$ we see that the irreducible components of $D|_V$ are in bijection with that of $D$. Since $V$ is smooth, $D|_V$ is a Cartier divisor on $V$. Since $\codim_X(X \backslash f^{-1}(V))\>2$, $f^*(D|_V)$ extends to a unique divisor on $X$ which we define to be the $f^*D$. By construction, we have $\Supp (f^*D)=f^{-1}(\Supp D)$.
\end{proof}~\\

\begin{lemma}\label{lem:ramification-formula}
Let $f:X\to Y$ be a finite surjective   morphism between two normal varieties. Let $B=\sum b_iB_i$ be a $\mbQ$-divisor on $Y$ such that $\sum B_i$ contains the branch locus of $f$ (here we allow $b_i=0$). Let $(f^*\sum B_i)_{\red}=\sum D_{ij}$, where $f(D_{ij})=B_i$ for all $j$ and $i$. Let $e_{ij}$ be the branching order of $f$ around general points of $D_{ij}$. Then we have the following ramification formula
\begin{equation}\label{eqn:ramification-formula-with-boundary}
        f^*(K_Y+B)= K_X+\sum_{i, j} (1-(1-b_i)e_{ij})D_{ij}.
\end{equation} 
Moreover, if we set $B':=\sum_{i, j} (1-(1-b_i)e_{ij})D_{ij}$, then $(X, B')$ is  sub-klt (resp. sub-lc, resp. sub-lc) if and only if $(Y, B)$ is sub-klt (resp. sub-lc, resp. sub-lc).         
\end{lemma}

\begin{proof}
        Since in the first part we are only interested about divisors, and $X$ and $Y$ are normal, by removing some appropriate closed analytic subsets of codimension at least $2$   from $X$ and $Y$, we may assume that $X$ and $Y$ are both smooth. Then we have
        \begin{equation}\label{eqn:ramification-formula}
                K_X=f^*K_Y+R,
        \end{equation}
where $R$ is the ramification divisor.\\
Now observe that we have $f^*B=\sum_{i, j} b_ie_{ij}D_{ij}$, and $R=\sum_{i, j}(e_{ij}-1)D_{ij}$. Then \eqref{eqn:ramification-formula-with-boundary} follows from \eqref{eqn:ramification-formula}.

The second part follows from \cite[Proposition 5.20]{KM98}.

\end{proof}~\\

The following result is very useful in practice. The proof below is suggested by Burt Totaro.
\begin{lemma}[Galois Closure]\label{lem:galois-closure}
 Let $f:X\to Y$ to be a finite surjective morphism between normal varieties. Then there exists a normal variety $\widetilde X$, a finite morphism $\tilde f:\widetilde X\to Y$ and a finite group $G$ acting on $\widetilde X$ such that $\tilde f$ factorizes through $f$ and $Y\cong \widetilde X/G$.   
\end{lemma}

\begin{proof}
    Let $Z\subset Y$ be the branch locus of $f$. Then $Z$ is an analytic subset of $Y$ and $g:=f|_{f^{-1}U}: f^{-1}U\to U$ is a finite \'etale morphism, where $U:=Y\setminus Z$. In particular, $g$ is a topological covering space. Let $d$ be the degree of $g$ and $V=f^{-1}U$. We fix a point $v_0\in V$.  Then $H:=g_*\pi_1(V, v_0)$ is a subgroup of $\pi_1(U, g(v_0))$ of index $d$. Now define $\widetilde H:=\cap_{\gamma\in \pi_1(U, g(v_0))}\gamma^{-1}H\gamma$. Then $\widetilde H$ is the largest normal subgroup of $\pi_1(U, g(v_0))$ contained in $H$ and the index $[\pi_1(U, g(v_0)):\widetilde H]$ is finite. Let $\tilde g:\widetilde V\to U$ be the covering space corresponding to $\widetilde H$. 
    Then $\widetilde{V}$ is naturally an analytic space, and $\tilde{g}$ is the Galois closure of  $g:V\to U$. 
    Let $G$ the group of deck transformations of $\tilde g:\widetilde V\to U$. 
    Then $\widetilde V/G\cong U$. By \cite[Theorem XII.5.4]{SGA1}, $\tilde g:\widetilde V\to U$ extends to a unique finite morphism $\tilde f:\widetilde X\to Y$ of analytic varieties. Furthermore, from the proof of \cite[Theorem XII.5.4]{SGA1}, we see that $G$ acts naturally on $\widetilde{X}$ and  $\widetilde X/G\cong Y$.
\end{proof}

\begin{lemma}\label{lem:semiample-preserved}
	Let $f:X\to Y$ be a proper surjective morphism between two normal varieties and $\msL$ is a line bundle on $Y$. Then $\msL$ is semi-ample if and only if $f^*\msL$ is semi-ample. 
\end{lemma}

\begin{proof}
	Clearly if $\msL$ is semi-ample, then $f^*\msL$ is semi-ample. So assume that $f^*\msL$ is semi-ample. Replacing $f^*\msL$ by a suitable multiple we may assume that $f^*\msL$ is generated by global sections. By Stein factorization it is enough to consider two cases: (i) $f$ has connected fibers, i.e. $f_*\mcO_X=\mcO_Y$, and (ii) $f$ is a finite morphism.\\
In the case (i) we have $H^0(Y, \msL)\cong H^0(X, f^*\msL)\neq \{0\}$. Choose $y\in Y$ and $x\in f^{-1}(y)$. Since $f^*\msL$ is globally generated, there is a global section $s\in H^0(X, f^*\msL)$ such that $s$ does not vanish at $x$. Let $t\in H^0(Y, \msL)$ be the unique global section of $\msL$ such that $f^*t=s$. Then $t$ does not vanish at $y$, and hence $\msL$ is globally generated.\\

In case (ii) $f$ is a finite morphism. Replacing $f$ by its Galois closure (see Lemma \ref{lem:galois-closure}) we may assume that $f:X\to Y$ is  Galois  with Galois group $G$. Choose $y\in Y$ and let $s\in H^0(X, f^*\msL)$ be a section such that $s$ does not vanish at any point of $f^{-1}(y)$. Let $s_1, s_2,\ldots, s_d$ be the set of all Galois conjugates of $s$. Then $s_i$ does not vanish at any point of $f^{-1}(y)$ for any $i\in\{1,\ldots, d\}$. Now observe that $\otimes_{i=1}^d s_i$ is a $G$-invariant global section of $f^*\msL^d$. Therefore there exists a section $t\in H^0(Y, \msL^d)$ such that $f^*t=\otimes_{i=1}^d s_i$. Then clearly $t$ does not vanish at $y\in Y$, and hence $\msL^d$ is globally generated.  
\end{proof}~\\

\begin{lemma}\label{lem:mmp-preserving-terminal-singularities}
        Let $X$ be a variety with $\mbQ$-factorial terminal singularities such that $(X, \Delta)$ has canonical singularities for some $\Delta\>0$. Let $f:X\to Z$ be a projective morphism. Suppose that we run a $(K_X+\Delta)$-MMP over $Z$ which ends with a minimal model $g:(X', \Delta')\to Z$ over $Z$, i.e., $K_{X'}+\Delta'$ is $g$-nef. Further, assume that this MMP does not contract any component of $\Delta$. Then $X'$ has terminal singularities.  
\end{lemma}

\begin{proof}
        By induction it is enough to show that every $(K_X+\Delta)$-flip or divisorial contraction over $Z$ preserves the singularity type of $X$. 
        Let $f:X\to Y$ be a $(K_X+\Delta)$-flipping contraction over $Z$ and $f':X'\to Y$ its flip over $Z$. Set $\Delta':=\phi_*\Delta$, where $\phi:X\bir X'$ is the induced bimeromorphic map. Then $(X', \Delta')$ has canonical singularities. We will show that $X'$ has terminal singularities. To the contrary assume that there exists an exceptional divisor $E$ over $X$ such that $a(E, X')\<0$. Then $a(E, X', \Delta')=0$, since $(X', \Delta')$ is canonical. Note that $E$ is also an exceptional divisor over $X$, since $\phi:X\bir X'$ is an isomorphism in codimension $1$. We consider two cases below.\\

\textbf{Case I:} $\Center_X(E)\subset \Ex(f)$ or $\Center_{X'}(E)\subset \Ex(f')$. In this case from \cite[Lemma 3.38]{KM98} it follows that $a(E, X, \Delta)<a(E, X', \Delta')=0$. This is a contradiction, since $(X, \Delta)$ is canonical.\\

\textbf{Case II:} $\phi$ is an isomorphism around some neighborhoods of general points of $\Center_X(E)$ and $\Center_{X'}(E)$. In this case $a(E, X)=a(E, X')\<0$, this is a contradiction since $X$ has terminal singularities.\\

Now we will consider the case of divisorial contraction. Assume that $\psi:X\to Y$ is a $(K_X+\Delta)$-divisorial contraction over $Z$. Set $\Delta_Y:=\psi_*\Delta$. Since $(X, \Delta)$ is canonical and $\psi$ does not contract any component of $\Delta$, by \cite[Corollary 4.3(2)]{KM98} $(Y, \Delta_Y)$ has canonical singularities. The rest of the argument works similarly as above.

\end{proof}

\section{Some MMP Results}
In this section we will collect some results of MMP which are known for projective varieties but not for K\"ahler spaces. One such result is the termination of flips for dlt pairs; we will show that Kawamata's proof in \cite{Kaw92} works of K\"ahler spaces after some modifications.\\
In \cite{Sho92} Shokurov stated and proved the existence of flips for $3$-folds klt pairs where the underlying variety $X$ is allowed to be analytic. Based on this result one can easily deduce the following version of the existence of dlt flips; alternatively, see the arguments in \cite[Theorem 4.3]{CHP16}.  
\begin{theorem}[Existence of flips]\cite{Sho92}\cite[Theorem 4.3]{CHP16}\label{thm:existence-of-flips}
        Let $(X, \Delta)$ be a $\mbQ$-factorial compact K\"ahler $3$-fold dlt pair. Let $\vphi:X\to Y$ be a $(K_X+\Delta)$-flipping contraction. Then the flip $\vphi^+:X^+\to Y$ exists. Moreover, $X^+$  is a normal $\mbQ$-factorial compact K\"ahler 3-fold and $(X^+, \Delta^+)$ is dlt, where $\Delta^+=(\vphi^+)^{-1}_*(\vphi_*\Delta)$.
\end{theorem}

Special termination follows similarly from \cite{Sho92}, more specifically, the proof as explained in \cite{Fuj07} works since the MMP for dlt K\"ahler surfaces is known (see Section \ref{sec:surface-mmp}).
\begin{theorem}[Special Termination]\cite[Theorem 4.20]{CHP16}\cite{Fuj07, Sho92, Kol92}\label{thm:special-termination}
Let $(X, \Delta)$ be a $\mbQ$-factorial compact K\"ahler $3$-fold dlt pair. Set $(X_0, \Delta_0)=(X, \Delta)$ and let
\[
        (X_0, \Delta_0)\bir (X_1, \Delta_1)\bir\cdots (X_i, \Delta_i)\bir\cdots
\]
be a sequence of $(K_X+\Delta)$-flips. Then for all $i\gg 0$, the flipping locus (and the flipped locus) is disjoint from $\lrd\Delta_i\rrd$. In particular, if for every $i$ the flipping locus is contained in the support of $\lrd\Delta_i\rrd$, then the sequence of $(K_{X_i}+\Delta_i)$-flips terminates.
\end{theorem}

 \subsection{Termination of dlt flips} In this subsection we will show that Kawamata's proof	\cite{Kaw92} of termination of dlt flips (for algebraic varieties) works for compact K\"ahler $3$-folds after appropriate modifications. First note that Kawamata's notion of `weakly log terminal' singularity is what we call \emph{divisorially log terminal} or \emph{dlt} singularity in this article. 
\begin{theorem}\label{thm:termination}
        Let $(X, B)$ be a $\mbQ$-factorial compact K\"ahler $3$-fold dlt pair. Let
        \begin{equation}\label{eqn:termination}
               \xymatrix{ (X_0, B_0)\ar@{-->}[rr]\ar[dr]_{\vphi_0} && (X_1,B_1)\ar@{-->}[rr]\ar[dl]^{\vphi_0^+}\ar[dr]_{\vphi_1} && (X_2, B_2)\ar@{-->}[rr]\ar[dl]^{\vphi_1^+}\ar[dr]_{\vphi_2} &&\\
			   & Y_0 & & Y_1 && Y_2
			    }
        \end{equation}
        be a sequence of $(K_X+B)$-flips, where $(X_0, B_0):=(X, B)$. Then this is a finite sequence, i.e., every sequence of $(K_X+B)$-flips terminates.
\end{theorem}

\begin{remark}\label{rmk:kawamata's-proof}
Note that in Kawamata's proof in \cite{Kaw92} there are some terminologies, arguments and tools which are either not available on analytic varieties or hard to work with. We list them out below.
	\begin{enumerate}
		\item By special termination (see Theorem \ref{thm:special-termination}) after a finitely many steps we may assume that $\lrd B_i\rrd\cap\Ex(\vphi_i)=\emptyset$ for $i\gg 0$. In Kawamata's proof he replaces $X$ by $X\setminus\lrd B_i\rrd$ to assume that $(X, B)$ is klt for all $i\gg 0$. However, we can not do the same since the MMP is not known for open K\"ahler spaces. Instead in our case we will replace $B$ by $B_i-\lrd B_i\rrd$ so that $(X, B_i)$ is klt for $i\gg 0$; note that we can do this since $X$ is $\mbQ$-factorial.
		
		\item For a klt pair $(X, B)$, the number of exceptional divisors $E$ over $X$ with discrepancy $a(E, X, B)\<0$ is an invariant of the pair $(X, B)$. Kawamata denotes this invariant by $e(X, B)$. In page 655 of \cite{Kaw92} he redefines the invariant $e(X, B)$ using \emph{exceptional discrete valuations} of the function field $\mbC(X)$. However, for an analytic variety $X$, the field of meromorphic functions $\mbC(X)$ has transcendental degree over $\mbC\<\dim X$ in general, consequently, we can not characterize $e(X, B)$ using valuations in the analytic case. However, this does not create any problem for us, since Kawamata only used valuative characterization of $e(X, B)$ for convenience.
		
		\item We note that there are other more subtle technical difficulties which appear when we try to use Kawamata's arguments in the analytic case. They are discussed in the proof of Theorem \ref{thm:termination} below.
	\end{enumerate} 
\end{remark}

Now we will start proving Theorem \ref{thm:termination} step by step following Kawamata's idea as in \cite{Kaw92}. We will start with the special case of $e(X, B)=0$ as in \cite[Lemma 4]{Kaw92}.

\begin{theorem}\label{thm:terminal-termination}\cite[Lemma 4]{Kaw92}
	Theorem \ref{thm:termination} holds if $(X, B)$ is klt and $e(X, B)=e(X_0, B_0)=0$.
\end{theorem}

\begin{proof}
	Write $B=\sum_i b_{i}S_i$, $0\<b_{i}<1$ for all $i$ and let $S^n_i$ be the pushforward of $S_i$ on $X_n$. Note that we may assume that either $B=0$ or there is at least one component, say $S_1$ of $B$ such that $S^n_1\cap\Ex(\vphi_n)\neq\empty$ for infinitely many $n$, since otherwise we can remove all the components $S_i$ of $B$ (as $X$ is $\mbQ$-factorial) and replace the pairs $(X_n, B_n)$ by $(X_n, 0)$ for all $n\gg 0$. The rest of the proof works as in the proof of \cite[Lemma 4]{Kaw92}.\\
\end{proof}

Now we will show the existence of a terminal model for klt pairs as in \cite[Theorem 5]{Kaw92}. Note that Kawamata's original proof works in our case with minor changes, however, since this result is of independent interest to us and will be repeatedly used in other sections, we will present a detailed proof here. 
\begin{theorem}[Terminalization]\label{thm:terminal-model}\cite[Theorem 5]{Kaw92}
Let $(X, B)$ be a compact K\"ahler $3$-fold klt pair. Then there exists a $\mbQ$-factorial terminal pair $(Y, B_Y)$ and a projective bimeromorphic morphism $f:Y\to X$ such that 
\begin{enumerate}
	\item $K_Y+B_Y=f^*(K_X+B)$, and
	\item the number of exceptional divisors of $f$ is equal to $e(X, B)$.
\end{enumerate}	
\end{theorem}

\begin{proof}
	Let $g_0:Y_0\to X$ be a log resolution of $(X, B)$ such that $K_{Y_0}+B_0=g^*_0(K_X+B)+F$, where $B_0$ and $F$ are both effective $\mbQ$-divisors without common components, and $\Supp B_0$ is smooth (see \cite[Proposition 2.36(1)]{KM98}). Then ${g_0}_*B_0=B$ and $(Y_0, B_0)$ is terminal (see \cite[Corollary 2.31(3)]{KM98}). Now we will run $(K_{Y_0}+B_0)$-MMP over $X$. Note that since $g_0$ is a projective morphism, $X$ is compact and $(Y_0, B_0)$ is klt, the existence of $(K_{Y_0}+B_0)$-negative extremal rays over $X$ and the corresponding contractions (over $X$) are guaranteed by the Cone and Contraction theorem of Nakayama \cite[Theorem 4.12]{Nak87}. Existence of flips follows from Theorem \ref{thm:existence-of-flips}, and termination from Theorem \ref{thm:terminal-termination}, since $(Y_0, B_0)$ has terminal singularities. So we only need to show that the terminal singularities are preserved at every step of this MMP. By \cite[Corollary 3.43]{KM98} it is enough to show that if $(Y_i, B_i)$ is terminal and $\vphi_i:(Y_i, B_i)\to (Y_{i+1}, B_{i+1})$ is a divisorial contraction associated to a $(K_{Y_0}+B_0)$-MMP over $X$, then $\Ex(\vphi_i)\not\subset \Supp(B_i)$. To the contrary assume that $E$ is an irreducible component of $B_i$ which is contracted by $\vphi_i$. Then $a(E, X, B)<0$. Let $g_i:Y_i\to X$ be the induced projective bimeromorphic morphism. Then we have $K_{Y_i}+B_i=g^*_i(K_X+B)+F_i$, where $F_i\>0$ is the pushforward of $F$ on $Y_i$. Let $C\subset Y_i$ be a curve contained in a fiber of $\vphi_i$. Then $-(K_{Y_i}+B_i)\cdot C>0$ and $g_i(C)=\pt$. In particular, $F_i\cdot C<0$ and thus $E$ is a component of $F_i$, and hence $a(E, X, B)>0$, which is a contradiction.\\
	
	Finally assume that this MMP ends with $(Y_N, B_N)$, i.e., $K_{Y_N}+B_N$ is $g_N$-nef and $K_{Y_N}+B_N=g^*_N(K_X+B)+F_N$. Then by the Negativity Lemma \ref{lem:negativity} $F_N\<0$, and thus $F_N=0$, since it is effective. In particular, we have $K_{Y_N}+B_Y=g^*_N(K_X+B)$. Set $(Y, B_Y):=(Y_N, B_N)$ and $f:=g_N$ and we are done.
		
\end{proof}

Now we will prove Theorem \ref{thm:termination} as in \cite[Theorem 1]{Kaw92}.

\begin{proof}[Proof of Theorem \ref{thm:termination}]
   First we reduce the problem to the klt case. By special termination (Theorem \ref{thm:special-termination}), after finitely many steps the flipping locus does not intersect the support of $\lrd B_i\rrd$. Thus reindexing $X_i$'s we may assume that the flipping locus does not intersect the support of $\lrd\Delta_i\rrd$ for all $i\>0$. In particular, any $(K_X+B)$-flip is also a $(K_X+\{B\})$-flip, where $\{B\}$ is the fractional part of $B$. Therefore replacing $B$ by $\{B\}$ we may assume that $(X, B)$ is klt and the given sequence is a sequence of klt flips.  Now we follow the steps $0, 1, 2$ and $3$ as in Kawamata's proof of \cite[Theorem 1]{Kaw92}.\\
  \noindent 
   \textbf{Step 0:} This step works exactly as in the proof of \cite[Theorem 1]{Kaw92} and we obtain that \cite[Lemma 7]{Kaw92} holds for $e(X, B)=0$.\\
 \textbf{Step 1:} This step also works exactly as in the proof of \cite[Theorem 1]{Kaw92} and we obtain that for every $e\in\mbZ^+$, Theorem \ref{thm:terminal-model} for all pairs $(X, B)$ such that $e(X, B)\<e-2$ implies \cite[Lemma 6]{Kaw92} for all pairs $(X, B)$ with $e(X, B)=e$.\\
 \textbf{Step 2:} This step also works as in the proof of \cite[Theorem 1]{Kaw92} and we obtain that for every $e\in\mbZ^+$, \cite[Lemma 6]{Kaw92} for all $\mbQ$-factorial klt pair $(X, B)$ with $e(X, B)=e$ and \cite[Lemma 7]{Kaw92} with $e(X, B)=e-1$ together imply \cite[Lemma 7]{Kaw92} in the case of $e(X, B)=e$.\\

\noindent        
\textbf{Step 3:} For a fixed $e\in\mbZ^+$, \cite[Lemma 7]{Kaw92} for $e(X, B)=e$ and Theorem \ref{thm:termination} for $e(X, B)<e$ together imply Theorem \ref{thm:termination} for all $\mbQ$-factorial klt pairs $(X, B)$ with $e(X, B)=e$.

This step also works as in the proof of \cite{Kaw92} for the most part, but there is a non-trivial component to it. So here we discuss the proof in full detail.\\

Let $(X, B)$ be a $\mbQ$-factorial klt pair with $e(X, B)=e$. If $e(X_n, B_n)<e$ for some $n\>0$, then the sequence of the flips terminates by induction hypothesis, so we may assume that $e(X_n, B_n)=e$ for all $n\>0$. Now we will define an invariant $\discrep(X, B)^+$ for compact K\"ahler pairs $(X, B)$ as follows:
\[
	\discrep(X, B)^+:=\inf \{a(E, X, B)>0: E \mbox{ is an exceptional divisor over } X \}.
\]
Since in our case $(X, B)=(X_0, B_0)$ is klt, by passing to a log resolution of $(X, B)$ and then using \cite[Corollary 2.31(3)]{KM98} and its proof, it follows that $\discrep(X, B)^+=c>0$ for some positive rational number $c>0$. Then from \cite[Lemma 3.38]{KM98} it follows that $\discrep(X_n, B_n)^+\>c$ for all $n\>0$. Note that under our induction hypothesis, \cite[Lemma 7]{Kaw92} holds, and hence there exists a positive integer $q>0$ such that $qD$ is Cartier and $\mcO_{X_n}(qK_{X_n})$ is a line bundle, for every Weil divisor $D$ on $X_n$ and for all $n\>0$. Now replacing $q$ by a suitable multiple so that $qB_0$ is Cartier, we may assume that $q(K_{X_n}+B_n)$ is Cartier for all $n\>0$. In particular, it follows that for every divisor $E$ over $X_n$ and for all $n\>0$, the discrepancy $a(E, X_n, B_n)$ is contained the discrete set $\frac{1}{q}\mbZ$.\\

  Let $F_1, F_2,\ldots, F_e$ be the divisors over $X$ such that $a(F_j, X, B)\<0$. Now we claim that $a(F_j, X_n, B_n)\<0$ for all $1\<j\<e$ and for every $n\>0$. Indeed, if not, then since $e(X_n, B_n)=e$ for all $n\>0$, there is an integer $n_0>0$ and a divisor $E$ over $X_{n_0}$ distinct from $F_j$ for all $1\<j\<e$ such that $a(E, X_{n_0}, B_{n_0})\<0$. Then from \cite[Lemma 3.38]{KM98} it follows that $a(E, X_{n_0-1}, B_{n_0-1})\<a(E, X_{n_0}, B_{n_0})\<0$. This implies that $e(X, B)\>e+1$, a contradiction.\\
 Now if there is a $j\in\{1,\ldots, e\}$ and $n_0\>0$ such that the center $C_{jn_0}$ of $F_j$ on $X_{n_0}$ is contained $\Ex(\vphi_{n_0})$, then $\lambda_{n}:=\sum_{j=1}^e a(F_j, X_{n}, B_{n})$ strictly increases by \cite[Lemma 3.38]{KM98} for all $n\>n_0$, i.e. $\lambda_n<\lambda_{n+1}$ for all $n\>n_0$. But since $\lambda_n\<0$ for all $n\>0$, we have $\lambda_n\in [\lambda_{n_0}, 0]$ for all $n\>n_0$. However, recall that $\lambda_n\in\frac{1}{q}\mbZ$ for all $n\>0$, and thus $\lambda_n\in[\lambda_{n_0}, 0]\cap\frac{1}{q}\mbZ$ for all $n\>n_0$; this is a contradiction, since  $[\lambda_{n_0}, 0]\cap\frac{1}{q}\mbZ$ is a finite set.
  So after reindexing if necessary, we may assume that $C_{jn}$ (the center of $F_j$ on $X_n$) is not contained in $\Ex(\vphi_n)$ for any $n\>0$ and any $1\<j\<e$. In particular, $|a(F_j, X_n, B_n)|=c_j$ is constant for all $n\>0$ and each $1\<j\<e$ (as it is an increasing sequence in $\frac{1}{q}\mbZ$ bound above by $0$). Suppose that $c_1\>c_2\>\cdots\>c_e$. Recall that $\Center_{X_n}(F_1)=C_{1n}$ and $a(F_1, X_n, B_n)=c_1$ for all $n\>0$.
  Next we make the following claim. 
  
  \begin{claim}\label{clm:termination}
	If $C_{1n}\cap\Ex(\vphi_n)=\emptyset$ for all $n\gg 0$, then Theorem \ref{thm:termination} holds.
  \end{claim}
  
 \begin{proof}[Proof of Claim \ref{clm:termination}]
Our strategy is to extract $F_1$ over each $X_n$ and construct a sequence of flips $(W_n, B_{W_n})\bir (W_{n+1}, B_{W_{n+1}})$ over $(X_n, B_n)\bir (X_{n+1}, B_{n+1})$ for all $n\>0$. Then this new sequence terminates by induction as the invariant $e(W_n, B_{W_n})=e-1$ for $n\>0$. From this we argue that the original sequence must terminate.\\ 
	To this end, let $f_n:W_n\to X_n$ be the extraction of the divisor $F_1$ with $\Center_X(F_1)=C_{1n}$ by running an appropriate MMP over $X$; note that this MMP terminates by our induction hypothesis in this step that Theorem \ref{thm:termination} holds for $e(X, B)<e$. Then $\rho(W_n/X_n)=1$, since $X_n$ is $\mbQ$-factorial, and hence $\rho(W_n/Y_n)=2$. Let $K_{W_n}+B_{W_n}=f^*_n(K_{X_n}+B_n)$. Since $f_n(F_1)\cap\Ex{\vphi_n}=\emptyset$, the strict transform of the curves in $\Ex(\vphi_n)$ generate a $(K_{W_n}+B_{W_n})$-negative extremal ray $R_n$ of the relative Mori cone $\NE(W_n/Y_n)$. Let $g_n:W_n\to Z_n$ be the flipping contraction over $Y_n$ corresponding to extremal ray $R_n$ (see \cite[Theorem 4.12(2)]{Nak87}) and $g^+_n:W_{n+1}\to Z_n$ is the associated flip over $Y_n$ (see Theorem \ref{thm:existence-of-flips}), and $h_n:Z_n\to Y_n$ is the induced morphism. Let $B_{W_{n+1}}$ be the pushforward of $B_{W_n}$ on $W_{n+1}$. Then $K_{W_{n+1}}+B_{W_{n+1}}$ is $g^+_n$-ample and $\rho(W_{n+1}/Z_n)=1$. Now observe that, since $f_n(F_1)\cap \Ex(\vphi_n)=\emptyset$, from our construction above it follows that $F_1\cong g_n(F_1)\cong {g^+}^{-1}_n(g_n(F_1))=:F^{n+1}_1$. Moreover, since $\Ex(h_n\circ g_n)=\Ex(\vphi_n\circ f_n)$, it follows that $K_{W_{n+1}}+B_{W_{n+1}}$ is $(h_n\circ g^+_n)$-nef. We claim that if $C\subset W_{n+1}$ is a curve such that $(h_n\circ g^+_n)(C)=\pt$, then $(K_{W_{n+1}}+B_{W_{n+1}})\cdot C=0$ if and only if $C\subset F^{n+1}_1$. Indeed, the if part is clear, so assume that for a $h_n\circ g^+_n$-vertical curve $C\subset W_{n+1}$ we have $(K_{W_{n+1}}+B_{W_{n+1}})\cdot C=0$. To the contrary assume that $C\not\subset F^{n+1}_1$. Then $C\not\subset F^{n+1}_1\cup\Ex(g^+_n)$, as $(K_{W_{n+1}}+B_{W_{n+1}})\cdot C=0$ and $K_{W_{n+1}}+B_{W_{n+1}}$ is $g^+_n$-ample. In particular, $g^+_n(C)\neq \pt$ and $h_n(C)=\pt$. Let $C_{W_n}$ be the strict transform of $g^+_n(C)$ under $g_n$; then $C_{W_n}\not\subset F^n_1\cup\Ex(g_n)$. Consequently, $f_n(C_{W_n})\neq \pt$ and $\vphi_n(f_n(C_{W_n}))=\pt$. This is a contradiction to the fact that $g_n(C_{W_n})=g^+_n(C)\neq \pt$, as $\vphi_n(f_n(C_{W_n}))=\pt$ implies that the numerical class of $C_{W_n}$ in $\NE(W_n/Y_n)$ is contained the ray $R_n$ defined above, and hence $C_{W_n}$ is contracted by $g_n$ according to our construction.
 Then from the relative Base-point free theorem for projective morphism as in \cite[Theorem 4.8]{Nak87}, it follows that $K_{W_{n+1}}+B_{W_{n+1}}$ is semi-ample over $Y_n$. Then there exist projective bimeromorphic morphisms $f_{n+1}:W_{n+1}\to V_{n+1}$ and $\vphi_{V_{n+1}}:V_{n+1}\to Y_n$ such that $K_{V_{n+1}}+B_{V_{n+1}}:=f_{n+1*}(K_{W_{n+1}}+B_{W_{n+1}})$ is $\vphi_{V_{n+1}}$-ample and $K_{W_{n+1}}+B_{W_{n+1}}=f^*_{n+1}(K_{V_{n+1}}+B_{V_{n+1}})$. Note that $f_{n+1}$ contracts the divisor $F^{n+1}_j$.  
	
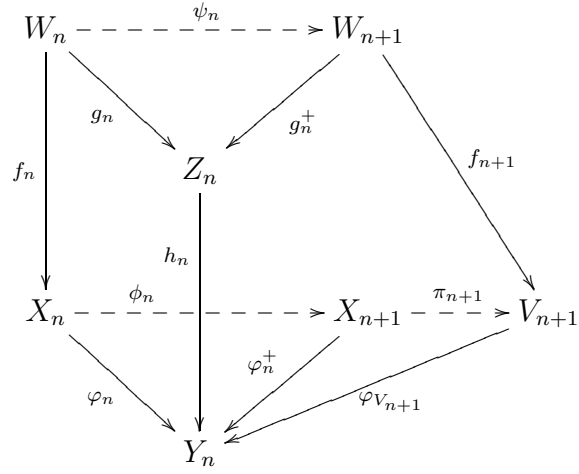
\begin{equation*}
	\xymatrixcolsep{3pc}\xymatrixrowsep{3pc}\xymatrix{
	W_n\ar@{-->}[rr]^{\psi_n}\ar[dr]_{g_n}\ar[dd]_{f_n} && W_{n+1}\ar[dl]^{g^+_n}\ar[ddr]^{f_{n+1}} &\\
	& Z_n\ar[dd]_(.3){h_n} &&\\
	X_n\ar@{-->}[rr]^(.3){\phi_n}\ar[dr]_{\vphi_n} && X_{n+1}\ar[dl]_{\vphi^+_n}\ar@{-->}[r]^{\pi_{n+1}} & V_{n+1}\ar[dll]^{\vphi_{V_{n+1}}}\\
	& Y_n &&
	}
\end{equation*}	
	
From the diagram above it follows that $(\pi_{n+1}\circ\phi_n)_*(K_{X_n}+B_n)=\pi_{n+1 *}(K_{X_{n+1}}+B_{n+1})=K_{V_{n+1}}+B_{V_{n+1}}$ and $\codim_{V_{n+1}} \Ex(\vphi_{V_{n+1}})\>2$. Thus $\vphi_{V_{n+1}}:V_{n+1}\to Y_n$ is a flip of $\vphi_n:X_n\to Y_n$, and since flip of a flipping contraction is unique, 
it follows that $V_{n+1}\cong X_{n+1}$ and $\vphi_{V_{n+1}}=\vphi^+_n$.\\
The above construction gives us a sequence of klt-flips $\psi_n: (W_n, B_{W_n})\bir (W_{n+1}, B_{W_{n+1}})$ such that $e(W_n, B_{W_n})=e-1$ for all $n\>0$. Thus by induction this is a finite sequence, i.e. there exists an integer $n_0\>0$ such that $\psi_n:W_n\bir W_{n+1}$ is an isomorphism for all $n\>n_0$. 
We claim that $\phi_n:X_n\bir X_{n+1}$ is an isomorphism for all $n\>n_0$. Indeed, let $\Gamma_n$ be the normalization of the graph of $\phi_n$ for $n\>n_0$, and $p_n:\Gamma_n\to X_n$ and $q_n:\Gamma_n\to X_{n+1}$ are the induced bimeromorphic morphisms. Then by the universal property of graph it follows that there is a proper bimeromorphic morphism $\theta_n:W_n\to \Gamma_n$ for $n\>n_0$ such that $f_n=p_n\circ\theta_n$ and $f_{n+1}=q_n\circ\theta_n$ (note that here we are identifying $W_n$ and $W_{n+1}$ via the isomorphism $\psi_n$). Now since $\rho(W_n/X_n)=1$, either $\theta_n$ is finite or $p_n$ is finite. If $\theta_n$ is finite and hence isomorphism, then identifying $W_n$ with $\Gamma_n$ and $f_n$ with $p_n$ for $n\>n_0$ we see that $p_n(F_1)=\Ex(\vphi_n)$, this is a contradiction, since $p_n(F_1)\cap\Ex(\vphi_n)=f_n(F_1)\cap \Ex(\vphi_n)=\emptyset$.  On the other hand, if $p_n$ is finite, then it is in fact an isomorphism, and thus $\phi_n$ is a morphism. Then since $X_{n+1}$ is $\mbQ$-factorial and $\phi_n$ is an isomorphism is codimension $1$, it follows that $\phi_n$ is an isomorphism for all $n\>n_0$. This completes the proof of the claim.\\
		
 \end{proof}

 Thus we may assume that $C_{1n}\cap\Ex(\vphi_n)\neq\emptyset$ for all $n\>0$. Then the rest of the proof works exactly as in the proof of Step $3$ of \cite[Theorem 1]{Kaw92}.
 \end{proof}~\\

Using the results above we can run a relative MMP for projective morphism between K\"ahler varieties of dimension $\<3$. 

\begin{proposition}[MMP for Projective Morphisms]\label{pro:relative-projective-mmp}
	Let $(X, \Delta)$ be a $\mbQ$-factorial dlt pair and $f:X\to Y$ a projective surjective morphism between two normal compact K\"ahler varieties. If $\dim X\<3$, then we can run a relative $(K_X+\Delta)$-MMP over $Y$ which terminates with either a minimal model or a Mori fiber space, according to whether $K_X+\Delta$ is pseudo-effective over $Y$ or not.
Moreover, if $X=X_0\bir X_1\bir\cdots \bir X_n\cdots$ are the steps of a $(K_X+\Delta)$-MMP over $Y$, then every $X_i$ is a K\"ahler variety for $i\>0$; additionally, if $\psi:X_n\to Y'$ is a Mori fiber space over $Y$, then $Y'$ is also K\"ahler.
\end{proposition}

\begin{proof}
	Since $f$ is projective, the (relative) cone and contraction theorems are known due to \cite[Theorem 4.12]{Nak87}. Since $\dim X\<3$, the existence of flips (over $Y$) follows from \cite{Sho92}. The termination of flips (over $Y$) follow from Theorem \ref{thm:termination}. The proof of the fact that a $(K_X+\Delta)$-MMP over $Y$ terminates either with a minimal model or a Mori fiber space according to whether $K_X+\Delta$ is pseudo-effective over $Y$ or not, works exactly as in the algebraic case, since $f:X\to Y$ is a projective morphism.
	 If $g_i:X_i\to Z_i$ is a contraction of a $(K_{X_i}+\Delta_i)$-negative extremal ray of $\NE(X_i/Y)$, then from the relative base-point free theorem \cite[Theorem 4.10]{Nak87} it follows that the induced morphism $h_i:Z_i\to Y$ is projective. Then from \cite[Proposition 1.3.1, page 24]{Var89} it follows that $Z_i$ is K\"ahler. If $g_i$ is a flipping contraction and $g_i^+:X_{i+1}\to Z_i$ is the flip, then again $X_{i+1}$ is K\"ahler by the same argument. In the Mori fiber space case again by a similar argument it follows that $Y'$ is K\"ahler.
\end{proof}

\begin{corollary}\label{c-rccres} Let $X$ be a $\mathbb{Q}$-factorial compact K\"ahler $3$-fold with klt singularities and $\mu:X'\to X$ a proper bimeromorphic morphism. Then every fiber of $\mu$ is rationally chain connected. 
\end{corollary}
\begin{proof} 
Let $\nu:\tilde X\to X$ be a resolution of singularities of $X$ dominating $X'$. Using Chow's lemma for bimeromorphic morphism (see \cite[Corollary 2]{Hir75}) and then passing to a higher resolution we may assume that $\mu\circ \nu:\tilde X\to X$ is a projective morphism. 
 It suffices to show that  every fiber of $\nu$ is rationally chain connected. Thus replacing $X'$ by $\tilde{X}$ and $\mu$ by $\mu\circ\nu$ we may assume that $X'$ is smooth and $\mu:X'\to X$ is projective. We run a $(K_{X'}+{\rm Ex}(\mu ))$-MMP over $X$ via Proposition \ref{pro:relative-projective-mmp}:
	
	\begin{equation}\label{eqn:rcc-mmp-digram}
	\xymatrixcolsep{3pc}\xymatrix{ X'=X'_0\ar@{-->}[r] & X'_1\ar@{-->}[r] & \cdots\ar@{-->}[r] & X'_n. }
	\end{equation}
Since $X$ is $\mbQ$-factorial and has klt singularities, we have $X'_n=X$. Let $\phi_i:X'\bir X'_i$ and $E_i$ be the push-forward of ${\rm Ex}(\mu )$ by $\phi_i$.

If $f_i:X'_i\to X'_{i+1}$ is a $(K_{X'_i}+E_i)$-divisorial contraction of a divisor $E$ in the above MMP \eqref{eqn:rcc-mmp-digram}, then by construction $E$ is contained in the support of $E_i$ . In this case $f_i$ is also a $(K_{X'_i}+E)$-divisorial contraction. Now by adjunction $K_E+B_E=(K_{X'_i}+E)|_E$ and  $(E,B_E)$ is a klt surface. Note that $-(K_E+B_E)$ is relatively ample with respect to the morphism $f_i|_E: E\to f_i(E)$. It is well known that $E$ is relatively rationally chain connected, for example, see \cite[Corollary 1.3]{HM07}.\\ 

If instead we have $f_i:X'_i\bir X'_{i+1}$ a $(K_{X'_i}+E_i)$-flip, then the flipping locus is contained in a component $E$ of $E_i$. Let $F=E+(1-\epsilon)(E_i-E)$ for some $0<\epsilon \ll 1$. Then $f_i$ is also a $(K_{X'_i}+F)$-flip. Now by adjunction $(K_{X'_i}+F)|_E=K_E+B_E$ such that $(E,B_E)$ has klt singularities, and in particular, the surface $E$ is $\mbQ$-factorial. Note that if $g:X'_i\to Z$ is the $(K_{X'_i}+F)$-flipping contraction associated to the flip $f_i$, then $-(K_E+B_E)$ is relatively ample with respect to the morphism $g|_E: E\to g(E)$.  
Let $C$ be any flipping curve, then $(K_E+B_E)\cdot C<0$ and $C^2<0$ (since it is an exceptional curve contained in the surface $E$). It follows that $(K_E+C).C<0$ and hence that $C$ is rational. Finally, we argue that every flipped curve $C^+$ is rational. Note that $C^+$ is exceptional over $X$ and hence must be contained in a component say $E^+$ of $E_{i+1}:=f_{i, *}(E_i)$ (since $E_{i+1}$ is the exceptional locus of $X'_{i+1}\to X$). Let $g^+: X'_{i+1}\to Z$ be the flipped contraction and $E^+_Z=g^+_*(E^+)$. 
Let $E^-=(f_i^{-1})_*E^+$ be its strict transform on $X_i$ and $K_{E^-}+B_{E^-}=(K_{X_i'}+E_i)|_{E^-}$. Since $(E^-,B_{E^-})$ is dlt and $-(K_{E^-}+B_{E^-})$ is $g|_{E^-}$ ample, it follows that $(E_Z,(g|_{E^-})_*B_{E^-})$ is dlt, where $E_Z:=g_*E^-$, and in particular $E_Z$ has rational $\mbQ$-factorial singularities. Thus $C^+$ is a rational curve.

From what we have seen above, it follows easily that the fibers of $X'_i\to X$ are rationally chain connected if and only if the fibers of $X'_{i+1}\to X$ are rationally chain connected. Since $X'_n=X$, this concludes the proof.
\end{proof}

In the next result (Lemma \ref{l-proj}) we will show that if $X$ is a uniruled compact K\"ahler 3-folds with $\mbQ$-factorial klt singularities such that the base of the MRC(C) fibration of $X$ has dimension $\<1$, then $X$ is a projective variety. In fact we prove something stronger: if $\phi:X\bir X'$ is a bimeromorphic contraction, for example, a divisorial contraction or a flip, then $X'$ is also projective. This is very useful in running MMP.

\begin{lemma}\label{l-proj} 
	Let $X$ be a uniruled normal $\mbQ$-factorial compact K\"ahler $3$-fold with klt singularities and $\phi:X\bir X'$ a bimeromorphic map to a $\mbQ$-factorial compact K\"ahler $3$-fold $X'$ with klt singularities such that $\phi$ does not extract any divisor, i.e. $\phi^{-1}:X'\bir X$ does not contract any divisor. If the base of the MRC(C) fibration of $X$ has dimension less than $2$, then $X'$ is a projective variety. In particular, if $X$ is non-algebraic, then the base of the MRC(C) fibration of $X$ has dimension $2$.
\end{lemma}
	
\begin{proof} Let $\pi:X\bir Z$ be the MRC(C) fibration (see \cite[Remark 6.10]{CH20}). By assumption $\dim Z=0$ or $1$. 
First note that, since $X$ has rational singularities, by \cite[Corollary 1.7]{Nam02}, $X$ is projective if and only if any resolution $\tilde X$ of the  singularities of $X$ is projective. 
Note also that by Corollary \ref{c-rccres}, the fibers of $\nu:\tilde X\to X$ are rationally chain connected, thus it follows that $\tilde X \bir Z$ is also a MRC(C) fibration. Note that as $\tilde X$ is smooth, then its MRC and MRCC fibrations coincide.
Thus replacing $X$ by a resolution of singularities we may assume that $X$ is a compact K\"ahler manifold and $\phi:X\to X'$ is a morphism.   Possibly replacing $X$ further by a higher resolution we may assume that $\pi$ and $\phi$ are both morphisms.  Since the general fibers of $\pi$ are rationally connected, by \cite[Corollary 4.18]{Deb01} $H^0(F, \Omega ^i _F)=0$ for all $i\>1$ where $F$ is a general fiber. We claim that $H^0(X, \Omega^2_X)=0$. If $\dim Z=0$, then this is clear, so assume that $\dim Z=1$. Then observe that the following exact sequence 
	\begin{equation*}
		\xymatrixcolsep{3pc}\xymatrix{ \pi^*\Omega_Z\ar[r] & \Omega_X\ar[r] & \Omega_{X/Z}\ar[r] & 0 }
	\end{equation*}
is left exact over Zariski open dense subset of $Z$. This follows from the generic smoothness of $\pi$ and the fact that the MRC fibration is an almost holomorphic map
Restricting this sequence to a general fiber $F$ of $\pi$ we get the following short exact sequence
	\begin{equation}\label{eqn:mcc-forms-1}
		\xymatrixcolsep{3pc}\xymatrix{ 0\ar[r] & \mcO_F\ar[r] & \Omega_X|_F\ar[r] & \Omega_F\ar[r] & 0. }
	\end{equation}
	Thus we have a short exact sequence
	\begin{equation}\label{eqn:mcc-forms-2}
		\xymatrixcolsep{3pc}\xymatrix{ 0\ar[r] & \Omega_F\ar[r] & \Omega ^2_X|_F\ar[r] & \Omega ^2_F\ar[r] & 0. }
	\end{equation}
 It follows from this exact sequence that $H^0(X, \Omega^2_X)=0$. Then $H^2(X, \mcO_X)=\overline{H^0(X, \Omega^2_X)}=0$, and by the Kodaira's projectivity criterion we have that $X$ is projective. 
	
		Finally, since $X'$ has $\mbQ$-factorial klt singularities, running a $(K_X+\Ex(\phi))$-MMP over $X'$ will recover $X'$, and hence $X'$ is also projective.\\

  Now if $X$ is non-algebraic, then set $X':=X$ and $\phi$ as the identity morphism. Thus from the previous part it follows that the base of the MRC fibration of $X$ must be $2$ in this case.  
 \end{proof}

\begin{lemma}\label{lem:analytic-to-projectivity}
Any compact analytic variety (possibly singular) of dimension $1$ is a projective algebraic curve.
\end{lemma}
\begin{proof}
        
  Let $C$ be a compact analytic curve and $p\in C_\textsubscript{sm}$ a smooth point of $C$. Then $\mcO_C(p)$ is a line bundle on $C$ with a section  $0\neq s\in H^0(C, \mcO_C(p))$ such that $\div(s)_0=p$, and hence by Grauert's criteria \cite[Theorem I.19.3]{BHPV04}, $C$ is projective. 
\end{proof}
\begin{remark}
        The above lemma holds more generally for any \emph{reduced} compact complex space of dimension $1$ with finitely many irreducible components. The proof follows similarly by working on each component separately. 
\end{remark}

\section{Length of Extremal rays and MMP with scaling}
\subsection{Length of extremal rays} In this section following Kawamata's article \cite{Kaw91} we will show that for a compact K\"ahler $3$-fold klt pair $(X, \Delta)$, the length of the rational curves generating the $(K_X+\Delta)$-negative extremal rays is bounded above by a universal constant.

\begin{lemma}\cite[Lemma]{Kaw91}\label{lem:hodge-type-inequality}
        Let $f:X\to Y$ be projective bimeromorphic morphism of normal analytic varieties, $H$ an $f$-ample Cartier divisor on $X$ and $E$ an irreducible component of $\Ex(f)$ of maximal dimension. Let $e=\dim E$ and $\nu:\bar{E}\to E$ be the normalization morphism. Further suppose that there is an effective $\mbQ$-divisor $\Delta\>0$ such that $(X, \Delta)$ is klt, and $f(E)$ is a point in $Y$. Then
        \begin{equation*}
                (H^{e-1}\cdot(K_X+\Delta)\cdot E)>((\nu^*H)^{e-1}\cdot K_{\bar{E}}).
        \end{equation*}
\end{lemma}

\begin{proof}
 This proof is essentially same as Kawamata's original proof in \cite{Kaw91} with minor modifications, which we will discuss now. First observe that the question is local on the base $Y$ and $f(E)$ is a point in $Y$. Moreover, we have a $f$-ample Cartier divisor $H$ on $X$. So we can choose a small Stein open subset $U\subset Y$  containing $f(E)$ and positive integers $m, n>0$ such that $mH$ gives a closed embedding $f^{-1}U\injective \mbP^n\times U$. Then replacing $Y$ by $U$ and $X$ by $f^{-1}U$ we may assume that $\iota: X\injective \mbP^n\times Y$ is a closed embedding such that $\pr_2\circ\iota=f$. Then $\pr_1^*\mcO(1)$ is a globally generated line bundle on $\mbP^n\times Y$. 
We can use the global sections of ($\pr_1\circ\;\iota)^*\mcO(1)$ to cut down the dimension of $X$ as in the proof of \cite[Lemma]{Kaw91}. We will need a Bertini-type theorem here to guaranteed that the general hyperplane sections of $X$ are irreducible and normal. This is achieved by \cite[Corollary II.5]{Man82}. We will also need a relative Kawamata-Viehweg vanishing theorem in this proof. Such a statement is proved in \cite[Theorem 3.7]{Nak87}. With these tools Kawamata's arguments in \cite[Lemma]{Kaw91} work in our settings. Interested reader may also consult \cite[Lemma 7.47]{Deb01} for a more detailed account of this proof in the algebraic case. 
\end{proof}~\\

\begin{theorem}\label{thm:length-of-extremal-rays}
        Let $(X, \Delta\>0)$ be a klt pair and $f:X\to Y$ is a a projective morphism to a Stein variety $Y$ such that $Y$ embeds into an open subset of $\mbC^N$ as an analytic subset. Let $E$ be an irreducible component of $\Ex(f)$ of maximal dimension, and $n=\dim E-\dim f(E)$. Assume further that $-(K_X+\Delta)$ is $f$-ample. Then for a general point $p\in f(E)$, $E_p=E\cap f^{-1}(p)$ is covered by a family of rational curves $\{\Gamma_\lambda\}_{\lambda\in\Lambda}$ such that $-(K_X+\Delta)\cdot \Gamma_\lambda\<2n$.

\end{theorem}

\begin{proof}
        First replacing $f$ by its Stein factorization we may assume that $f$ has connected fibers and $Y$ is normal. There are two cases now depending on whether $f$ is bimeromorphic or not. The bimeromorphic case is more involved and we show the details here, the other case is left as an exercise to the reader.

    Assume that $f$ is bimeromorphic. Let $d=\dim f(E)$ and $Y_0$ is the intersection of $Y$ by $d$ very general hyperplanes of $\mbC^N$.  Then $Y_0$ is normal by \cite[Corollary II.5]{Man82}. Moreover, $Y_0\cap f(E)$ is a set of finite points. Shrinking $Y$ further if necessary we may assume that $Y_0\cap f(E)$ is a single point. Let $X_0=f^{-1}Y_0, \Delta_0=\Delta\cap X_0, E_0=E\cap X_0$ and $f_0=f|_{X_0}$. Then again by \cite[Theorem II.5]{Man82} and \cite[Lemma 5.17]{KM98} it follows that $X_0$ is normal and $(X_0, \Delta_0)$ is klt. Moreover, we also have $f_0$ is projective and $-(K_{X_0}+\Delta_0)$ is $f_0$-ample. Let $\nu:\bar{E}_0\to E_0$ be the normalization morphism, and $H$ a $f_0$-very ample Cartier divisor on $X_0$. Note that $E_0$ is a projective algebraic variety, since $f_0$ is a projective morphism and $f_0(E_0)$ is a point in $Y_0$. Let $C$ be the intersection of $n-1$ general members of the linear system $|\nu^*H|$ (note that $\dim E_0=n$). Then by Bertini's theorem, $C$ is a smooth curve contained in the smooth locus of $\bar{E}_0$. Then by Lemma \ref{lem:hodge-type-inequality} applied to $f_0$, we get 
        \begin{equation}\label{eqn:hodge-type-inequality}       
                (H^{n-1}\cdot(K_{X_0}+\Delta_0)\cdot E_0)>((\nu^*H)^{n-1}\cdot K_{\bar{E}_0}).
        \end{equation}  

From the left hand side of \eqref{eqn:hodge-type-inequality} we get $H^{n-1}\cdot(K_{X_0}+\Delta_0)\cdot E_0=\nu^*(K_{X_0}+\Delta_0)\cdot C<0$, since $-(K_{X_0}+\Delta_0)$ is ample on $E_0$. Thus from the right hand side of \eqref{eqn:hodge-type-inequality} it follows that $K_{\bar{E}_0}\cdot C<0$. Then by \cite[Theorem 5.8]{Kol96} through every point of $x\in C$ there is a rational curve $\bar{\Gamma}_x$ in $\bar{E}_0$ such that
\[
-\nu^*(K_X+\Delta)\cdot \bar{\Gamma}_x\<2\dim \bar{E}_0\frac{-\nu^*(K_X+\Delta)\cdot C}{-K_{\tilde{E}_0}\cdot C}.
\]
 Taking the images of $\bar{\Gamma}_x$'s in $E_0$ we see that  $E_0$ is covered by a family of rational curves $\{\Gamma_\lambda\}_{\lambda\in\Lambda}$ such that $-(K_X+\Delta)\cdot\Gamma_\lambda<2n$, where $n=\dim E-\dim f(E)$.\\

When $f$ is not bimeromorphic, the general fiber $F$ of $f$ is a projective variety of positive dimension, and thus the usual argument of Mori's Bend and Break work on $F$ (see \cite[Theorem 5.8]{Kol96}). The details of this proof is left as an exercise to the reader.

\end{proof}~\\

In the following we discuss how to run a MMP with scaling of an effective divisor.\\
\begin{proposition}\label{pro:mmp-with-scaling}
	Let $(X, \Delta)$ be a $\mbQ$-factorial compact K\"ahler $3$-fold klt pair. Let $D$ be a nef $\mbQ$-Cartier divisor and $K_X+\Delta$ is not nef. Further assume that the  cone theorem and contraction theorem hold for the $(K+\Delta)$-MMP. Set
\[
\lambda:=\sup\{t\>0: D+t(K_X+\Delta)\text{ is nef }\}.
\]
Then $\lambda$ is a rational number and there is a $(K_X+\Delta)$-negative extremal ray $R$ such that
\[
(D+\lambda(K_X+\Delta))\cdot R=0.
\]
\end{proposition}

\begin{proof}
        Let $k$ be a natural number such that $k(K_X+\Delta)$ and $kD$ are both Cartier. Then by Theorem \ref{thm:length-of-extremal-rays}, for each $(K_X+\Delta)$-negative extremal ray $R_i$ for $i\in I$, there is a rational curve $C_i$ such that $R_i=\mbR_{\>0}[C_i]$ and $-(K_X+\Delta)\cdot C_i\<2n$, where $n$ is the dimension of $X$. In particular, we have       
\begin{equation}\label{eqn:bounded-length}
\mu_i=\frac{D\cdot C_i}{-(K_X+\Delta)\cdot C_i}=\frac{kD\cdot C_i}{-k(K_X+\Delta)\cdot C_i}\>\frac{(kD)\cdot C_i}{2nk}\mbox{ for all } i\in I.
\end{equation}
Note that $-k(K_X+\Delta)\cdot C_i$ is an integer contained in the set $\{1, 2,\ldots, 2nk\}$ for all $i$, and the $(kD)\cdot C_i$'s are also integers for all $i$. Therefore the $\mu_i$'s are contained in the discrete set $\mathcal{S}=\{\frac{m}{l}:m, l\in\mbZ, m\>0, 1\<l\<2nk\}$. In particular, $\mu:=\inf\{\mu_i: i\in I\}$ is contained in $\{\mu_i: i\in I\}$, i.e., $\mu$ is rational, and there is an extremal ray $R=R_{i_0}$ for some $i_0\in I$ such that
\[
(D+\mu(K_X+\Delta))\cdot R=0.
\]
From the cone theorem it follows that $D+\mu(K_X+\Delta)$ is nef and $D+s(K_X+\Delta)$ is not nef for any $s>\mu$, since $\mu$ is the minimum of all $\mu_i$. Therefore $\mu$ is the nef threshold, and we set $\lambda=\mu$.

\end{proof}

\begin{corollary}\label{cor:mmp-with-scaling}
Let $(X, \Delta)$ be a $\mbQ$-factorial compact K\"ahler $3$-fold klt pair. Suppose that $K_X+\Delta$ is not nef and there is an effective $\mbQ$-divisor $H$ on $X$ such that $K_X+\Delta+H$ is nef. Then there is a $(K_X+\Delta)$-negative extremal ray $R$ and a rational number $0<\lambda\<1$ such that $K_X+\Delta+\lambda H$ is nef and $(K_X+\Delta+\lambda H)\cdot R=0$.
\end{corollary}

\begin{proof}
        It follows directly from Proposition \ref{pro:mmp-with-scaling} by setting $D=K_X+\Delta+H$.
\end{proof}~\\

\subsection{Running MMP with scaling}\label{subsec:mmp-with-scaling}
Let $(X, \Delta)$ be a $\mbQ$-factorial compact K\"ahler $3$-fold klt pair. Suppose that $K_X+\Delta$ is not nef and there is an effective $\mbQ$-divisor $H$ on $X$ such that $K_X+\Delta+H$ is nef. We will run a $(K_X+\Delta)$-MMP with the scaling of $H$. Set $(X_0, \Delta_0)=(X, \Delta)$, $H_0=H$, and let $\lambda_0=\inf\{t\>0: K_X+\Delta+tH\mbox{ is nef }\}$. If $\lambda_0=0$, then $K_X+\Delta$ is nef and we stop. Otherwise, by Corollary \ref{cor:mmp-with-scaling} there exists a $(K_{X_0}+\Delta_0)$-negative extremal ray, say $R_0$, such that $(K_{X_0}+\Delta_0+\lambda_0H)\cdot R_0=0$. Let $\vphi_0:X_0\to Y$ be the associated contraction. Then $K_{X_0}+\Delta_0+\lambda_0H=\vphi^*L$, for some $\mbQ$-Cartier divisor $L$ on $Y$. There are two cases now
\begin{enumerate}
        \item $\vphi:X_0\to Y$ is a divisorial contraction, or
        \item $\vphi$ is a flipping contraction.
\end{enumerate} 
If $\vphi$ is a divisorial contraction, then $Y$ is a normal $\mbQ$-factorial compact K\"ahler $3$-fold. In this case wet set $X_1=Y, \Delta_1=\vphi_*\Delta_0$ and $H_1=\vphi_*H_0$. Then we have $(X_1, \Delta_1)$ a klt pair such that $K_{X_1}+\Delta_1+\lambda_0H_1$ is nef. We then set $\lambda_1=\inf\{t\>0: K_{X_1}+\Delta_1+tH_1\mbox{ is nef }\}$ and continue the program.\\
If $\vphi$ is a flipping contraction, then $Y$ is not $\mbQ$-factorial, so we take a flip $\vphi^+:(X^+_0, \Delta^+_0)\to Y$ of $\vphi$, where $\Delta^+_0=\phi_*\Delta$ and $ H^+=\phi_*H$, and $\phi:X\to X^+$ is the induced birational map. Then $X^+$ is a normal $\mbQ$-factorial compact K\"ahler $3$-fold. Set $X_1=X^+, \Delta_1=\Delta^+$ and $H_1=H^+$. Then we have $(X_1, \Delta_1)$ a klt pair such that $K_{X_1}+\Delta_1+\lambda_0H_1$ is nef. We then set $\lambda_1=\inf\{t\>0: K_{X_1}+\Delta_1+tH_1\mbox{ is nef }\}$ and continue the program.\\
Note that in way we get a monotonically decreasing sequence $\lambda_0\>\lambda_1\>\cdots$. This program will stop if either $\lambda_i=0$ for some $i$, in which case $K_{X_i}+\Delta_i$ is nef, or there is a contraction $\vphi:X_i\to Y_i$ such that $\dim Y_i<\dim X_i$ and $-(K_{X_i}+\Delta_i)$ is $\vphi_i$-ample; this is called a $(K_{X_i}+\Delta_i)$-Mori fiber space.\\

\begin{theorem}[Special MMP]\label{subsec:special-mmp} Let $(X, \Delta)$ be a $\mbQ$-factorial compact K\"ahler $3$-fold klt pair. Suppose that $K_X+\Delta$ is not nef and there is an effective $\mbQ$-divisor $H$ on $X$ such that $K_X+\Delta+H$ is nef. Then we can run a $(K_X+\Delta)$-MMP with the scaling of $H$ which satisfies exactly one of following two properties:
 \[
        (X, \Delta)=(X_0, \Delta_0)\bir (X_1, \Delta_1)\bir\cdots (X_i, \Delta_i)\bir\cdots
 \]
\begin{enumerate}
        \item  Every step of this MMP is $(K_X+\Delta+H)$-trivial, i.e., $(K_{X_i}+\Delta_i+H_i)\cdot R_i=0$ for all $i\>0$, where $R_i$ is the corresponding extremal ray. In this case the program terminates with either a $(K_X+\Delta)$-minimal model or $(K_X+\Delta)$-Mori fiber space,
         \begin{center}
                or
         \end{center}
        \item there is a smallest $i\>0$ such that $K_{X_i}+\Delta_i+(1-\ve)H_i$ is nef for all $0<\ve\ll 1$. In this case we stop the program at this stage.
\end{enumerate}
\end{theorem}   

\begin{proof}
        Using the same notations as in the beginning of Subsection \ref{subsec:mmp-with-scaling} we see that the property $(1)$ holds if $\lambda_i=1$ for all $i\< k-1$ and $\lambda_k=0$ for some $k\>1$, such that $K_{X_k}+\Delta_k$ is either nef or there is a $(K_{X_k}+\Delta_k)$-Mori fiber space $g:X_k\to Z$. Otherwise, assume that $0<\lambda_i<1$ with the smallest index $i\>0$. Then $\lambda_{i-1}=1$ and $K_{X_i}+\Delta_i+(1-\ve)H_i$ is nef for all $0<\ve\ll 1$ (we choose $\lambda_{-1}=1$).
\end{proof}

We will also need the following version of relative directed MMP.
\begin{lemma}\label{lem:directed-relative-mmp}
	Let $(X, \D)$ be a $\mbQ$-factorial compact K\"ahler $3$-fold lc pair such that $X$ has klt singularities. Let $f:X\to Y$ be a projective contraction of relative dimension $1$ such that the general fibers of $f$ are isomorphic to $\mbP^1$. Assume that $K_X+\D$ is nef over $Y$ and $(K_X+\D)\cdot F=0$ for general fibers $F$ or $f$. We run a relative $K_X$-MMP over $Y$ with the scaling of $\Delta$
	\[
		(X, \Delta)=:(X_0, \Delta_0)\bir (X_1, \D_1)\bir\cdots \bir (X_n, \D_n)=:X'\to Y'
	\]
and end with a Mori fiber space $f':X'\to Y'$ over $Y$.\\
Then every step of this MMP is $(K_X+\Delta)$-trivial.	
\end{lemma}

\begin{proof}
 From the construction of MMP with scaling we understand that it is enough to show that the nef thresholds 
 \[\lambda_i:=\inf\{t\>0: K_{X_i}+t\D \mbox{ is nef$/Y$}\}=1,\]
  for all $0\<i\<n$. Note that $\{\lambda_i\}$ form a monotonically decreasing sequence and $\lambda_i\<1$ for all $0\<i\<n$. Therefore it is enough to show that $\lambda_n=1$. To that end, first observe that the composite bimeromorphic map $\phi:X\bir X'$ is an isomorphism over a dense Zariski open subset of $Y$. In particular, if $F(\cong\mbP^1)$ is a general fiber of $X\to Y$, then $(K_{X'}+\Delta')\cdot F=(K_X+\Delta)\cdot F=0$, i.e. $\Delta'\cdot F=2$. Note that the general fibers of $X\to Y$ are also general fibers of $f':X'\to Y'$, and according to the construction of MMP with scaling, $f'$ is a $(K_{X'}+\lambda_n\Delta')$-trivial contraction, i.e. $(K_{X'}+\lambda_n\Delta')\cdot F=0$ for a general fiber $F$ of $f'$. Thus $\Delta'\cdot F=\frac{2}{\lambda_n}$ and therefore $\lambda_n=1$.	
\end{proof}

\section{Relative MMP for dlt pairs and existence of dlt model}
In this section we will prove a relative cone theorem for dlt pairs and then establish the existence of dlt models for log canonical pairs.

\begin{theorem}\label{thm:dlt-cone}
	Let $(X, \Delta)$ be a $\mbQ$-factorial compact K\"ahler dlt pair of arbitrary dimension, and $f:X\to Y$ is a projective surjective morphism to a compact analytic variety $Y$. Then there are at most countable number of curves $\{C_i\}_{i\in I}$ such that $f(C_i)=\pt$ for all $i\in I$, and 
	\[
		\NE(X/Y)=\NE(X/Y)_{(K_X+\Delta)\>0}+\sum_{i\in I}\mbR^+\cdot[C_i].
	\] 
\end{theorem}

\begin{proof}
	Fix a $f$-ample divisor $H$ on $X$. Then for any $n\in\mbN$ we can write $K_X+\Delta+\frac 1n H=K_X+(1-\ve)\Delta+(\frac 1n H+\ve \Delta)$ such that $\frac 1n H+\ve \Delta$ is $f$-ample for $\ve\in\mbQ^+$ sufficiently small (depending on $n$). Note that $(X, (1-\ve)\Delta)$ is a klt pair. Thus by \cite[Theorem 4.12]{Nak87}, there are finitely many $(K_X+\Delta+\frac 1n H)$-negative extremal rays generated by curves $\{C_i\}_{i\in I_n}$ such that 
	\begin{equation}\label{eqn:perturbed-cone}
		\NE(X/Y)=\NE(X/Y)_{(K_X+\Delta+\frac 1n H)\>0}+\sum_{i\in I_n}\mbR^+\cdot [C_i].
	\end{equation}
Define $I:=\cup_{n\>1} I_n$. Then clearly $\NE(X/Y)=\NE(X/Y)_{(K_X+\Delta+\frac 1n H)\>0}+\sum_{i\in I}\mbR^+\cdot [C_i]$. Note that we also have 

\[\NE(X/Y)_{(K_X+\Delta)\>0}=\cap_{n=1}^\infty \NE(X/Y)_{(K_X+\Delta+\frac 1n H)\>0}.\]
 Therefore from \eqref{eqn:perturbed-cone} we have

\begin{align*}
	\NE(X/Y) &=\cap_{n=1}^\infty \left(\NE(X/Y)_{(K_X+\Delta+\frac 1n H)\>0}+\sum_{i\in I}\mbR^+\cdot [C_i]\right)\\
			 &\supset\cap_{n=1}^\infty \left(\NE(X/Y)_{(K_X+\Delta+\frac 1n H)\>0}\right)+\sum_{i\in I}\mbR^+\cdot [C_i]\\
			 &=\NE(X/Y)_{(K_X+\Delta)\>0}+\sum_{i\in I}\mbR^+\cdot [C_i].
\end{align*}
Suppose now that the inclusion is strict and so we have an element $v\in \cap_{n=1}^\infty \left(\NE(X/Y;W)_{(K_X+\Delta+\frac 1n H)\>0}+\sum_{i\in I}\mbR^+\cdot [C_i]\right)$ not contained in 
	\[\NE(X/Y;W)_{(K_X+\Delta)\>0}+\sum_{i\in I}\mbR^+\cdot [C_i].\]
	Intersecting $\NE(X/Y;W)$ with an appropriate affine hyperplane $\mathcal H$ we may assume that $\NE(X/Y;W)\cap \mathcal H$ is compact and convex and $v\in \NE(X/Y;W)\cap \mathcal H$. For each $n\>1$, we can write $v=v_n+w_n$, where $v_n\in \NE(X/Y;W)_{(K_X+\Delta+\frac 1n H)\>0}\cap \mathcal H$ and $w_n\in \sum_{i\in I}\mbR^+\cdot [C_i]\cap \mathcal H$. By compactness, passing to a subsequence, we may assume that limits exist, and  $v_\infty =\lim v_i$ and $w_\infty =\lim w_i$ such that $v=v_\infty +w_\infty$. Since $\NE(X/Y;W)_{(K_X+\Delta)\>0}= \cap_{n=1}^\infty \NE(X/Y;W)_{(K_X+\Delta+\frac 1n H)\>0}$ is closed, $v_\infty\in \NE(X/Y;W)_{(K_X+\Delta)\>0}\cap \mathcal H$.
	Since $\overline {\sum_{i\in I}\mbR^+\cdot [C_i]}\cap \mathcal H$ is compact, $w_\infty\in\overline {\sum_{i\in I}\mbR^+\cdot [C_i]}\cap \mathcal H$. 
	By standard arguments (see the end of the proof of \cite[Theorem III.1.2]{Kol96}) one sees that $\NE(X/Y;W)_{(K_X+\Delta)\>0}+{\sum_{i\in I}\mbR^+\cdot [C_i]}$ is closed and hence \[\overline {\sum_{i\in I}\mbR^+\cdot [C_i]}\subset \NE(X/Y;W)_{(K_X+\Delta)\>0}+{\sum_{i\in I}\mbR^+\cdot [C_i]}.\]
	Thus $w_\infty =v_0+w'_\infty$, where $v_0\in \NE(X/Y;W)_{(K_X+\Delta)\>0}$ and $w'_\infty\in {\sum_{i\in I}\mbR^+\cdot [C_i]}$. Finally, since $v=(v_\infty +v_0)+w'_\infty$, we obtain the required contradiction.
 
\end{proof}~\\

\begin{theorem}[DLT Modification]\label{thm:dlt-model}
	Let $(X, \Delta)$ be a compact K\"ahler $3$-fold log canonical pair. Then there exists a projective bimeromorphic morphism $f:(X', \Delta')\to (X, \Delta)$ from a $\mbQ$-factorial dlt pair $(X', \Delta')$ such that $K_{X'}+\Delta'=f^*(K_{X'}+\Delta')$.
\end{theorem}

\begin{proof}
	Let $g:Y\to X$ be the log resolution of $(X, \Delta)$. Define $\Delta_Y:=f^{-1}_*\Delta+E$, where $E=\Ex(f)$ is the reduced $f$-exceptional divisor. Then $(Y, \Delta_Y)$ has dlt singularities.  We will run a $(K_Y+\Delta_Y)$-MMP over $X$. Let $C\subset Y$ be a curve generating a $(K_Y+\Delta)$-negative extremal ray of $\NE(Y/X)$ as in Theorem \ref{thm:dlt-cone}. Then for some $0<\ve\ll 1$ we still have $(K_Y+(1-\ve)\Delta_Y)\cdot C<0$ and $(Y, (1-\ve\Delta_Y))$ is klt. Therefore by \cite[Theorem 4.12(2)]{Nak87} there exist projective bimeromorphic morphisms $\vphi:Y\to Z$ and $h:Z\to X$ such that $\vphi$ contracts the extremal ray $R=\mbR^{\>0}[C]$ and $g=h\circ\vphi$. Now since $X$ is K\"ahler and $h:Z\to X$ is a projective morphism, from \cite[Pro. 1.3.1, page 24]{Var89} it follows that $Z$ is K\"ahler. Now if $\vphi$ is a flipping contraction, then the flip, say $\vphi^+:Y^+\to Z$ exists by Theorem \ref{thm:existence-of-flips}. The fact that $Y^+$ is K\"ahler again follows from \cite[Pro. 1.3.1, page 24]{Var89}, since $Z$ is K\"ahler and $\vphi^+$ is projective. The termination of flips follows from Theorem \ref{thm:termination}. Finally, since $f$ is bimeromorphic, $K_Y+\Delta_Y$ is $f$-big, and consequently this MMP terminates with a minimal model over $X$. Then applying the negativity lemma we get our required dlt model.

\end{proof}

\begin{corollary}
\label{cor:exist-terminal-dlt-modification}
Let  $(X, \Delta)$ be a compact K\"ahler $3$-fold lc pair. Then there exists a bimeromorphic model  $f : (X', \Delta') \to (X,\Delta)$ such that 
\begin{enumerate}
\item  $X'$ has $\mathbb{Q}$-factorial terminal singularities,
\item $(X',\Delta')$ is a dlt pair, and
\item $K_{X'}+\Delta'=f^*(K_X+\Delta)$.
\end{enumerate}
\end{corollary}

\begin{proof}
By Theorem \ref{thm:dlt-model},  there is a dlt  model $g: (\hat{X}, \hat{\Delta}) \to (X, \Delta)$  such that  $K_{\hat{X}}+\hat{\Delta}=g^*(K_X+\Delta)$. Hence, by replacing $(X,\Delta)$ by $(\hat{X}, \hat{\Delta})$, we may assume that $(X,\Delta)$ is dlt. Let $U\subset X$ be the largest open set such that $(U, \Delta|_U)$ is a SNC pair. Then $\codim_X (X \backslash U)\>2$. Let $f: (X',\Theta')\to (X, 0)$ be a terminal model of $(X,0)$ as in Theorem \ref{thm:terminal-model} such that $K_{X'}+\Theta'=f^*K_X$. Then $f$ is an isomorphism over the smooth locus of $X$; in particular $f$ is an isomorphism over $U$. Let $Z=X\backslash U$. Define $\Delta':=\Theta'+f^*\Delta$ on $X'$ so that 
\[K_{X'}+\Delta'=f^*(K_X+\Delta),\]
and $(X', \Delta')$ is lc.\\
It remains to show that $(X',\Delta')$ is a dlt pair.  Let $U'=f^{-1}(U)$ and $Z' = X' \backslash U'$. Then $(U',\D'|_{U'})$ is a SNC pair. If $E$ is an exceptional divisor with center in  $Z'$, then its center in $X$ is contained in $Z$. Hence $a(E; X', \Delta')=a(E; X, \Delta)>-1$. This completes the proof.
\end{proof}

%%%%%%%%%%%%%%%%%%%%%%%%%%%%%%%%%%%%%%%%%%%%%%%%%%%%%%%%%%%%%%%%%%%%%%%%%%%%%%%%%%%%%%%%%%%%%%%%%%%%%%%%%%%%%%%%%%%%%%%%%%%%%%%%%%%%%%%%%%%%%%%%%%%%%%%%%%%%%%%%%%%%%
%%%%%%%%%%%%%%%%%%%%%%%%%%%%%%%%%%%%%%%%%%%%%%%%%%%%%%%%%%%%%%%%%%%%%%%%%%%%%%%%%%%%%%%%%%%%%%%%%%%%%%%%%%%%%%%%%%%%%%%%%%%%%%%%%%%%%%%%%%%%%%%%%%%%%%%%%%%%%%%%%%%%%
%%%%%%%%%%%%%%%%%%%%%%%%%%%%%%%%%%%%%%%%%%%%%%%%%%%%%%%%%%%%%%%%%%%%%%%%%%%%%%%%%%%%%%%%%%%%%%%%%%%%%%%%%%%%%%%%%%%%%%%%%%%%%%%%%%%%%%%%%%%%%%%%%%%%%%%%%%%%%%%%%%%%%

 \section{LMMP and Abundance for K\"ahler surfaces}\label{sec:surface-mmp}

In  this section, we will recall the log minimal model programs for dlt pair $(X,\D)$, where $X$ is a compact K\"ahler surface. Since $(X,\D)$ is dlt, $X$ has $\mathbb{Q}$-factorial rational singularity. Now if $X$ is uniruled, then from the classification of compact complex surfaces it follows that $X$ is algebraic. So we will only focus on the case when $X$ is not uniruled.  Let $\vphi:\tilde{X} \to X$ be the minimal resolution of $X$. Then there is an effective $\mbQ$-divisor $E\> 0$ such that \[K_{\tilde{X}} \sim_{\mathbb{Q}} \vphi^*K_X-E.\] Since $\tilde{X}$ is not uniruled, by abundance for smooth compact K\"ahler surfaces it follows that $K_{\tilde{X}}$ is $\mathbb{Q}$-effective. As a consequence, $K_X$ is also $\mathbb{Q}$-effective. We can write  
\begin{equation}\label{eqn:surface-non-vanishing}
K_X\sim_{\mathbb{Q}} \sum_{\mathrm{finite}} \lambda_i C_i,
\end{equation}
where $C_i$ are irreducible curves and $\lambda_i\>0$. By \cite[Proposition 3.4]{Bou04}, $\vphi^*(K_X+\Delta)$ is nef if and only if it has non-negative intersection with every curve in $\tilde X$. Thus from the proof of \cite[Lemma 3.13]{HP16} it follows that $K_X+\Delta$ is nef if and only if it intersects every curve in $X$ non-negatively. In this case, we have the  following cone theorem.

\begin{theorem}
\label{thm:surface-cone-thm}
Let $(X,\D)$ be a dlt pair such that $X$ is a non-uniruled compact K\"ahler surface.  If $K_X+\D$ is not nef, then there are finitely many irreducible curves $\Gamma_i$ such that \[-(K_X+\D)\cdot \Gamma_i > 0 \] and that \[\overline{\mathrm{NA}}(X) = \overline{\mathrm{NA}}(X)_{(K_X+\D)\>0} + \sum_\textsubscript{\mbox{finite}} \mathbb{R}^{+} [\Gamma_i].\]
\end{theorem}

\begin{proof}
 Since $K_X+\Delta$ is not nef, as explained above, there is a curve $\Gamma\subset X$ such that $(K_X+\Delta)\cdot\Gamma<0$. In particular, from \eqref{eqn:surface-non-vanishing} it follows that $\Gamma\subset\Supp(\Delta+\sum\lambda_iC_i)$ and $\Gamma^2<0$. Let $\{\Gamma_i\}_{i\in I}$ be a finite set of curves contained in the support of $\Delta+\sum\lambda_iC_i$ such that $\Gamma_i^2<0$. 
 We claim that the following holds
 \begin{equation}\label{eqn:generalized-mori-cone}
     \NA(X)=\NA(X)+\sum_{i\in I}\mbR^+[\Gamma_i].
 \end{equation}
 To see this, let $\phi:\hat X\to X$ be a resolution of singularities of $X$. Then from \cite[Corollary 0.4]{DP04} it follows that $\NA(\hat X)$ is generated by the set of all K\"ahler classes $\hat\omega$ and classes of curves $\hat C$ on $\hat X$. Since $\phi_*\NA(\hat X)=\NA(X)$ (cf. \cite[Proposition 3.14]{HP16}), it follows that classes $\phi_*\hat\omega$ and $\phi_*[\hat C]$ generate $\NA(X)$. So it is enough to show that these classes are contained in the closed cone $V:=\NA(X)_{(K_X+\Delta)\>0}+\sum_{i\in I}\mbR^+[\Gamma_i]$. To that end, first observe that $(K_X+\Delta)\cdot\phi_*\hat\omega=\phi^*(K_X+\Delta)\cdot\hat\omega\>0$, since $\hat\omega$ is K\"ahler and $K_X+\Delta$ is effective. Thus $\phi_*\hat\omega\in V$. Now let $C:=\phi_*\hat C$. If $(K_X+\Delta)\cdot C\>0$, then clearly $[C]\in V$. If $(K_X+\Delta)\cdot C<0$, the from the argument above it follows that $C\num a\Gamma_i$ for some $i\in I$ and $a\in\mbR^+$; hence $[C]\in V$. This completes the proof.
 
\end{proof}

In order to establish the MMP, we still need the following contraction theorem. 

\begin{theorem}
\label{thm:surface-contraction-thm}
With the notation and hypothesis as in Theorem \ref{thm:surface-cone-thm}, for each irreducible curve $\Gamma$ with $-(K_X+\D)\cdot \Gamma > 0$, there is a bimeromorphic morphism $f:X\to Y$ which contracts precisely the curve $\Gamma$ and $Y$ is a compact K\"ahler surface.
\end{theorem}

\begin{proof}
Since $K_X$ is $\mbQ$-effective as discussed above, it follows that $\Gamma \cdot \Gamma <0$. Then by \cite[Theorem 1.2]{Sak84}, there is a contraction $f:X\to Y$ to a normal surface $Y$ such that $f(\Gamma)=\pt$.

Note that since $\Gamma^2<0$, from adjunction it follows that $\Gamma$ is a rational curve. Then by a similar argument as in proof of \cite[Corollary 3.1]{CHP16} it follows that $Y$ is K\"ahler.
\end{proof}~\\

\begin{theorem}[Abundance]\label{thm:surface-abundance}
        Let $X$ be a normal compact K\"ahler space of dimension $2$ and $(X, \Delta)$ a lc pair. If $K_X+\Delta$ is nef, then it is semi-ample.
\end{theorem}

\begin{proof}
        Replacing $X$ by its minimal resolution we may assume that $X$ is smooth. If $\nu(K_X+\Delta)=0 \text{ or } 1$, then from \cite[Pro. 5.1 and Lemma 5.1]{CHP16} it follows that $K_X+\Delta$ is semi-ample.\\
        
         If $\nu(K_X+\Delta)=2$, then by \cite[Corollary 6.8]{Dem96} $K_X+\Delta$ is big; in particular $X$ is a Moishezon manifold. Then by \cite[Theorem 1.6]{Nam02} $X$ is projective, and the semi-ampleness of $K_X+\Delta$ follows from \cite[Theorem 11.1.3]{Kol92}.

\end{proof}

\section{Abundance for 3-fold pairs $(X, \Delta)$ with $\kappa(X, K_X+\Delta)\>1$}
In this section we will prove the log abundance for a $3$-fold log canonical pairs $(X, \Delta)$ such that $\kappa(X, K_X+\Delta)\>1$. The proof is based on the proof of \cite[Theorem 3.1]{DW19} which is originally inspired from the proof of \cite[Theorem 7.8]{Kaw85}. 
\begin{lemma}\label{lem:generic-isomorphism}
Let $f:X\to Y$ be a proper surjective morphism from a normal compact K\"ahler variety $X$ to a normal projective variety $Y$ of $\dim Y>0$. Suppose that $\msL$ is a line bundle on $X$ such that $H^0(X, \msL)\neq 0$ and $\msL|_F\lin \mcO_F$ for general fibers $F$ of $f$. Let $D$ be a non-zero effective Cartier divisor on $X$ given by the zeros of a non-zero global section of $\msL$. Then there exists an effective ample Cartier divisor $H$ on $Y$ such that $f^*H-D$ is effective.

\end{lemma}

\begin{proof}
	Since $\msL|_F\cong\mcO_F$ for general fibers $F$ of $f$, it follows that $D|_F\sim 0$. We claim that $\Supp (D)\cap F=\emptyset$ for general fibers $F$ of $f$. Indeed, since $D$ is effective, for a general fiber $F$, $D|_F$ is an effective divisor on $F$. Since $D|_F\sim 0$ and $F$ is a normal compact K\"ahler variety, it follows that $D|_F=0$, i.e. $\Supp D\cap F=\emptyset$ for general fibers $F$ of $f$. In particular, $f(D)$ does not dominate $Y$. Let $H$ be an effective ample Cartier divisor $Y$ such that $\Supp H$ contains $f(D)$. Then $\Supp D$ is contained in the $\Supp f^*H$. Thus replacing $H$ by a sufficiently high multiple we may assume that $f^*H-D$ is effective.

\end{proof}

\subsection{Iitaka Fibration}\label{subc:iitaka-fibration}
Let $(X, \Delta)$ be a terminal pair, where $X$ is compact K\"ahler 3-fold. Assume that with $\kappa(X, K_X+\Delta)\geq 0$. Then the Iitaka fibration as in \cite[Theorem 2.1.33]{Laz04a} gives a diagram
                \[
                \xymatrixcolsep{3pc}\xymatrixrowsep{3pc}\xymatrix{  &Y\ar[dl]_\mu\ar[dr]^f&\\
                X\ar@{-->}[rr] && Z}
                \]
satisfying the following conditions:
\begin{enumerate}
        \item $Y$ is a smooth compact K\"ahler $3$-fold, and $Z$ is a smooth projective variety with $\dim Z=\kappa(X, K_X+\Delta)$.
        \item $\mu:Y\to X$ is a log resolution of $(X, \Delta)$ and $f$ is a surjective morphism with connected fibers.
        \item $\kappa(W, (K_Y+\Delta_Y)|_W)=0$, where $W$ is a general fiber of $f$, and $\Delta_Y=\mu^{-1}_*\Delta$ (see \cite[Lemma 2.3]{LP18}).\\
        \end{enumerate}

\begin{theorem}
\label{thm-kappa>0-implies-semiample}
        Let $(X, \Delta)$ be a compact K\"ahler $3$-fold dlt pair. If $\kappa(X, K_X+\Delta)\>1$ and $K_X+\Delta$ is nef, then it is semi-ample. 
\end{theorem}

\begin{proof}
        If $\kappa(X, K_X+\Delta)=3$, then $X$ is a Moishezon space. Since dlt singularities are rational, it follows from \cite[Theorem 1.6]{Nam02} that $X$ is projective algebraic varieties. Then semi-ampleness follows from \cite[Theorem 1.1]{KMM94}.\\

In the following discussion we will consider the remaining cases where $\kappa(X, K_X+\Delta)=1\text{ or } 2$. 
By passing to a terminal model by Theorem \ref{thm:terminal-model} we may assume that $(X, \Delta)$ has $\mbQ$-factorial terminal singularities.
        
                Let the following diagram be the Iitaka fibration of $K_X+\Delta$ as defined above
                \[
                \xymatrixcolsep{3pc}\xymatrixrowsep{3pc}\xymatrix{  &Y\ar[dl]_\mu\ar[dr]^f&\\
                X\ar@{-->}[rr] && Z.}
                \]

Let $\{E_i\}$ be the exceptional divisors of $\mu$, and  set $\Delta_Y=\mu^{-1}_*\Delta$. Then we have    
\begin{equation}\label{eqn:mu-log}
        K_Y+\Delta_Y=\mu^*(K_X+\Delta)+\sum r_iE_i, \quad r_i>0 \text{ for all } i.
        \end{equation}  
        
Let $W$ be a general fiber of $f: Y\to Z$. Since $Y$ is smooth, by generic smoothness it follows that $W$ is smooth and $(W, \Delta_W)$ is a SNC pair, where $\Delta_W=\Delta_Y|_W$. In particular, $(W, K_W+\Delta_W)$ has dlt singularities.

\begin{claim}\label{clm:trivial-restriction}
         $\mu^*(K_X+\Delta)|_W\sim_\mbQ 0$.
\end{claim}

Assuming the claim for the time being first we will complete the proof here. Let $m>0$ be an integer such that $\mu^*(m(K_X+\Delta))$ is a Cartier divisor. Let $D$ be an effective Cartier divisor on $Y$ given by the zeros of a non-zero global section of the line bundle $\mcO_Y(m(\mu^*(K_X+\Delta)))$. Then by Lemma \ref{lem:generic-isomorphism}, there exists an ample Cartier divisor $H$ on $Z$ such that $D\<f^*H$. Therefore $\nu(X, K_X+\Delta)=\nu(Y, \mu^*(m(K_X+\Delta)))=\nu(Y, D)\<\nu(Y, f^*H_2)=\kappa(Y, f^*H_2)=\kappa(Z, H_2)=\dim Z=\kappa(X, K_X+\Delta)$. Thus $\kappa(X, K_X+\Delta)=\nu(X, K_X+\Delta)$ by \cite[Proposition 6.10]{Dem96}. Then by \cite[Theorem 4.8]{Fuj11} $K_X+\Delta$ is semi-ample.\\
\end{proof}

\begin{proof}[Proof of Claim \ref{clm:trivial-restriction}]
We spilt the proof into two cases depending on the dimension of $W$.\\

\textbf{Case I:} $\dim W=1$. In this case $W$ is a smooth proper curve, hence it is a smooth projective algebraic curve. Since $\kappa(W, K_W+\Delta)=0$, it follows that $K_W+\Delta_W\sim_\mbQ 0$. Restricting both side of \eqref{eqn:mu-log} to $W$ we get
\[
        \mu^*(K_X+\Delta)|_W\sim_\mbQ -\sum r_iE_i|_W.
\]
 Since $K_X+\Delta$ is nef and $r_i>0$ for all $i$, we conclude that $E_i|_W=0$ for all $i$, i.e.,  $\mu^*(K_X+\Delta)|_W\sim_\mbQ 0$.\\

 \textbf{Case II:} $\dim W=2$. In this case $W$ is a smooth K\"ahler surface and $(W, \Delta_W)$ is a dlt pair. Now by Theorem \ref{thm:surface-cone-thm}, \ref{thm:surface-contraction-thm} and \ref{thm:surface-abundance} we run the $(K_W+\Delta_W)$-MMP which terminates with a dlt pair $(W', \Delta_{W'})$ such that $K_{W'}+\Delta_{W'}$ is semi-ample. Let $\vphi:W\to W'$ be the induced bimeromorphic morphism such that $K_{W'}+\Delta_{W'}=\vphi_*(K_W+\Delta_W)$. Then we have
\begin{equation}\label{eqn:sigma-log} 
                K_W+\Delta_W=\vphi^*(K_{W'}+\Delta_{W'})+\sum s_jF_j, \mbox{ where } s_j>0 \text{ for all } j,
        \end{equation}  
        where $F_j$'s are the $\vphi$-exceptional divisors.\\ 
        
Since $\kappa(W', K_{W'}+\Delta_{W'})=\kappa(W, K_W+\Delta_W)=0$ and $K_{W'}+\Delta_{W'}$ is semi-ample, $K_{W'}+\Delta_{W'}\sim_\mbQ 0$.         
Therefore we have 
\begin{equation}\label{eqn:sigma-log-trivial}
        K_W+\Delta_W\sim_\mbQ\sum s_jF_j, \mbox{ where } s_j>0 \text{ for all } j.
        \end{equation}          
From adjunction on $W$ we also get that
\begin{equation}\label{eqn:mu-log-adj}
        K_W+\Delta_W\sim_\mbQ\mu^*(K_X+\Delta)|_W+\sum r_iE_i|_W
        \end{equation}
Thus we have
\begin{equation}\label{eqn:mu-sigma-relation}
        \mu^*(K_X+\Delta)|_W\sim_\mbQ\sum s_jF_j-\sum r_iE_i|_W=G^+-G^-,
        \end{equation}          
such that $G^+\>0$ and $G^-\>0$ are two effective $\mbQ$-divisors on $W$ with no common irreducible components.\\

We will show that $G^+=G^-=0$. On the contrary first assume that $G^+\neq 0$. It is clear that $G^+$ is $\vphi$-exceptional. Now by \cite[Theorem III.2.1]{BHPV04} the intersection matrix of the exceptional divisors of $\vphi: W\to W'$ is a negative definite matrix. Thus in particular $(G^+)^2<0$.\\
On the other hand, since $\mu^*(K_X+\Delta)|_W$ is nef, from relation \eqref{eqn:mu-sigma-relation} we get $G^+\cdot (G^+-G^-)\>0$, which is a contradiction.       Therefore $G^+=0$ and $\mu^*(K_X+\Delta)|_W\sim_\mbQ -G^-$. Again since $\mu^*(K_X+\Delta)|_W$ is nef, $G^-=0$. Therefore we have
\begin{equation}\label{eqn:mu-linearly-0}
        \mu^*(K_X+\Delta)|_W\sim_\mbQ 0.
        \end{equation}

\end{proof}

\begin{corollary}\label{cor:lc-abundance}
        Let $(X, \Delta)$ be a compact K\"ahler $3$-fold lc pair. If $\kappa(X, K_X+\Delta)\>1$ and $K_X+\Delta$ is nef, then it is semi-ample.
\end{corollary}
\begin{proof}
        It result follows immediately from Theorem \ref{thm-kappa>0-implies-semiample} by passing to a dlt model using Theorem \ref{thm:dlt-model}. 
\end{proof}

\section{ $\mathbb{Q}$-linear system}

Let $X$ be a normal compact analytic variety and $L$ a $\mbQ$-Cartier divisor on $X$. Then the $\mathbb{Q}$-linear system $|L|_{\mathbb{Q}}$ is defined as the set of all effective $\mbQ$-Cartier divisors $L'\>0$ such that $m(L-L')$ is an integral Weil divisor for some positive integer $m\in\mbN$ and that $m(L-L')\sim 0$. The \textit{stable of base locus} of $L$ is defined as 
\[\SBs|L|:=\bigcap_{m\in\mbN}\Bs|mkL|,\] 
where $k$ is the small positive integer such that $kL$ is Cartier.\\
Note that since $X$ is compact, it follows that $\SBs|L|$ is an analytic subset of $X$, in fact $\SBs|L|=\Bs|mkL|$ for some $m\in\mbN$ (see \cite[Pro. 2.1.21]{Laz04a}).
The \emph{base locus} of the $\mbQ$-linear system $|L|_{\mbQ}$ is defined as
\[
        \Bs|L|_{\mbQ}=\bigcap_{L'\in|L|_{\mbQ}}\Supp L'.
\]
Note that $\SBs|L|=\Bs|L|_\mbQ$.\\
For a $\mbQ$-Cartier divisor $D$ on $X$, we define 
\[D+|L|_{\mathbb{Q}} = \{D+P \ | \ P\in |L|_{\mathbb{Q}} \}.\] 
 Assume that $|L|_{\mathbb{Q}}$ is not empty. Then  for any prime Weil divisor $E$ on $X$, we define the \emph{multiplicity} of $D+|L|_{\mathbb{Q}}$ along $E$ as 
 \[\mult_E \, (D+|L|_{\mathbb{Q}}) = \inf\{\mult_E \, (D+P) \ | \   P \in |L|_{\mathbb{Q}}  \}.\]   
We call $D+|L|_{\mathbb{Q}}$  a  \emph{boundary} (resp. \emph{pure boundary}, resp. \emph{sub-boundary}, and resp. \emph{pure sub-boundary}) if its multiplicity along every prime divisor lies in the interval $[0,1]$ (resp. $[0,1)$, resp. $(-\infty,1]$, and resp.$(-\infty, 1)$).\\
 We also consider the pair $(X, D+|L|_{\mathbb{Q}})$.  For a proper \emph{generically finite} morphism $f:X'\to X$ from a normal variety $X'$, we define the log pullback of $K_X+D+|L|_{\mathbb{Q}}$ as $K_{X'}+D'+|L'|_{\mathbb{Q}}$, where  \[L'=f^*L  \quad\mbox{and } K_{X'}+D'=f^*(K_X+D).\]    

We then define the discrepancy of a divisorial valuation $\nu$ on $X$ for the pair $(X, D+|L|_{\mathbb{Q}})$ as follows: choose a proper bimeromorphic morphism  $f:X'\to X$ from a smooth variety $X'$ such that $\nu$ corresponds to a prime divisor $E$ on $X'$. Then the \emph{discrepancy} of $\nu$ with respect to $(X, D+|L|_{\mathbb{Q}})$ is defined as  \[a(\nu; X, D+|L|_{\mathbb{Q}})=a(E; X, D+|L|_{\mathbb{Q}}):= - \mathrm{mult}_E\, (D'+|L'|_{\mathbb{Q}}).\]
  The pair $(X, D+|L|_{\mathbb{Q}})$ is called \emph{sub-klt} if the discrepancies of all divisorial valuations are strictly greater than $-1$.  It is called \emph{klt} if it is sub-klt and $D$ is effective. Other singularities are defined similarly.

Now suppose that $L$ is a Cartier divisor. The we can decompose it as: \[|L|=|M|+F,\] where $|M|$ is the movable part, i.e. $\Bs|M|$ does not have any codimension $1$ component, and $F\>0$ is the fixed part. Note that the divisor $M$ may not be Cartier here, so in general $|M|$ is a linear system of Weil divisors.\\
If $D$ is another $\mbQ$-Cartier divisor, then we define the \emph{fixed part} of $D+|L|$  as $D+F$, and the \emph{movable part} as $|M|$.\\

In the remainder of this section we establish some properties of $\mathbb{Q}$-linear systems of the form  $D+|L|_{\mathbb{Q}}$ and the corresponding pairs $(X, D+|L|_\mbQ)$.

\begin{lemma}
\label{lem-sub-boundary-system-implies-sub-boundary}
Let $X$ be normal $\mbQ$-factorial compact analytic variety, $|L|_\mbQ$ a $\mbQ$-linear system on $X$, and $D$ is a $\mbQ$-divisor on $X$. If $D+|L|_\mbQ$ is a pure sub-boundary, then there exist an element $N\in |L|_{\mathbb{Q}}$ such that $D+N$ is a pure sub-boundary.
\end{lemma}

\begin{proof}

 First replacing $L$ by an effective $\mbQ$-Cartier divisor in $|L|_{\mbQ}$ we may assume that $L$ is effective. Let $k$ be the smallest positive integer such that $kL$ is Cartier. Then for any $m\>1$ we decompose $|mkL|=|M_m|+F_m$ into movable and fixed parts. Note that these fixed parts give a monotonically decreasing sequence of divisors $\{\frac{1}{mk}F_m\}$ on $X$, i.e., $\frac{1}{mk}F_m\>\frac{1}{(m+1)k}F_{m+1}$ for all $m\>1$. Now since $D+|L|_\mbQ$ is a pure sub-boundary, there is a $m_0\gg 0$ such that the coefficients of the divisor $D+\frac{1}{m_0k}F_{m_0}$ along each of its irreducible component is strictly less than $1$. Now we focus on the movable part $\frac{1}{m_0k}|M_{m_0}|$. We make the following claim.
 \begin{claim}\label{clm:bertini}
        There is an effective reduced divisor $N'\>0$ in the linear system of $|M_{m_0}|$, i.e., the coefficients of $N'$ are all $1$, such that the support of $N'$ does not contain any irreducible components of $D+\frac{1}{m_0k}F_{m_0}$. 
 \end{claim}
 Assuming the claim for the time being we will complete the proof first. Set $N:=\frac{1}{m_0k}N'+\frac{1}{m_0k}F_{m_0}\in|L|_\mbQ$; then clearly $D+N$ is a pure sub-boundary.  
 
 \begin{proof}[Proof of Claim \ref{clm:bertini}]
        First note that, since $X$ is compact, the linear system $|mkL|$ is finite dimensional for all $m\>1$; more specifically, $|mkL|$ is given by the finite dimensional $\mbC$-vector space $H^0(X, \mcO_X(kL))$ for all $m\>1$. In particular, $|M_m|$ is a finite dimensional linear system of Weil divisors on $X$, for all $m\>1$. Let $X_{\sing}$ be the singular locus of $X$; note that $\codim_X X_{\sing}\>2$, since $X$ is normal. Now consider the closed subset $Y=\Supp(D+\frac{1}{m_0k}F_{m_0})\cup\Bs|M_{m_0}|\cup X_{\sing}\subset X$ and set $U=X\setminus Y$. Then $U$ is a smooth complex space. Notice that the restricted linear system $|M_{m_0}|_U$ is a finite dimensional linear sub-system of $|(M_{m_0}|_U)|$ of Cartier divisors (since $U$ is smooth). Furthermore, $|M_{m_0}|_U$ is base-point free on $U$. Then by a Bertini type theorem for complex spaces \cite[Theorem II.5]{Man82}, a general enough member $N'_U$ of $|M_{m_0}|_U$ is a reduced divisor on $U$. Note that a very general member of $|M_{m_0}|$ does not contain any codimension $1$ component of $X\setminus U$, and thus $N'_U$ extends to a reduced divisor $N'$ on $X$ such that $N'\in|M_{m_0}|$. 
 \end{proof}

\end{proof}~\\

\begin{lemma}
\label{lem-klt-check-boundary}
Let $X$ be a normal $\mathbb{Q}$-factorial compact  analytic variety, $|L|_\mbQ$ a $\mbQ$-linear system on $X$, and $D$ is a $\mbQ$-divisor on $X$. Then the following assertions are all equivalent:
\begin{enumerate}
\item $(X,D+|L|_{\mathbb{Q}})$ is sub-klt;
\item For any projective bimeromorphic morphism $f:X'\to X$ with $X'$ normal, $ D'+|L'|_{\mathbb{Q}}$  is a pure sub-boundary, where $D'+|L'|_{\mathbb{Q}}$ is the log pullback of $D+|L|_{\mathbb{Q}}$, i.e., $L'=f^*L$ and $K_{X'}+D'=f^*(K_X+D)$. 

\end{enumerate}
\end{lemma}

\begin{proof}
	It follows easily from the defintion above.

\end{proof}

\begin{lemma}
\label{lem-subklt-system-implies-subklt}
Let $X$ be a normal $\mathbb{Q}$-factorial compact analytic surface, $|L|_\mbQ$ a $\mbQ$-linear system on $X$, and $D$ is a $\mbQ$-divisor on $X$. If $(X, D+|L|_\mbQ)$ is sub-klt, then there is an element $N\in |L|_{\mathbb{Q}}$ such that $(X,D+N)$ is sub-klt.
\end{lemma}

\begin{proof}
Let $f: X' \to X$ be a resolution of singularities of $X$  such that $\Ex(f)$, $f^*D, f^*L$ and  $f^{-1}\Bs(|L|_{\mathbb{Q}})$ are all simple normal  crossing divisors and they intersect each other transversally. Let $D'+|L'|_\mbQ$ be the log pullback of $D+|L|_\mbQ$, i.e.,
\[
        K_{X'}+D'=f^*(K_X+D)\quad\mbox{and}\quad L'=f^*L.
\]
Then by Lemma \ref{lem-klt-check-boundary}, $D'+|L'|_\mbQ$ is a pure sub-boundary. We need to show that there is some $N'\in |L'|_{\mathbb{Q}}$ such that $D'+N'$ is a pure sub-boundary and it also has simple normal crossing support.\\
By a similar argument as in the proof of Lemma \ref{lem-sub-boundary-system-implies-sub-boundary} we see that there is a sufficiently large and divisible positive integer $m$ such that $D'+\frac{1}{m}F'$ is a pure sub-boundary, where $|mL|=|M'|+F'$ is the decomposition of $mL$ into movable and fixed parts. Indeed, since $m$ is sufficiently large and divisible and $X$ is compact, it follows that $\Bs|mL|=\Bs|L|_\mbQ$. Then from the construction of $D'$ and $L'$ above it follows that $D'+\frac{1}{m}F'$ is a simple normal crossing divisor. Note that since $(X, D+|L|_{\mbQ})$ is sub-klt and $m$ is sufficiently large, it follows from the arguments in the proof of Lemma \ref{lem-sub-boundary-system-implies-sub-boundary} that $D'+\frac{1}{m}F'$ is a pure sub-boundary. Now $M'$ is a  movable Cartier divisor on the compact complex surface $X'$. In particular,   $\Bs|M'|$ is a finite set. Thus by \cite[Corollary 1.14]{Fuj83}, $\Bs|M'|=\emptyset$, i.e., $|M'|$ is base point free. Then by a Bertini type theorem \cite[Corollary II.7]{Man82}, there is a smooth divisor $N''\in|M'|$ such that the support of $N''$ does not contain any component of $D+\frac{1}{m}F'$ and   $D+\frac{1}{m}N''+\frac{1}{m}F'$ has simple normal crossing support. Set $N':=\frac{1}{m}(N''+F')\in|L'|_\mbQ$; then clearly $D+N'$ is a pure sub-boundary.      

 \end{proof}

In the following two lemmas we compare $D+|L|_{\mathbb{Q}}$ with its log pullback by a generically finite morphism.

\begin{lemma}
\label{lem-finite-sub-boundary}
Let $X$ be a normal $\mathbb{Q}$-factorial compact analytic variety, $|L|_\mbQ$ a $\mbQ$-linear system on $X$, and $D$ is a $\mbQ$-divisor on $X$. Let $p:X'\to X$ be a proper finite morphism with $X'$ normal. Let $D'+|L'|_{\mathbb{Q}}$  be  the log pullback of $D+|L|_{\mathbb{Q}}$ under $p$.  Then $ D+|L|_{\mathbb{Q}}$ is a pure sub-boundary if and only if $ D'+|L'|_{\mathbb{Q}} $ is a pure sub-boundary.
\end{lemma}

\begin{proof}
Assume that $ D+|L|_{\mathbb{Q}}$ is a  pure sub-boundary. Then by Lemma \ref{lem-sub-boundary-system-implies-sub-boundary}, there is some $N\in |L|_{\mathbb{Q}}$ such that $D+N$ is a pure sub-boundary. We claim that $D'+N'$ is a pure sub-boundary, where $N'=p^*N$. Let $E'$ be a prime divisor on $X'$ and set $E:=p(E')$. We will compute the multiplicity of $D'+N'$ along $E'$. Assume that $p^*E$ has multiplicity $m$ along $E'$. Let $d$ and $n$  be the multiplicities of $D$ and $N$ respectively along $E$. Then  we have \[\mathrm{mult}_{E}\, (D+N) -1  = d+n-1<0.\] 
From the ramification formula we have 
\[\mathrm{mult}_{E'}\, N' = nm, \  \mathrm{mult}_{E'}\, D' = dm-(m-1). \] 
 
 Thus
\begin{eqnarray*}
 \mathrm{mult}_{E'}\, (D'+N') -1 &=&  dm-(m-1) + nm  -1\\
  &=&   m(d+n-1)\\ 
  &=&   m(\mathrm{mult}_{E}\, (D+N) -1)  \\
  &<& 0.
\end{eqnarray*}
This shows that $ D'+N'$ is a  pure sub-boundary.\\

 Now we assume that $ D'+|L'|_{\mathbb{Q}}$ is a  pure sub-boundary. Using the result we just proved above we may replace $X'$ by the Galois closure of the morphism $p$ (see Lemma \ref{lem:galois-closure}), and thus assume that $p:X'\to X$ is a Galois morphism with Galois group $G$. Now since $D'+|L'|$ is a pure sub-boundary, by Lemma \ref{lem-sub-boundary-system-implies-sub-boundary}
 there is an element $B'\in |L'|_{\mathbb{Q}}$ such that $D'+B'$ is a pure sub-boundary. Now set $N':=\frac{1}{\#G}\sum_{g\in G} g(B')$.  Then $N'\in |L'|_{\mathbb{Q}}$ and $D'+N'$ is also a pure sub-boundary. Moreover, since $N'$ is $G$-invariant,  there is an element $N\in |L|_{\mathbb{Q}}$  such that $N'=p^* N$.   We claim that $D+N$ is a pure  sub-boundary. Let $E$ be a prime Weil divisor on $X$ and let $E'$ be an irreducible component of $p^*E$.   Then as in the computation above we have 
 \[ \mathrm{mult}_{E}\, (D+N) -1  = \frac{1}{m} (\mathrm{mult}_{E'}\, (D'+N') -1) <0.\]
   Therefore $D+N$ is a pure sub-boundary, and so is $ D+|L|_{\mathbb{Q}}$.
\end{proof}

 \begin{lemma}
\label{lem-gen-finite-sub-boundary}
Let $X$ be a normal $\mathbb{Q}$-factorial compact analytic variety, $|L|_{\mathbb{Q}}$ a $\mathbb{Q}$-linear system on $X$, and $D$ is a $\mbQ$-divisor on $X$. Let $p:X'\to X$ be a  generically finite proper morphism from a normal variety $X'$.  Let $D'+|L'|_{\mathbb{Q}}$  be  the log pullback of $D+|L|_{\mathbb{Q}}$ under $p$.  Then $ D+|L|_{\mathbb{Q}}$ is a  pure sub-boundary if   $ D'+|L'|_{\mathbb{Q}} $ is a pure sub-boundary.
\end{lemma}

\begin{proof}
By Stein factorization, we only need to prove the lemma in the following two cases: either $p$ is a proper bimeromorphic morphism, or $p$ is a finite proper morphism. The case when $p$ is bimeromorphic is trivial.  The case when $\pi$ is finite follows from Lemma \ref{lem-finite-sub-boundary}. 
\end{proof}

 \part{Construction of LMMP for pseudo-effective pair}

 \section{Non-vanishing theorem}

 The objective of this section is to prove the following non-vanishing theorem.

 \begin{theorem}
 \label{thm-non-vanishing}
 Let $(X, \Delta)$ be a $\mbQ$-factorial compact K\"ahler $3$-fold dlt pair. Assume that $X$ is non-algebraic and that the base of the MRC fibration for $X$ has dimension $2$.  Then the following assertions are equivalent:
 \begin{enumerate}
 \item $K_X+\D$ is $\mathbb{Q}$-effective;
 \item $K_X+\D$ is pseudo-effective;
 \item $(K_X+\D)\cdot F\>0$ for general fibers $F$ of the MRC fibration of $X$.
 \end{enumerate} 
 \end{theorem}

\begin{remark}\label{rmk:non-vanishing}
        Note that if $(K_X+\D)\cdot F>0$ for a general fiber $F$ of the MRC fibration, then $\D$ has at least one component which dominates the base of the MRC fibration. Reducing the coefficients of these components appropriately we may assume that $(K_X+\D)\cdot F=0$. Thus in order to prove the above theorem it is enough to show that, if $(K_X+\D)\cdot F=0$, then $K_X+\D$ is $\mathbb{Q}$-effective.
\end{remark}

Using Lemma \ref{l-proj} we will reduce the general non-vanishing Theorem \ref{thm-non-vanishing-general-setting} to Theorem \ref{thm-non-vanishing}.

 \begin{theorem}
 \label{thm-non-vanishing-general-setting}
 Let $(X, \Delta)$ be a $\mbQ$-factorial compact K\"ahler $3$-fold dlt pair. If $K_X+\D$ is pseudo-effective, then it is $\mathbb{Q}$-effective.
 \end{theorem}

 \begin{proof}
 If $X$ is projective, then the theorem follows from \cite[Theorem 1.1]{KMM94} after running a $(K_X+\Delta)$-MMP. Hence we may assume that $X$ is not algebraic.
 
   If $X$ is not uniruled, then first using Corollary \ref{cor:exist-terminal-dlt-modification} we may assume that $X$ has terminal singularity. Then the non-vanishing follows from \cite[Theorem 0.3]{DP03} after running a $K_X$-MMP as in \cite[Theorem 1.1]{HP16}. It remains to treat the case when $X$ is uniruled and non-algebraic. In this case the base of the MRC fibration of $X$ has dimension $2$, by Lemma \ref{l-proj}, and thus our result follows from Theorem \ref{thm-non-vanishing}.
 \end{proof}~\\

 \subsection{Case of Mori fiber space}

 The aim of this subsection is to prove the Theorem \ref{thm-base-boundary-effective} below, which is a generalization of \cite[Corollary 3.8]{KMM94}. First we state a combinatorial result from \cite{KMM94} which will be very useful to us in this section.\\
 
 Let $I_1, I_2,\ldots, I_p$ be a partition of the set $\{1, 2,\ldots, k\}$. Let $a_1, a_2,\ldots, a_k$ be a sequence of rational numbers such that $0\<a_i\<1$ and $\sum_{i=1}^ka_i=2$. Set $\alpha_q=\sum_{i\in I_q}a_i$.
 \begin{lemma}\cite[Lemma 3.2]{KMM94}\label{lem:combinatorix}
        With the notations and hypothesis as above assume that $\alpha_q$ satisfies $0\<\alpha_q\<1$ for all $q\in\{1,\ldots, p\}$. Then the system of equations
\[\sum_{j=1}^k b_{ij}=a_i\quad \mbox{for all } i=1, 2\ldots, k\] 
has a non-negative rational solution $b_{ij}\>0$ such that $b_{ij}=b_{ji}$ for all $i, j$, and $b_{ij}=0$ if $i$ and $j$ belong to the same indexing set $I_q$.
 \end{lemma}

 \begin{theorem}
 \label{thm-base-boundary-effective}
 Let $(X, \Delta)$ be a $\mbQ$-factorial compact K\"ahler $3$-fold dlt pair. Assume that $X$ is non-algebraic and that the base of the MRC fibration of $X$ has dimension $2$.
Let $f:X\to Y$ be a Mori fibration of relative dimension $1$ such that $K_X+\D\sim_{\mbQ, f} 0$. Then there is an effective $\mathbb{Q}$-divisor $B\>0$ on $Y$ such that  $K_X+\D \sim_{\mathbb{Q}} f^*(K_Y+B)$. Moreover, if $(X,\D)$ is klt, then $(Y,B)$ is klt.
 \end{theorem}

We first deal with the case where every component of $\Delta$ is bimeromorphic to the base $Y$. \\

The following elementary lemmas will be useful in our proof.

\begin{lemma}\label{lem:canonical-pushforward}
	Let $f:Y\to X$ be a proper bimeromorphic morphism between compact analytic varieties and $X$ is normal.  Let $\omega_Y$ be the canonical sheaf of $Y$ such that $\omega_Y^{[m]}$ is a line bundle for some $m\in\mbN$. Let $g:Z\to Y$ be a resolution of singularities of $Y$ and $\omega_Z^m\cong g^*\omega^{[m]}_Y(E)$ for some Weil divisor $E$ on $Z$. Then $E$ is $(f\circ g)$-exceptional.
\end{lemma}
 
 \begin{proof}
 	Let $\nu:\overline{Y}\to Y$ be the normalization of $Y$. Then $\omega^{[m]}_{\overline{Y}}\cong\nu^*\omega^{[m]}_Y(B)$ for some effective Weil divisor $B$ on $\overline{Y}$. The divisor $B$ is determined by the conductor ideal of the normalization. Since $f:Y\to X$ is bimeromorphic and $X$ is normal, $\Supp (\nu_*B)\subset\Ex(f)$. Now since the resolution $g:Z\to Y$ factors through the normalization $\overline{Y}$ and $\omega_Z^m\cong g^*\omega^{[m]}_Y(E)$, it follows that the components of $E$ are either exceptional for the induced morphism $Z\to\overline{Y}$ or they are strict transform of the components of the conductor divisor $B$. In particular, we have $\Supp E\subset \Ex(f\circ g)$.

 \end{proof}~\\

 In the next two lemmas we will deal with case when the horizontal components of $\Delta$ are bimeromorphic to the base of the MRC fibration. 
 
 \begin{lemma}\label{lem:vertical-pullback}
 	Let $(X, \Delta)$ be a $\mbQ$-factorial compact K\"ahler $3$-fold dlt pair. Let $f:X\to Y$ be a Mori fiber space contracting a $(K_X+\Delta)$-negative extremal ray of $\NA(X)$. Let $D$ be a $\mbQ$-divisor on $X$ such that $D$ is vertical over $Y$, i.e. $f(\Supp D)$ does not dominate $Y$. Then there is a $\mbQ$-Cartier $\mbQ$-divisor $B$ on $Y$ such that $D=f^*B$.
 \end{lemma}
 
 \begin{proof}
 	Let $m$ be a positive integer such that $mD$ is a Cartier divisor. Write $D=D^+-D^-$, where $D^+$ and $D^-$ effective $\mbQ$-divisors without common components. Note that by hypothesis $D^+$ and $D^-$ are both vertical over $Y$. Thus $mD^+\cdot C=0$ for curves $C\subset X$ contained in the fibers of $f$. Then by \cite[Proposition 3.1(6)]{CHP16} there exists a line bundle $\msL$ on $Y$ such that $\mcO_X(mD^+)\cong f^*\msL$. Now since $f_*\mcO_X=\mcO_Y$, from the projection formula it follows that $\{0\}\neq H^0(X, \mcO_X(mD^+))=H^0(X, f^*\msL)=H^0(Y, \msL)$. Therefore there exists a non-zero section $s\in H^0(Y, \msL)$ such that the divisor of zeros of $f^*s$ is the divisor $mD^+$, i.e. $mD^+=\div(f^*s)_0$. Set $L:=\div(s)_0$. Then $L$ is an effective Cartier divisor on $Y$ such that $mD^+=f^*L$. Similarly, there exists an effective Cartier divisor $M$ on $Y$ such that $mD^-=f^*M$. Define $B:=\frac{1}{m}(L-M)$. Thus we have $D=(D^+-D^-)=f^*B$.

 \end{proof}~\\

 \begin{lemma}
 \label{lem-all-horizontal-bimeromorphic}
 Let $(X, \Delta)$ be a  $\mathbb{Q}$-factorial compact K\"ahler $3$-fold pair with canonical singularities such that $X$ has terminal singularities.  Let $f:X\to Y$ be a $K_X$-Mori fibration of relative dimension $1$ such that $K_X+\D\equiv_f 0$. We assume that every component of $\D$ is bimeromorphic to $Y$ \textit{via} $f$.   Then there is an effective $\mbQ$-divisor $L\>0$ on $Y$ such that $K_X+\D \sim_{\mathbb{Q}} f^*(K_Y+L)$ and that the base locus of the $\mbQ$-linear system $|L|_{\mathbb{Q}}$ has no divisorial component.
 \end{lemma}

 \begin{proof}
Write $\D=\sum_{i=1}^k a_i\D_i$ such that $\D_i$'s are the irreducible components of $\D$ and $a_i\in (0, 1]$ for all $i$. Since $f$ is a $K_X$-Mori fiber space of relative dimension $1$, the general fibers of $f$ are isomorphic to $\mbP^1$; in particular, from $K_X+\Delta\num_f 0$ we obtain that $\sum_{i=1}^k a_i =2$. Then by Lemma \ref{lem:combinatorix} there exist nonnegative rational numbers $b_{ij}\>0$ with $b_{ij}=b_{ji}$ for $1\leq i,j \leq k$ such that 
\[  a_i=\sum_{j;\; j\neq i}b_{ij}.\] Then we have
 \begin{equation}\label{eqn:cbf}
        \sum_{i, j;\; i\neq j} \frac{1}{2} b_{ij}(K_{X/Y}+\D_i+\D_j) =  K_{X/Y}+\sum_{i=1}^k a_i\D_i.
 \end{equation}
 
Set $B_{ij}:=f_*(\D_i \cdot \D_j)$ for any $i \neq j$, and let $L=\frac{1}{2}\sum_{i\neq j}  b_{ij}B_{ij}\>0$. We make the following claim.
\begin{claim}\label{clm:cbf}
        The effective $\mbQ$-divisor $L\>0$ on $Y$ satisfies $(K_{X/Y}+\Delta)\sim_\mbQ f^*L$. 
\end{claim}

\begin{proof}[Proof of Claim \ref{clm:cbf}]
Let $F$ be a general fiber of $f$; then $F\cong\mbP^1$. Since $\Delta_i$'s are bimeromorphic to $Y$ via $f$, we have $(K_{X/Y}+\Delta_i+\Delta_j)\cdot F=0$. Now since $f$ is a contraction of a $K_X$-negative extremal ray, by \cite[Proposition 3.1(6)]{CHP16} there exists a positive integer $m>0$ and a line bundle $\msL_{ij}$ on $Y$ such that 
\begin{equation}\label{eqn:partial-pullback}
	\mcO_X(m(K_{X/Y}+\Delta_i+\Delta_j))\cong f^*\msL_{ij}
\end{equation}
 is an isomorphism of line bundles.\\
Now since $X$ has terminal singularities, $X$ is smooth outside a set of finitely many points. Thus by adjunction we have $(K_X+\Delta_i+\Delta_j)|_{\Delta_i}=K_{\Delta_i}+\Delta_j|_{\Delta_i}$.\\
Recall that $g=f|_{\Delta_i}:\Delta_i\to Y$ is bimeromorphic according to our hypothesis. Now observe that if $K_{\Delta_i}$ and $K_Y$ were given by Weil divisors, then we would 
have $g_*{K_{\Delta_i}}=K_Y$. This would have made our proof significantly simpler by simply restricting both sides of \eqref{eqn:partial-pullback} onto $\Delta_i$, and then pushing forward by $f|_{\Delta_i}$ would show that $\msL_{ij}$ is given by a Cartier divisor on $Y$. However, the canonical bundle of an analytic variety only exists as a sheaf and it may not correspond to any Weil divisor in general.\\
 To remedy this problem we pullback everything to a resolution of $\Delta_i$. Let $h:Z\to \Delta_i$ be a resolution of singularities of $\Delta_i$. Then by Lemma \ref{lem:canonical-pushforward} there exists a $(g\circ h)$-exceptional Weil divisor $E$ on $Z$ such that $h^*\omega_{\D_i}^{[m]}\cong\omega_Z^m(E)$. Moreover, since $Y$ is $\mbQ$-factorial, by our defintion of $\mbQ$-factoriality, $\omega_Y^{[m]}$ is a line bundle for $m\in\mbN$ sufficiently divisible. Thus there exists a $(g\circ h)$-exceptional divisor $F$ on $Z$ such that $\omega_Z^m\cong(g\circ h)^*\omega_Y^{[m]}(F)$. Now restricting both sides of \eqref{eqn:partial-pullback} to $\Delta_i$ and then pulling back to $Z$ we get
\begin{equation}\label{eqn:on-Z}
	h^*(\mcO_X(m(K_{X/Y}+\Delta_i+\Delta_j))|_{\Delta_i})\cong h^*((f^*\msL_{ij})|_{\Delta_i})\cong (g\circ h)^*\msL_{ij}.
\end{equation}
Now
\begin{align*}
    h^*(\mcO_X(m(K_{X/Y}+\Delta_i+\Delta_j))|_{\Delta_i}) & \cong\omega^m_Z(E)\otimes (g\circ h)^*\omega^{[-m]}_Y(mh^*(\Delta_j|_{\Delta_i}))\\
														  & \cong (g\circ h)^*\omega^{[m]}_Y(E+F)\otimes (g\circ h)^*\omega^{[-m]}_Y(mh^*(\Delta_j|_{\Delta_i}))\\
														  & \cong \mcO_Z(E+F+mh^*(\Delta_j|_{\Delta_i})).
\end{align*}
 Therefore from \eqref{eqn:on-Z} it follows that the line bundle $(g\circ h)^*\msL_{ij}$ is given by a Cartier divisor on $Z$. Since $g\circ h$ is bimeromorphic, it follows that $\msL_{ij}$ corresponds to a Cartier divisor, say $M_{ij}$ on $Y$. Let $L_{ij}:=\frac{1}{m}M_{ij}$; then from \eqref{eqn:on-Z} we have
 \begin{equation}\label{eqn:divisor-form}
 	\frac{1}{m}(E+F)+h^*(\Delta_j|_{\Delta_i})\sim_{\mbQ} (g\circ h)^*L_{ij}.
 \end{equation}
 Now according to Lemma \ref{lem:canonical-pushforward}, $E$, $F$ and the conductor divisor of the normalization of $\Delta_i$ are all exceptional over $Y$. Thus pushing forward both sudes of \eqref{eqn:divisor-form} we get $L_{ij}\sim_{\mbQ} (g\circ h)_*(h^*(\Delta_j)|_{\Delta_i})=f_*(\Delta_i\cdot \Delta_j)=B_{ij}$. Then from \eqref{eqn:cbf} we have $K_{X/Y}+\Delta\sim_\mbQ f^*L$.         
\end{proof}

In the following we will show that $L$ satisfies the other two properties.\\ 
 Let $E_1,...,E_s$ be the distinct irreducible components of $L$. We will show  that the base locus of $|L|_{\mathbb{Q}}$ does not have a divisorial component. To that end it is enough to show that for each $p\in\{1,\ldots,s\}$, there exists an effective $\mbQ$-divisor $L'\>0$ on $Y$ such that $L'\sim_\mbQ L$ and $E_p$ is not contained in the support of $L'$.
 We fix some $p\in\{1, 2,\ldots, s\}$, and let $C_1,...,C_q$ be the curves in $X$ contained in the support of the $1$-cycles $\Delta_i\cdot\Delta_j$ for all $i\neq j\in\{1,\ldots, k\}$ and dominate $E_p$, i.e., $f(C_i)=E_p$ for all $i=1,\ldots, q$. Now for each $r\in\{1,\ldots, q\}$ we define 
 \[I_r:=\{i\ |\  C_r \subseteq \Supp\D_i\}.\] 
Since $\Delta_i\to Y$ is bimeromorphic for each $i$,  the sets $I_1,\ldots, I_q$ form a partition of the set $\{1,\ldots, k\}$ (recall that $\Delta=\sum_{i}^ka_i\Delta_i$).\\
 Now since $X$ has  terminal singularities, $X$ is smooth in a neighborhood of the general points of each $C_r$. Moreover, since $(X,\D)$ is canonical, by passing to a log resolution of $(X, \Delta)$ which extracts an exceptional divisor dominate $C_r$, it is easy to see that $\sum_{i\in I_r} a_i\leqslant 1$. Then by Lemma \ref{lem:combinatorix} there exist non-negative rational numbers $c_{ij}\>0$ such that $c_{ij}=c_{ji}$ for all $i, j$, and that $c_{ij}=0$  whenever $i,j$ belong to the same indexing set $I_r$. 
Following a similar construction as $L$ given above we define $L':=\frac{1}{2}\sum_{i\neq j}c_{ij}f_*(\Delta_i\cdot\Delta_j)$. Then it is clear that $E_p$ is not contained in the support of $L'$. Finally, from the construction of $L$ and $L'$ as above it also follows that $0\<L'\sim_\mbQ L$.
 \end{proof}~\\

 \begin{lemma}
 \label{lem-horizontal-bimeromorphic}
 Let $(X, \Delta)$ be a $\mathbb{Q}$-factorial compact K\"ahler $3$-fold sub-lc pair such that $X$ has  terminal singularities. Let $f:X\to Y$ be a $K_X$-Mori fiber space of relative dimension $1$ such that $K_X+\D\equiv_f 0$.  We can decompose $\D=\D_{\hor}+\D_{\ver}$ into horizontal and vertical parts over $Y$, and assume that $\Delta_{\hor}$ is an effective $\mbQ$-divisor. Assume that $(X, \D_{\hor})$ is a canonical pair and that all components of $\D_{\hor}$ are bimeromorphic to $Y$ \textit{via} $f$.  Then there exist $\mbQ$-divisors $D$ and $L$ on $Y$ such that
 \begin{enumerate}
 \item $f^*D=\D_{\ver}$,
 \item $L$ is effective and $\Bs|L|_{\mbQ}$ has no divisorial component,
 \item $K_X+\D\sim_{\mathbb{Q}}  f^*(K_Y+D+L)$, and
 \item $D+|L|_{\mathbb{Q}}$ is a pure sub-boundary if $(X,\D)$ is  sub-klt.
 \end{enumerate} 
 \end{lemma}

 \begin{proof}
 Let $F$ be a general fiber of $f$. Then $\Delta_{\ver}\cdot F=0$. Since $f$ is the contraction of a $K_X$-negative extremal ray, by Lemma \ref{lem:vertical-pullback} there exists a $\mbQ$-Cartier divisor $D$ on $Y$ such that $\Delta_{\ver}=f^*D$. By Lemma \ref{lem-all-horizontal-bimeromorphic}, there is an effective $\mbQ$-divisor $L\>0$ on $Y$ such that $K_X+\D_{\hor}\sim_{\mathbb{Q}}  f^*(K_Y+L)$ and $\Bs|L|_{\mathbb{Q}}$ has no divisorial component.\\
 Now if $(X, \Delta)$ is sub-klt, then $\Delta_{\ver}$ is a pure sub-boundary, and thus so is $D$, since $\Delta_{\ver}=f^*D$. Finally, since $\Bs|L|_{\mathbb{Q}}$ does not have a divisorial component, it follows that $D+|L|_{\mathbb{Q}}$ is a pure sub-boundary.
 \end{proof}~\\

 We will now deal with the case where the components of $\D_{\hor}$ are not all bimeromorphic to the base $Y$. We will use the following covering trick.

 \begin{lemma}[Covering Trick]
 \label{lem-Stein-factorization-base-change}
 Let $f:X\to Y$ be a proper fibration of relative dimension $1$ between normal compact analytic varieties. Let $S$ be a prime Weil divisor on $X$ such that it is horizontal over $Y$. Then there is a commutative diagram 

 \centerline{
 \xymatrix{
 X'\ar[d]_{f'} \ar[r]^p  & X \ar[d]^{f} \\
  Y'\ar[r]_{\pi} & Y
 }
 }
 \noindent such that 
 \begin{enumerate}
 \item $Y'$ is normal and $\pi$ is a proper finite morphism,
 \item $X'$ is the normalization of the main component of $X\times_Y Y'$ which dominates $Y'$,
 \item $p$ is a proper finite morphism such that it is \'etale over a dense Zariski open subset of $X$, 
 \item $S'= p^*S$ is a $\mbZ$-Weil divisor, and every component of $S'$ is bimeromorphic to $Y'$ \textit{via} $f'$.
 \end{enumerate}
 \end{lemma}

 \begin{proof}
Note that the induced morphism $f|_S:S\to Y$ is generically finite.  Let $S\to Y''\to Y$ be the Stein factorization of $f|_S:S\to Y$, and $\pi:Y'\to Y''$ the normalization of $Y''$. Then $S\to Y''$ is a bimeromorpic morphism. 
  Let $X'$ be the normalization of the main component of $X\times_Y Y'$ dominating $Y'$, and $f':X'\to Y'$ and $p:X'\to X$ are the projections. Note that since $Y'\to Y$ is \'etale over a dense Zariski open subset of $Y$ and \'etale morphisms are stable under base change, it follows that $p$ is \'etale over a dense Zariski open subset of $X$.\\
  Now let $S':=p^*S$ be defined as in Lemma \ref{lem:finite-pullback}, and $S''$ is the main component of $S\times_{Y''}Y'$ dominating $Y'$. Then $S''\to Y'$ is bimeromorphic. The morphisms $S''\to S\to X\to Y$ and $S''\to Y'\to Y$ induce a morphism $S''\to X'$. Note that the image of $S''$ on $X'$ is an irreducible component of $S'$; we denote this image again by $S''$ in $X'$.
Then we can decompose $S'$ as $S'=kS''+T'$, where $k\in\mbZ^+$ and $S''$ is not contained in the support of $T'$. Thus $S'$ has a component $S''$ which is bimeromorphic to $Y'$. We can continue this process until all the components of $S'$ becomes bimeromorphic to $Y'$. It is easy to see that this process stops after a finitely many steps, since the degree of $T'$ with respect to the general fiber of $f':X'\to Y'$ is strictly smaller than the degree of $S$ with respect to the general fibers of $f:X\to Y$.
 \end{proof}~\\

 \begin{lemma}
 \label{lem-base-pure-sub-boundary}
 Let $(X,\D)$ be a $\mbQ$-factorial compact K\"ahler $3$-fold sub-lc pair such that $X$ has klt singularities. Assume that $X$ is non-algebraic and that the base of the MRC fibration of $X$ has dimension $2$. Let $f:X\to Y$ be a $K_X$-Mori fiber space of relative dimension $1$ such that $K_X+\D\sim_{\mbQ, f} 0$. We decompose $\D=\D_{\hor}+\D_{\ver}$, and assume that $\D_{\hor}$ is effective. Then there exist two divisors $L$ and $D$ on $Y$ such that
 \begin{enumerate}
 \item $f^*D=\D_{\ver}$,
 \item $|L|_\mbQ$ is non-empty, 
 \item $K_X+\D\sim_{\mathbb{Q}}  f^*(K_Y+D+L)$, and
 \item $D+|L|_{\mathbb{Q}}$ is a pure sub-boundary if $(X,\D)$ is  sub-klt.
 \end{enumerate} 
 \end{lemma}

 \begin{proof}
 Since $f$ is the contraction of a $K_X$-negative extremal ray and $X$ has $\mbQ$-factorial klt singularities, by \cite[Proposition 3.5(6)]{CHP16} there exists a  line bundle $\msL$ on $Y$ such that 
 \begin{equation}\label{eqn:on-domain}
 	\mcO_X(m(K_X+\D_{\hor}))\cong f^*(\mcO_X(mK_Y)\otimes\msL)
 \end{equation}  
 is an isomorphism of lines bundles for some $m\in\mbN$.\\
 Next by Lemma \ref{lem:vertical-pullback} there exists a $\mbQ$-Cartier divisor $D$ on $Y$ such that $\Delta_{\ver}=f^*D$. We will now show that $\msL\cong\mcO_X(mL)$ for some $\mbQ$-Cartier divisor $L$ such that the $\mbQ$-linear system $|L|_{\mbQ}$ is non-empty. To that end, first applying Lemma \ref{lem-Stein-factorization-base-change} to the components of $\Delta_{\hor}$ several times and then taking log resolution at the end we obtain a commutative diagram satisfying the following properties:

% \centerline{
\begin{equation}
\xymatrixcolsep{3pc}\xymatrixrowsep{3pc}\xymatrix{
 (Z,\Theta)\ar[d]_{g} \ar[r]^p  & (X,\D) \ar[d]^{f} \\
  Y'\ar[r]_{\pi} & Y
 }
 \end{equation}
% }
 %\noindent such that 
 \begin{enumerate}[label=(\roman*)]
 \item $\pi$ and $p$ are generically finite proper morphisms, and $g$ is a projective contraction,
 \item $Y'$ is smooth and $Z$ is a  desingularization of the normalization of the
 main component of $X\times_Y Y'$ dominating $Y'$,
 \item   $K_Z+\Theta \sim_{\mathbb{Q}} p^*(K_{X}+\D)$,
 \item  $(Z, \Theta)$ is a log smooth sub-lc pair such that $\Supp p^{-1}_*\Delta$ is a smooth closed subset of $Z$,
 \item if we decompose $\Theta=\Theta_{\hor}+\Theta_{\ver}$ into horizontal and vertical parts over $Y'$, then $\Theta_{\hor}$ is effective,
 \item every component of $\Theta_{\hor}$ is bimeromorphic to $Y'$ \textit{via} $g$, and
 \item  $(Z,\Theta_{\hor})$ has canonical singularities.
 \end{enumerate}
Note that (i) and (vi) follows from the repeated application of Lemma \ref{lem-Stein-factorization-base-change}; (iii) follows from (ii) and Lemma \ref{lem:ramification-formula}; (iv) holds by \cite[Corollary 2.32(2)]{KM98}.\\
 Now we explain (v) and (vii). Since the desingularization $Z$ is constructed via successive blow ups of a normal variety along smooth centers of codimension at least $2$, the exceptional divisors of $p$ are all vertical over $Y'$. Moreover, from Lemma \ref{lem-Stein-factorization-base-change} it follows that $p$ is \'etale over a dense Zariski open subset of $X$, in particular, the ramification divisors of $p$ are vertical over $Y'$. Therefore we have $\Theta_{\hor}=p^{-1}_*\Delta_{\hor}$, and hence $\Theta_{\hor}$ is an effective boundary divisor; this is (v). Then (vii) follows from \cite[Corollary 2.31(3)]{KM98} using the fact that $\Theta_{\hor}$ is a boundary divisor and $\Supp\Theta_{\hor}$ is smooth.\\

Now we run a $(K_Z+\Theta_{\hor})$-MMP over $Y'$ as in Proposition \ref{pro:relative-projective-mmp}, since $g:X'\to Y'$ is projective. Note that from our construction above it follows that $K_Z+\Theta_{\hor}+\Theta_{\ver}\sim_{\mbQ, g}0$. Thus this MMP yields a pair $(Z',\Theta'_{\hor})$ with $g': Z'\to Y'$ such that $K_{Z'}+\Theta'_{\hor}$ is $g'$-nef. We note that none of the components of $\Theta_{\hor}$ is contracted by this MMP. Consequently, $(Z',\Theta'_{\hor})$ has canonical  singularities, and $Z'$ has terminal singularities by Lemma \ref{lem:mmp-preserving-terminal-singularities}. Let $\Theta'_{\ver}$ be the bimeromorphic transformation of $\Theta_{\ver}$ by the induced bimeromorphic map $Z\bir Z'$. Then we have $K_{Z'}+\Theta'_{\hor}+\Theta'_{\ver}\sim_{\mbQ, g'}0$, and $(Z', \Theta'_{\hor}+\Theta'_{\ver})$ has sub-lc singularities, since every step of this MMP is $(K_Z+\Theta)$-trivial.\\
  
Now since $g':Z'\to Y'$ is a projective morphism, by Proposition \ref{pro:relative-projective-mmp} we can run a $K_{Z'}$-MMP over $Y'$ with the scaling of $\Theta'_{\hor}$. Note that the general fibers of $g'$ are isomorphic to $\mbP^1$ and $(K_{Z'}+\Theta'_{\hor})\cdot F=0$ for a general fiber $F$ of $g'$. Therefore this MMP terminates with a Mori fiber space, say $f': X'\to Y''$ such that $X'$ has $\mbQ$-factorial terminal singularities. Let $\Delta'_{\hor}$ and $\Delta'_{\ver}$ be the bimeromorphic transform of $\Theta'_{\hor}$ and $\Theta_{\ver}$, respectively, and $\Delta':=\D'_{\hor}+\D'_{\ver}$. From Lemma \ref{lem:directed-relative-mmp} it also follows that every step of this MMP is $(K_{Z'}+\Theta'_{\hor})$-trivial. In particular, $(K_{X'}+\Delta'_{\hor})\sim_{\mbQ, \tau\circ f'} 0$ and $(X', \Delta'_{\hor})$ has canonical singularities, where $\tau:Y''\to Y'$ is the induced bimeromorphic morphism. Moreover, we also have $(K_{X'}+\D')\sim_{\mbQ, \tau\circ f'}0$ and $(X', \Delta')$ is sub-lc.\\

\begin{equation}\label{eqn:mori-fiber-diagram}
        \xymatrixcolsep{3pc}\xymatrixrowsep{3pc}\xymatrix{
        (X', \Delta')\ar[d]_{f'} & (Z', \Theta')\ar@{-->}[l]\ar[dr]_{g'} & (Z, \Theta)\ar@{-->}[l]\ar[r]^p\ar[d]_g & (X, \Delta)\ar[d]^f\\
        Y''\ar[rr]^{\tau}&& Y'\ar[r]^{\pi} & Y 
        }
\end{equation}

 Now applying Lemma \ref{lem-horizontal-bimeromorphic} to $f':(X', \Delta')\to Y''$ we see that there exist $\mbQ$-Cartier divisors $L''$ and  $D''$ on $Y''$ such that   
 \begin{enumerate}[label={\alph*)}]
 \item $f'^*D''=\Delta'_{\ver}$,
 \item $L''$ is effective, and $\Bs|L''|_{\mathbb{Q}}$ does not have any divisorial component,
 \item $K_{X'}+\Delta_{\hor}'\sim_{\mathbb{Q}}  f'^*(K_{Y''}+L'')$, and 
 \item $D''+|L''|_{\mathbb{Q}}$ is a pure sub-boundary  if $(X',\Delta')$ is sub-klt.
 \end{enumerate} 
From our hypothesis it follows that $Y'$ is a smooth compact K\"ahler non-uniruled surface, and thus both of the line bundles $\omega_{Y'}^m$ and $\omega_{Y''}^{[m]}$ correspond to effective Cartier divisors for sufficiently divisible $m\in\mbN$.  Combining all of these with the Part (c) above it follows that $K_{Y''}+L''\sim_\mbQ \tau^*(K_{Y'}+L')$, where $L'=\tau_*L''$. Also, since $K_{X'}+\Delta'\sim_{\mbQ, \tau\circ f'}0$ and $K_{X'}+\Delta'_{\hor}\sim_{\mbQ, \tau\circ f'}0$, we have $\Delta'_{\ver}\sim_{\mbQ, \tau\circ f'}0$; in particular, from Part (a) above it follows that $D''=\tau^*D'$, where $D'=\tau_*D''$. Thus we have 
 \[\Delta'_{\ver}=(\tau\circ f')^*D', K_{X'}+\Delta'_{\hor}\sim_\mbQ (\tau\circ f')^*(K_{Y'}+L')\]
 \begin{center}
 	and
 \end{center}
 \[K_{X'}+\Delta'\sim_{\mbQ} (\tau\circ f')^*(K_{Y'}+L'+D').\]
  Finally, from Part (d) it follows that $D'+|L'|_\mbQ$ is a pure sub-boundary if $(X', \Delta')$ is sub-klt. Indeed, if $F'$ is a prime Weil divisor $Y'$ and $F'':=\tau^{-1}_*F'$, then there exists an effective $\mbQ$-Cartier divisor $M''\in|L''|_{\mbQ}$ such that $\mult_{F''}(D''+M'')<1$, since $D''+|L''|_{\mbQ}$ is a pure sub-boundary. Thus by pushing forward (as $\tau$ is bimeromorphic) we get that $\mult_{F'}(D'+M')<1$, where $M':=\tau_*M''\in|L'|$; hence 
  \begin{equation}\label{eqn:Y'-pure-sub-boundary}
  	 D'+|L'|_{\mbQ}\quad \mbox{is a pure sub-boundary.}
  \end{equation}
From \eqref{eqn:on-domain} we also have $\mcO_X(m(K_Z+\Theta))\cong(\pi\circ g)^*(\mcO_Y(m(K_Y+D))\otimes\msL)$ for sufficiently divisible $m\in\mbN$. Let $W$ be the normalization of the graph of the bimeromorphic contraction $\phi:Z\bir X'$ over $Y'$, and $q:W\to Z$ and $r:W\to X'$ are the corresponding projections. Then from our construction above and the negativity lemma it follows that $q^*(K_Z+\Theta)=r^*(K_{X'}+\Delta')$. In particular, we have
 \[
 	(\tau\circ f'\circ r)^*\mcO_{Y'}(m(K_{Y'}+D'+L'))\cong (\pi\circ g\circ q)^*(\mcO_Y(m(K_Y+D))\otimes\msL)
 \]
 for sufficiently divisible $m\in\mbN$.\\
 Pushing forward both side by $\tau\circ f'\circ r=g\circ p$ and then applying the projection formula we get 
\begin{equation}\label{eqn:model-comparison}
     \mcO_{Y'}(m(K_{Y'}+L'+D'))\cong\pi^*(\mcO_Y(m(K_Y+D))\otimes\msL).
\end{equation}
 Now since $Y$ and $Y'$ are both non-uniruled compact K\"ahler surface with klt singularities, it follows that $\mcO_Y(mK_Y)$ and $\mcO_{Y'}(mK_{Y'})$ are both given by Cartier divisors for sufficiently divisible $m\in\mbN$. Thus from \eqref{eqn:model-comparison} we see that $\msL$ correspond to a $\mbQ$-Cartier divisor $L$ on $Y$ such that $\msL\cong\mcO_Y(mL)$ for some sufficiently divisible $m$. Therefore from \eqref{eqn:model-comparison} we have
 \begin{equation}\label{eqn:divisor-form-comparison}
 	K_{Y'}+D'+L'\sim_{\mbQ} \pi^*(K_Y+L+D).
 \end{equation}
 Now our main goal is to show that $|L|_{\mbQ}$ is non-empty. To that end, first note that since $\pi:Y'\to Y$ is a generically finite morphism, we can write
 \begin{equation}\label{eqn:ramification-decomposition}
 K_{Y'/Y}=R_{\nex}+R_{\ex}, 	
 \end{equation} 
 where $R_{\ex}$ is exceptional over $Y$, $R_{\nex}$ is not exceptional over $Y$.\\
 Note that $R_{\nex}$ is effective. Then from \eqref{eqn:divisor-form-comparison} we can write 
 \begin{equation}\label{eqn:L-non-empty}
        \pi^*L \sim_{\mathbb{Q}} L'+E+F, 
 \end{equation}
 where $E=(D'-\pi^*D)_{\ex}+R_{\ex}$ and $F=(D'-\pi^*D)_{\nex}+R_{\nex}$.\\
 
 \begin{claim}\label{clm:effectivity}
        Then divisor $F=(D'-\pi^*D)_{\nex}+R_{\nex}$ is effective.
 \end{claim}
 Assuming the claim for the time being we will first show that $|L|_{\mbQ}$ is non-empty. Indeed, recall that $L''$ is an effective divisor such that the base-locus of $|L''|_\mbQ$ contains only finitely many points, and thus by \cite[Corollary 1.14]{Fuj83} $L''$ is semi-ample. Then from the projection formula it follows that $L'=\tau_*L''$ is algebraically nef on $Y'$, i.e. $L'\cdot C\>0$ for all curves $C\subset Y'$. Now if $F$ is effective, then  it is algebraically $\pi$-nef, i.e., $F\cdot C\>0$ for every $\pi$-exceptional curve $C\subset Y'$, since none of the components of $F$ are $\pi$-exceptional. From \eqref{eqn:L-non-empty}, we can write 
 \[ -E \sim_\mathbb{Q} L' + F -\pi^*L, \] 
 which is algebraically $\pi$-nef. 
 Since $E$ is $\pi$-exceptional, 
 by the negativity lemma it follows   that $E$ is effective. In particular, it again follows from \eqref{eqn:L-non-empty} that $|L|_\mbQ$ is non-empty. \\

\begin{proof}[Proof of Claim \ref{clm:effectivity}]

We will first show  the following fact, which will be useful for the proof, 
\begin{equation}\label{eqn:anti-effective}
        \Theta_{\ver}-g^*D'\<0.
\end{equation}
Recall the resolution $q:W\to Z$ and $r:W\to X'$ of the map $\phi:Z\bir X'$. We can write
\begin{equation}\label{eqn:discrepancy-comparison}
        r^*(K_{X'}+\Delta'_{\hor})=q^*(K_Z+\Theta_{\hor})+q^*(\Theta_{\ver}-g^*D').
\end{equation}
Recall that $\phi:Z\bir X'$ is a composition of a $(K_Z+\Theta_{\hor})$-MMP followed by a $K_{Z'}$-MMP. Since every step of the $K_{Z'}$ was $(K_{Z}+\Theta_{\hor})$-trivial, it follows that the difference $q^*(K_Z+\Theta_{\hor})-r^*(K_{X'}+\Delta'_{\hor})$ is an effective $\mbQ$-divisor. 
Thus we can deduce the inequality (\ref{eqn:anti-effective}).

Let $D'_1$ be a prime Weil divisor on $Y'$ which is not $\pi$-exceptional. We will show that $\mult_{D'_1}(F)\>0$. 
\begin{equation*}
\xymatrixcolsep{3pc}\xymatrixrowsep{3pc}\xymatrix{
 (Z,\Theta)\ar[d]_{g} \ar[r]^p  & (X,\D) \ar[d]^{f}    \\
  Y'\ar[r]_{\pi} & Y  
 }
 \end{equation*}
% }
To that end set $D_1:=\pi(D'_1)$ and let $\Delta_1 = \Supp(f^*D_1)$. Note that $\Delta_1$ is a prime divisor, since the  relative Picard number of $f$ is $1$. Now let $\Theta_1$ be a prime Weil divisor on $Z$ such that   $p(\Theta_1) = \Delta_1$. Let $m, n, m'$ and $n'$ be the coefficients of $D'_1$ in $\pi^*D_1$, $\Delta_1$ in $f^*D_1$, $\Theta_1$ in $p^*\Delta_1$ and $\Theta_1$ in $g^*D'_1$, respectively, i.e., $\mult_{D'_1}(\pi^*D_1)=m, \mult_{\Delta_1}(f^*D_1)=n, \mult_{\Theta_1}(p^*\Delta_1)=m'$ and $\mult_{\Theta_1}(g^*D'_1)=n'$.\\

Next, we will compute $\mult_{D'_1}(F)$ in terms of $m,n,m',n'$ and $\mult_{D_1}(D)$. 
We will also separate our discussion into two cases depending on whether $\mult_{D_1}(D)$ is $0$ or not.\\

\noindent
\textbf{Case I: } $\mult_{D_1}(D)\neq 0$, where $f^*D=\Delta_{\ver}$. In this case $\Delta_1$ is a component of $\Delta_{\ver}$. Let $\mult_{\Delta_1}(\Delta_{\ver})=\lambda$. Then $\mult_{D_1}(D)=\frac{\lambda}{n}$ and $\mult_{D'_1}(\pi^*D)=\frac{m\lambda}{n}$. 

Recall that by our construction the generically finite morphism $p:Z\to X$ is \'etale over a dense Zariski open subset of $X$. In particular, the ramification divisor of $p$ is vertical over $Y'$, and thus $\Theta_{\hor}=p^{-1}_*\Delta_{\hor}$ and $K_Z+\Theta_{\ver}\sim_{\mbQ} p^*(K_X+\Delta_{\ver})$. Therefore $\mult_{\Theta_1}(\Theta_{\ver})=m'\lambda-(m'-1)$. Moreover, from \eqref{eqn:anti-effective} it follows that $\mult_{D'_1}(D')\>\frac{1}{n'}\mult_{\Theta_1}(\Theta_{\ver})=\frac{m'\lambda-(m'-1)}{n'}$. Also note that $\mult_{D'_1}(R_{\nex})=m-1$. 
Therefore we have
\begingroup
\fontsize{10pt}{12pt}\selectfont
\begin{equation*}
        \begin{split}
                \mult_{D'_1}((D'-\pi^*D)_{\nex}+R_{\nex}) &\> \frac{m'\lambda-(m'-1)}{n'}-\frac{m\lambda}{n}+(m-1)\\
                                                                                                 &=\frac{m'\lambda-(m'-1)}{n'}-\frac{m'\lambda}{n'}+(m-1)\quad (\mbox{since } mn'=m'n)\\
                                                                                                 &=(m-1)-\frac{m'-1}{n'}\\
                                                                                                 &=\frac{1}{n'}(mn'-n'-m'+1)\\
                                                                                                 &=\frac{1}{n'}(dm'n'-n'-m'+1)\; (\mbox{since } m'=\frac{m}{d}, \mbox{ where } d=\gcd(m, n))\\
                                                                                                 &\>\frac{1}{n'}(m'-1)(n'-1)\>0.
        \end{split}
\end{equation*}
\endgroup

\noindent
\textbf{Case II: } $\mult_{D_1}(D)=0$. In this case $\Delta_1$ is not contained in the support of $\Delta_{\ver}$, however $\mult_{\Theta_1}(\Theta_{\ver})$ could still be non-zero if $\Theta_1$ is a ramification divisor of $p$. First assume that $\mult_{\Theta_1}(\Theta_{\ver})=0$. Then from \eqref{eqn:anti-effective} it follows that $\mult_{D'_1}(D')\>0$. Thus we have $\mult_{D'_1}((D'-\pi^*D)_{\nex}+R_{\nex})\>0$, since $\mult_{D'_1}(\pi^*D)=0$ and $R_{\nex}\>0$ is an effective divisor.\\

Now we consider the possibility that $\mult_{\Theta_1}(\Theta_{\ver})\neq 0$. In this case we have $\mult_{D'_1}(\pi^*D)=0, \mult_{D'_1}(R_{\nex})=m-1, \mult_{\Delta_1}(\Delta_{\ver})=0$ and $\mult_{\Theta_1}(\Theta_{\ver})=-(m'-1)$. From \eqref{eqn:anti-effective} we have $\mult_{D'_1}(D')\>-\frac{m'-1}{n'}$. Therefore we have
\begingroup
\fontsize{10pt}{12pt}\selectfont
\begin{equation*}
        \begin{split}
                \mult_{D'_1}((D'-\pi^*D)_{\nex}+R_{\nex}) &\> -\frac{m'-1}{n'}+(m-1)\\
                                                                                                  &=\frac{1}{n'}(mn'-n'-m'+1)\\
                                                                                                  &=\frac{1}{n'}(dm'n'-n'-m'+1)\;(\mbox{since } m'=\frac{m}{d}, \mbox{ where } d=\gcd(m, n))\\
                                                                                                  &\>\frac{1}{n'}(m'-1)(n'-1)\>0.
        \end{split}
\end{equation*}
\endgroup
        
\end{proof}  
\vspace{1cm} 

 It remains to show that $D+|L|_{\mathbb{Q}}$ is a pure sub-boundary  if $(X,\D)$ is sub-klt. By Lemma \ref{lem-gen-finite-sub-boundary} we only need to show that 
  \[ (\pi^*D-K_{Y'/Y}) + |\pi^*L|_{\mathbb{Q}}\] is a pure sub-boundary. 
  From \eqref{eqn:ramification-decomposition} and \eqref{eqn:L-non-empty} we see that 
  \begin{align*}
  	(\pi^*D-K_{Y'/Y}) &\sim_{\mbQ} \pi^*D-(R_{\ex}+R_{\nex})=D'-(E+F)\quad\mbox{and}\\
	\pi^*L &\sim_{\mbQ} L'+E+F,
  \end{align*}
 where $E=(D'-\pi^*D)_{\ex}+R_{\ex}$ and $F=(D'-\pi^*D)_{\nex}+R_{\nex}$ are both effective divisors as proved above.\\ 
 Therefore it is enough to show that
 \begin{equation}\label{eqn:Y-pure-sub-boundary}
 	D'-(E+F)+|L'+E+F|
 \end{equation}
  is a pure sub-boundary.\\
  It is easy to see that \eqref{eqn:Y-pure-sub-boundary} is a pure sub-boundary if $D'+|L'|$ is a pure sub-boundary, as $E+F$ is effective. But the fact that $D'+|L'|$ is a pure sub-boundary is already proved in \eqref{eqn:Y'-pure-sub-boundary}, so we are done.
  
 \end{proof}~\\

 We can now prove  Theorem \ref{thm-base-boundary-effective}. We note that a in the projective a much stronger result is proved by Ambro using different techniques, see \cite{Amb05}.

 \begin{proof}[{Proof of Theorem \ref{thm-base-boundary-effective}}]
  Let $D$ and $L$ be the two divisors in $Y$ given by Lemma \ref{lem-base-pure-sub-boundary}.  Since $\D$ is effective, $D$ is also effective.  Hence we can take  $B$ as any element in  $ D+|L|_{\mathbb{Q}}$.   It remains to show that $(Y,B) $ can be chosen to be klt if $(X,\D)$ is klt.  By Lemma \ref{lem-subklt-system-implies-subklt} it is enough to  show that $(Y,  D+|L|_{\mathbb{Q}})$  is sub-klt.  We will check this using Lemma \ref{lem-klt-check-boundary} and \ref{lem-gen-finite-sub-boundary}.  Let $\bar{\pi}:\bar{Y}\to Y$ be any proper bimeromorphic morphism with $\bar{Y}$ smooth. Let $\bar{X}$ be the resolution of singularities of the unique component of $X\times_Y\bar{Y}$ dominating $\bar{Y}$ so that $\bar{p}:\bar{X}\to X$ is a log resolution of the pair $(X, \Delta)$. Let $\bar{f}:\bar{X}\to\bar{Y}$ be the induced morphism and $K_{\bar{X}}+\bar{\D}=\bar{p}^*(K_X+\D)$. Then clearly $\bar{\D}_{\hor}=\bar{p}^{-1}_*\D_{\hor}$ and $K_{\bar{X}}+\bar{\D}_{\ver}=\bar{p}^*(K_X+\D_{\ver})$. Then as in the proof of Lemma \ref{lem-base-pure-sub-boundary} we can construct a diagram like \eqref{eqn:mori-fiber-diagram} such that $\bar{f}:\bar{X}\to\bar{Y}$ sits in between the morphism $g:Z\to Y'$ and $f:X\to Y$ and $Z$ is obtained from $\bar{X}$ via base-change by the components of $\bar{\Delta}_{\hor}$ followed a resolution of singularities.
  \begin{equation}\label{eqn:mori-fiber-diagram-modifieds}
          \xymatrixcolsep{3pc}\xymatrixrowsep{3pc}\xymatrix{
          (X', \Delta')\ar[d]_{f'} & (Z', \Theta')\ar@{-->}[l]\ar[dr]_{g'} & (Z, \Theta)\ar@{-->}[l]\ar@/^1pc/[rr]^p\ar[d]_g\ar[r] & (\bar{X}, \bar{\Delta})\ar[r]_{\bar{p}}\ar[d]_{\bar{f}} & (X, \Delta)\ar[d]_f\\
          Y''\ar[rr]^{\tau}&& Y'\ar@/_1pc/[rr]^{\pi}\ar[r]^\alpha & \bar{Y}\ar[r]^{\bar{\pi}} & Y 
          }
  \end{equation}

   Let $\alpha:Y'\to \bar{Y}$ be generically finite morphism constructed in the process above and $\pi=\bar{\pi}\circ\alpha:Y'\to Y$ is the composite morphism. Then as in the proof of Lemma \ref{lem-base-pure-sub-boundary} we know that 
  \begin{equation}\label{eqn:klt-pure-sub-boundary}
  	(\pi^*D-K_{Y'/Y})+|\pi^*L|\quad\mbox{is a pure sub-boundary}.
  \end{equation}
 Recall that we  want to show that $(Y, D+|L|)$ sub-klt, and by Lemma \ref{lem-klt-check-boundary} it is equivalent to show that the log pullback of $D+|L|$ onto $\bar{Y}$ defined as
 \begin{equation}\label{eqn:klt-log-pullback}
 	(\bar{\pi}^*D-K_{\bar{Y}/Y})+|\bar{\pi}^*L|
 \end{equation}
is a pure sub-boundary.\\
Now since $\alpha:Y'\to \bar{Y}$ is a generically proper morphism, by Lemma \ref{lem-gen-finite-sub-boundary} it is enough to show that the log pullback of \eqref{eqn:klt-log-pullback} onto $Y'$ defined as
\[
	(\alpha^*(\bar{\pi}^*D-K_{\bar{Y}/Y})-K_{Y'/\bar{Y}})+|\alpha^*\bar{\pi}^*L|
\]
is a pure sub-boundary.\\
Note that the above expression is same as $(\pi^*D-K_{Y'/Y})+|\pi^*L|$ and this is a pure sub-boundary by \eqref{eqn:klt-pure-sub-boundary}, so we are done.\\

 \end{proof}~\\

\subsection{Proof of the non-vanishing theorem}

We will first prove the following lemma.

\begin{lemma}
\label{lem-psef-pair-over-base-surface-eff}
Let $(X, \Delta)$ be a $\mbQ$-factorial compact K\"ahler $3$-fold dlt pair. Assume that $X$ is non-algebraic and $f:X\to Y$ is a projective MRC fibration of $X$ to a compact K\"ahler surface $Y$. Assume that every component of $\Delta$ is horizontal over $Y$ and $(K_X+\Delta)\cdot F\>0$ for a general fiber $F$ of $f$. Then $K_X+\D$ is $\mbQ$-effective.

\end{lemma}

\begin{proof}
It is enough to prove the case when   $K_X+\D$  has   intersection number  $0$  with general fibers of $f$. 
Since $f$ is projective, by Proposition \ref{pro:relative-projective-mmp} we can run a relative $(K_X+\D)$-MMP over $Y$. Such a MMP terminates by Theorem \ref{thm:termination}. We  then obtain a dlt  pair  $(Z,D)$ such that $K_{Z}+ D$ is relatively nef over $Y$.  It is enough to prove that $K_{Z}+  D$ is $\mathbb{Q}$-effective.

Let $g: Z\to Y$ be the induced projective contraction. We can run a relative $K_{Z}$-MMP over $Y$ such that every step is $(K_{Z}+  D)$-trivial. Again, by Theorem \ref{thm:termination} this MMP terminates and we can obtain a dlt pair $(X',   \D')$ with a induced projective morphism $h:X'\to Y$ such that  
\begin{enumerate}
\item either there is a $K_{X'}$-Mori fiber space $f': X'\to Y'$ such that  $(K_{X'}+  \D')\sim_{\mbQ, f'} 0$,
\item or every $K_{X'}$-negative extremal ray of $\NE(X'/Y)$  is $(K_{X'}+  \D')$-positive. 
\end{enumerate}
We claim that the case (2) above is not possible; indeed observe that there is a dense Zariski open subset $V\subset Y$ such that $g^{-1}V\cong h^{-1}V$. In particular, $(K_{X'}+\D')\cdot F=(K_Z+D)\cdot F=0$ for general fibers $F$ of $h$. However, $K_{X'}\cdot F=-2$, and thus by case (2) above we must have $(K_{X'}+\Delta')\cdot F>0$, this is a contradiction.\\
Therefore we are in the first case. By Theorem \ref{thm-base-boundary-effective} there is an effective divisor $B'$ in $Y'$ such that    $$K_{X'}+ \D' \sim_{\mathbb{Q}} f'^*(K_{Y'}+B').$$  Since $f'$ is the MRC fibration of $X'$ and $X'$ is non-algebraic,  it follows that $Y'$ is non-uniruled (see \cite[Remark 3.2]{HP15}). Hence $K_{Y'}$ is pseudo-effective and by abundance for K\"ahler surfaces (see Theorem \ref{thm:surface-abundance}) we  obtain that  $ K_{Y'}+B'$ is $\mathbb{Q}$-effective.   This completes the proof of the lemma. 
\end{proof}~\\

Now we can deduce Theorem \ref{thm-non-vanishing}.

\begin{proof}[{Proof of Theorem \ref{thm-non-vanishing}}]
It is enough to prove that (3) implies (1).  By passing to a dlt model as in Corollary  \ref{cor:exist-terminal-dlt-modification} we may assume that $(X, \Delta)$ has $\mbQ$-factorial dlt singularities such that $X$ has terminal singularities. By \cite[Theorem 1.1]{HP15} there is a $K_X$-MMP which terminates with a Mori fiber space $f:X'\to Y$
such that $Y$ is a $\mbQ$-factorial compact K\"ahler surface with klt singularities. Let $p:W\to X$ and $q:W\to X'$ be the resolution of the graph of the bimeromorphic map $\phi:X\bir X'$ such that $p$ is a log resolution of the pair $(X, \Delta)$ and $q$ is a projective morphism. Then we can write 
\[K_W+\D_W=p^*(K_X+\D)+E,\]
 where $E\>0$ and $\Delta_W\>0$ are two effective $\mbQ$-divisors without common components and $p_*\Delta_W=\D$ and $p_*E=0$. Thus it is enough to show that $K_W+\D_W$ is $\mathbb{Q}$-effective. We note that $(W, \D_W)$ is dlt and $f\circ q:W\to Y$ is the MRC fibration of $W$, since $f:X'\to Y$ is the MRC fibration of $X'$ and the fibers of $W\to X'$ are rationally chain connected by Proposition \ref{c-rccres}. Moreover, observe that there is a dense Zariski open subset $V\subset Y$ such that $\phi$ is an isomorphism over $f^{-1}V$. In particular, for general fibers $F$ of $f$ (also of $f\circ q$) we have $(K_W+\D_W)\cdot F=(K_{X'}+\Delta')\cdot F=(K_X+\D)\cdot F\>0$. Thus by Lemma \ref{lem-psef-pair-over-base-surface-eff}, $K_W+\Delta_W$ is $\mbQ$-effective, this completes the proof.\\

\end{proof}~\\

\section{Cone and Contraction Theorems for pseudo-effective pairs}
\label{sect-uniruled-LMMP}

In this section we will establish the cone and contraction theorem for a dlt pair $(X,\D)$ such that $X$ is a $\mathbb{Q}$-factorial uniruled non-algebraic compact K\"ahler threefold such that $K_X+\D$ is pseudo-effective. The proof is based on the same philosophy as in \cite{HP16, CHP16}. We work under the following hypothesis throughout this section.

\begin{assumption}
\label{asmp-base-MRC-dim-2} Let $X$ be a non-algebraic compact K\"ahler threefold with $\mathbb{Q}$-factorial singularities. Assume that $X$ is uniruled. Then the base of the MRC fibration of $X$ is a non-uniruled surface (see \cite[Remark 3.2]{HP15}). Assume that there is a boundary $\D$ such that $(X,\D)$ is dlt, and that $K_X+\D$ is pseudo-effective.  Then by  Theorem \ref{thm-non-vanishing}, $K_X+\D$ is $\mathbb{Q}$-effective and we can write 
\begin{eqnarray}
\label{eq-K_X+D=sum S_i}
K_X+\D \sim_{\mathbb{Q}} \sum_{i=1}^N \lambda_i S_i,
\end{eqnarray}
where $\lambda_i\>0$ are rational numbers and $S_i$ are prime Weil divisors on $X$.
\end{assumption}~\\

\subsection{Negative part of $K_X+\D$}
 
The goal of this subsection is to prove the following lemma, which is a key ingredient for various results in this section.

\begin{lemma}
\label{lem-non-psef-implies-uniruled}
With notations and hypothesis as in Assumption \ref{asmp-base-MRC-dim-2}, let $S\subset X$ be a surface such that $(K_X+\D)|_S$ is not pseudo-effective. Then $S=S_j$ is one of the component of the decomposition \ref{eq-K_X+D=sum S_i}. Moreover, $S$ is  Moishezon and its minimal resolution $\hat{S}$ is projective and uniruled.
\end{lemma}
We follow the same strategy as in \cite[Lemma 4.1]{HP16}. We will consider two cases according to whether $S$ is horizontal over the base of the MRC fibration of $X$ or not.

\begin{lemma}
\label{lem-horizontal-restriction-psef}
With notations and hypothesis as in Assumption \ref{asmp-base-MRC-dim-2}, let $S\subset X$ be a surface such that $S$ is horizontal over the base of the MRC fibration of $X$. Then $(K_X+\D)|_S$ is pseudo-effective.
\end{lemma}

\begin{proof}
Let $\lambda$ be the positive rational number such that $(K_X+\lambda\D)\cdot F=0$, where $F$ is the general fiber of the base of the MRC fibration $f:X\bir Z$. 
We note that $0<\lambda \< 1$. 
Then $K_X+\lambda\D$  is  $\mathbb{Q}$-effective by Theorem \ref{thm-non-vanishing}.   Moreover, since $S$ is horizontal over $Z$, $S\cdot F>0$ for general fibers $F$ of $f$. In particular, $S$ is not a component of  any  element of $|K_X+\lambda \D|_{\mathbb{Q}}$. Hence $(K_X+\lambda \D)|_{S}$ is pseudo-effective. Thus if $S$ is not a component of $\D$, then it follows that $(K_X+\D)|_S$ is pseudo-effective. So now we assume that $S$ is a component of $\D$. Since $\D$ is a boundary, there is some rational number $\mu \geqslant 1-\lambda>0$ such that the multiplicity of $(\lambda + \mu)\D$ along $S$  is $1$. Let $\pi: \hat{S}\to X$ be the minimal resolution of $S$. Then there is an effective divisor $E$ on $\hat{S}$ such that
 \[K_{\hat{S}}+E \sim_{\mathbb{Q}} \pi^*((K_X+(\lambda+\mu)\D)|_S).\] 
 We note that $S$ is not uniruled, since $f|_S:S\bir Z$ is a generically finite map and $Z$ is not uniruled. Hence $K_{\hat{S}}$ is pseudo-effective and consequently, $(K_X+(\lambda+\mu)\D)|_S$ is pseudo-effective. Since $\mu\geqslant 1-\lambda$, we obtain that $\gamma(K_X+\D)$ is a convex combination of  $K_X+\lambda\D$ and  $K_X+(\lambda+\mu)\D$ for some rational number $\gamma>0$. This proves that $(K_X+\D)|_S$ is pseudo-effective.
\end{proof}

\begin{lemma}
\label{lem-vertical-uniruled}
Let $X$ be a non-algebraic compact K\"ahler $3$-fold with $\mbQ$-factorial klt singularities. Let $S\subset X$ be a surface such that $S$ is vertical over the base of the MRC fibration of $X$. Then $S$ is Moishezon and its minimal resolution $\hat{S}$ is projective and uniruled. 
\end{lemma}

\begin{proof}
	First note that according to \cite[Remark 3.2]{HP15} the base of the MRC fibration of $X$ is a surface. Now replacing $X$ by a terminal model as in Theorem \ref{thm:terminal-model} we may assume that $X$ has $\mbQ$-factorial terminal singularity. Then by \cite[Theorem 1.1]{HP15} $X$ is bimeromorphic to a Mori fiber space $X'\to Y$ over some compact K\"ahler surface $Y$. Let $\tilde{X}$ be a common relosution of $X$ and $X'$ such that $\tilde{X}\to X'$ is projective. Let $\tilde{S}$ be the strict transform of $S$ in $\tilde{X}$. Then $\tilde{S}$ is vertical over $Y$ for the induced projective fibration $g:\tilde{X}\to Y$.  Now we run a $K_{\tilde{X}}$-MMP over $Y$, and obtain a Mori fiber space $\pi:\overline{X}\to T$ over $Y$. If $\tilde{S}$ is contracted  by this MMP, then it is uniruled by Proposition \ref{c-rccres}. Otherwise, let $\overline{S}$ be the bimeromorphic transform of $\tilde{S}$ onto $\overline{X}$ such that $\overline{S}$ is vertical over $T$. In this case, $\overline{S}$ is again uniruled by Proposition \ref{c-rccres}, since the Mori fiber space  $\pi:\overline{X}\to T$ is the MRC fibration of $\overline{X}$.  We can then conclude that $S$ is uniruled, and so is the minimal resolution $\hat{S}$. This implies that $\hat{S}$ is projective by classification, and thus $S$ is Moishezon.
\end{proof}

Now we are ready to conclude Lemma \ref{lem-non-psef-implies-uniruled}.

\begin{proof}[{Proof of Lemma \ref{lem-non-psef-implies-uniruled}}]
By Lemma \ref{lem-horizontal-restriction-psef}, the surface $S$ is vertical over the base of the MRC fibration.  The lemma then follows from Lemma  \ref{lem-vertical-uniruled}.
\end{proof}~\\

\begin{corollary}
\label{cor-K_X+D-nef-test-curve}
With notations and hypothesis as in Assumption \ref{asmp-base-MRC-dim-2}, $K_X+\D$ is nef if and only if it has non-negative intersection with every curve in $X$.
\end{corollary}

\begin{proof}
If $K_X+\Delta$ is nef, then it is clear that $(K_X+\Delta)\cdot C\>0$ for all curves $C\subset X$. So now assume that $K_X+\D$ has non-negative intersection with every curve in $X$. Then by \cite[Theorem 3, Page 413]{Paun98} it is enough to prove that $(K_X+\D)|_S$ is pseudo-effective for every surface $S\subset X$. Assume by contradiction that this is not the case for some surface $S$. Then by Lemma \ref{lem-non-psef-implies-uniruled}, the minimal resolution $\hat{S}$ of $S$ is projective. By assumption,  $\pi^*((K_X+\D)|_S)$ is not pseudo-effective, where $\pi:\hat{S} \to X$ is the induced morphism. By \cite[Theorem 0.2]{BDPP13}, $\hat{S}$ is covered by a family of curves which have negative intersection with $\pi^*((K_X+\D)|_S)$. Since $\pi$ is bimeromorphic onto its image, the image of these curves determine a covering family of curves on $S\subset X$ such that they have negative intersection with $K_X+\D$; this is a contradiction.\\
\end{proof}

\subsection{Bend-and-break theorem for pseudo-effective $K_X+\D$}~\\ 

In this subsection, we will show the following bend-and-break theorem.

\begin{theorem}
\label{thm-b&b-dlt-adjoint}
With notations and hypothesis as in Assumption \ref{asmp-base-MRC-dim-2}, $K_X+\D$ has the bend-and-break property, i.e. there is a rational number $d>0$ such that  if $C\subset X$ is a curve with $-(K_X+\D)\cdot C > d$, then \[[C]=[C_1]+[C_2],\] where $C_1,C_2$ are two integral nonzero effective $1$-cycles.
\end{theorem}

We will first prove this theorem under an additional assumption that $X$ has terminal singularities. The proof of the general case then follows from \cite[Lemma 4.7]{CHP16}.

\begin{lemma}
\label{lem-negative-curve-span-unique-surface}
With notations and hypothesis as in Assumption \ref{asmp-base-MRC-dim-2}, we further assume that $X$ has terminal singularities.   Let $C\subset X$ be a curve such that $(K_X+\D) \cdot C <0$ and $\mathrm{dim}_C\, \mathrm{Chow}(X) >0$. Then there is a unique surface $S_j$ in the decomposition \eqref{eq-K_X+D=sum S_i} such that $C$ and all of its deformations are contained in $S_j$. Moreover, we have $K_{S_j}\cdot C < K_X\cdot C$.
\end{lemma}

\begin{proof}
As in the first part of the proof of \cite[Lemma 5.4]{HP16} we see that there is some $S_j$ in the decomposition \eqref{eq-K_X+D=sum S_i} which contains a family of deformations of $C$. Since $C$ deforms, $S_k\cdot C\>0$ for any $k\neq j$. Then from the decomposition \eqref{eq-K_X+D=sum S_i} it follows that $S_j \cdot C<0$. In particular, $S_j$ is the unique surface which contains all of the deformations of $C$. Moreover, by adjunction we have 
\[
	K_X\cdot C=K_X|_{S_j}\cdot C>(K_X+S_j)|_{S_j}\cdot C\>K_{S_j}\cdot C.
\]
\end{proof}~\\

\begin{lemma}
\label{lem-finite-horizontal-negative-curve} 
With notations and hypothesis as in Assumption \ref{asmp-base-MRC-dim-2}, we further assume that $X$ has terminal singularities. Let $\mathscr{A}$ be the collection of curves $C\subset X$ which satisfy  all of the following conditions.
\begin{enumerate}
\item $(K_X+\D) \cdot C <0$,
\item $\D \cdot C < 0$
\item $C$ is contained in a component of $\D$ which is horizontal over the base of the MRC fibration $f:X\bir Z$.
\end{enumerate}
Then $\mathscr{A}$ is a finite set.
\end{lemma}

\begin{proof}
Let $S$ be a component of $\D$ which is horizontal over $Z$ and $C\subset S$ a curve from the set $\mathscr A$. Then there is a rational number $\mu\>0$ such that the coefficient of $S$ in  $(1+\mu)\D$ is $1$.  Moreover, we have
  \[K_X+(1+\mu)\D) \cdot C <0.\]
Let $\hat{S} \to S$ be the minimal resolution of $S$ and $\pi:\hat{S}\to X$ is the induced morphism. Observe that in order to prove the lemma it is enough to show that there are only finitely many curves in $\hat{S}$ which have negative intersection with $\pi^*(K_X+(1+\mu)\D)$.

Now by adjunction there is an effective $\mbQ$-divisor $E$ such that 
\[K_{\hat{S}}+E \sim_{\mathbb{Q}} \pi^*(K_X+(1+\mu)\D).\] 
Note that $S$ is not uniruled, since $f|_S:S\bir Z$ is generically finite and $Z$ is not uniruled. Thus $K_{\hat{S}}$ is pseudo-effective and by Theorem \ref{thm:surface-abundance} it follows that there is some effective $\mbQ$-divisor $0\<D\sim_{\mbQ} K_{\hat{S}}$. Then any curve in $\hat{S}$ which has  negative intersection with $\pi^*(K_X+(1+\mu)\Delta)$ must be a component of $D+E$. Hence $\mathscr A$ is a finite set.
\end{proof}~\\

 For every component $\D_i$ of $\D$ which is vertical over the base of the MRC fibration $f:X\bir Z$, there is a rational number $\mu_i\>0$ such that the multiplicity of    $(1+\mu_i)\D$ along $\D_i$ is $1$.  Let $\hat{\D}_i \to \D_i$ be the minimal resolution of $\D_i$ and $\pi:\hat{\Delta}_i\to X$ is the induced morphism. Now by adjunction there is an effective $\mbQ$-divisor $E_i$ such that \[K_{\hat{\D}_i}+E_i \sim_{\mathbb{Q}} \pi_i^*(K_X+(1+\mu_i)\D).\]  We let $\mathscr{B}$ be the collection of all curves in $X$ which are contained in $\pi_{i}(\Supp E_i)$ for some $i$. Then $\mathscr{B}$ is a finite set.

Finally, we let $\mathscr{C}$ be the collection of curves which lies either in the singular locus of the support of $\D$ or in the singular locus of $\sum S_i$, where $S_i$ are the components of the decomposition \eqref{eq-K_X+D=sum S_i}. Then again $\mathscr{C}$ is a finite set.

\begin{lemma}
\label{lem-pre-b&b-index}
With notations and hypothesis as in Assumption \ref{asmp-base-MRC-dim-2} we further assume that $X$ has terminal singularities. Define a rational number $b>0$ as follows:
  \[b= \max \{2,  -(K_X+\D)\cdot Z \ |\ Z \subset X \mbox{ is a curve and }  Z\in \mathscr{A}\cup \mathscr{B} \cup \mathscr{C} \},\]
  where $\mathscr A$ is a finite set defined in Lemma \ref{lem-finite-horizontal-negative-curve}.
Then  for any curve $C\subseteq X$, the inequality $$-(K_X+\D)\cdot C > b$$ implies that $\mathrm{dim}_C\, \mathrm{Chow}(X) > 0$.
\end{lemma}

\begin{proof}
	Let $C\subset X$ be a curve such that $-(K_X+\Delta)\cdot C>b$. Then from the defintion of $b$ it follows that $C\not\in\mathscr{A}\cup\mathscr{B}\cup\mathscr{C}$.
Now first we assume that $\D \cdot C <0$. Then there is some component $S$ of $\D$ which contains $C$. Since $C\notin \mathscr{A}$, $S$ is not horizontal over the base of the MRC fibration $f:X\bir Z$.  We also note that $S$ is a unique component of $\Delta$ containing $C$, since otherwise $C\in \mathscr{C}$. Since $(X, \Delta)$ is dlt, there is a rational number $\mu\>0$ such that the multiplicity of $(1+\mu)\D$ along $S$ is $1$.  Then  
\[-(K_X+(1+\mu)\D) \cdot C  > b.\] 
Let $\hat{S} \to S$ be the minimal resolution of $S$ and $\pi:\hat{S}\to X$ is the induced morphism. By adjunction there is an effective $\mbQ$-divisor $E$ such that 
\[K_{\hat{S}}+E \sim_{\mathbb{Q}} \pi^*(K_X+(1+\mu)\D).\]
 Since $C\notin \mathscr{B}$, $C$ is not contained in $\pi(\Supp E)$. Hence the strict transform of $C$ in $\hat{S}$ is a curve $\hat{C}$ and we have \[K_{\hat{S}} \cdot \hat{C} \< (K_{\hat{S}} + E) \cdot \hat{C} = (K_X+(1+\mu)\D) \cdot C < -b.\]
  Since $\hat{S}$ is a smooth projective surface and $b\> 2$, from \cite[Theorem II.1.15]{Kol96} we have 
\[\mathrm{dim}_{\hat{C}}\, \mathrm{Chow}(\hat{S}) > 0.\]
Pushing forward a deformation family of $\hat{C}$ onto $S$ gives deformation family of $C$ on $S$, in particular, $\mathrm{dim}_C\, \mathrm{Chow}(X)>0$.\\

Now we assume that $\D \cdot C \> 0$. Then we have
\[K_X\cdot C \leqslant (K_X+\D)\cdot C < -b \< -2.\]
Then by \cite[Theorem 4.5]{HP16} $C$ is not very rigid and some multiple $mC$ deforms in $X$. Hence there is some component $S=S_j$ in the decomposition \eqref{eq-K_X+D=sum S_i} such that $S \cdot C<0$ and that $S$ contains a family of deformations of $mC$. Moreover, since $C\notin \mathscr{C}$, the component $S=S_j$ is unique. 

Let $\hat{S} \to S$ be the minimal resolution of $S$ and $\pi:\hat{S}\to X$ is the induced morphism. By adjunction there is an effective $\mbQ$-divisor $E$ such that 
\[K_{\hat{S}}+E \sim_{\mathbb{Q}} \pi^* (K_X+S).\] 
Since $C\notin \mathscr{C}$, it is not contained in $\pi(\Supp E)$. Hence the strict transform of $C$ in $\hat{S}$ is a curve $\hat{C}$ and we have 
\[K_{\hat{S}} \cdot \hat{C} \< (K_{\hat{S}} + E) \cdot \hat{C} = (K_X + S) \cdot C < K_X \cdot C <  -b \< -2.\]
 As in the previous case, this implies that  $\mathrm{dim}_C\, \mathrm{Chow}(X) > 0$.
\end{proof}

\begin{lemma}
\label{lem-b&b-index}
With notations and hypothesis as in Assumption \ref{asmp-base-MRC-dim-2} we further assume that $X$ has terminal singularities. Define 
\[d:=\max \{3,b\},\]
 where $b$ is the positive rational number defined in Lemma \ref{lem-pre-b&b-index}. Then if $C\subset X$ is a curve, the inequality  $-(K_X+\D)\cdot C > d$ implies that \[[C]=[C_1]+[C_2],\]
 where $C_1,C_2$ are two non-zero integral effective $1$-cycles. In particular, $K_X+\D$ has the bend-and-break property.
\end{lemma}

\begin{proof}
As in the proof of Lemma \ref{lem-pre-b&b-index}, we need to consider two cases: either $\D \cdot C <0$ or $\D\cdot C \> 0$. As in the proof of Lemma \ref{lem-pre-b&b-index}, in both cases we end up with a unique uniruled surface $S$ such that $C$ is not contained in the singular locus of $S$. Let $\hat{S}\to S$ be the minimal resolution of $S$ and $\pi:\hat{S}\to X$ is the induced morphism. Let $\hat{C}$ be the strict transform of $C$ in $\hat{S}$.  Then as in the proof of Lemma \ref{lem-pre-b&b-index} we have $K_{\hat{S}} \cdot \hat{C} < -d \< -3$ in both cases. The rest of proof follows from a similar argument as in the proof of \cite[Lemma 5.7]{HP16}.
\end{proof}

We can now deduce Theorem \ref{thm-b&b-dlt-adjoint}.

\begin{proof}[{Proof of Theorem \ref{thm-b&b-dlt-adjoint}}]
By Corollary  \ref{cor:exist-terminal-dlt-modification}, there is a bimeromorphic morphism $f:(X', \D')\to (X,\D)$ such that 
\begin{enumerate}
\item $X'$ has  $\mathbb{Q}$-factorial  terminal singularities, and
\item $(X',\D')$ is a dlt pair such that $$K_{X'}+\D' \sim_{\mathbb{Q}} f^*(K_X+\D)$$
\end{enumerate}
By Lemma \ref{lem-b&b-index}, $K_{X'}+\D'$ has the bend-and-break property. Hence by \cite[Lemma 4.2]{CHP16}, $K_X+\D$ also has the bend-and-break property.
\end{proof}~\\

\subsection{Cone theorem for pseudo-effective DLT pairs}
\label{sect-cone-thm}

We recall that the   cone $\mathrm{\overline{NA}}(X)$ for a normal compact K\"ahler variety is  the closed cone generated by positive closed current of bidimension $(1,1)$. It replaces the Mori cone  $\mathrm{\overline{NE}}(X)$ in the non-algebraic setting, and is  dual to the closure of the cone generated by the classes of K\"ahler forms (see  \cite[Section 1]{HP16} for details).

\begin{theorem}\label{thm-NA-cone-uniruled-adjoint-pair}
	 With notations and hypothesis as in Assumption \ref{asmp-base-MRC-dim-2}, there is a rational number $d>0$ and an at most countable set of curves $\{\Gamma_i\}_{i\in I}$ such that  \[0< (K_X+\D) \cdot \Gamma_i \< d\] and that  \[\overline{\mathrm{NA}}(X) = \overline{\mathrm{NA}}(X)_{(K_X+\D) \> 0} + \sum_{i\in I} \mathbb{R}^{+} [\Gamma_i].\]
\end{theorem}

\begin{proof}
The proof works with almost verbatim argument as in the proof of \cite[Theorem 4.1, page 985]{CHP16}. We note that in \cite[Theorem 4.1]{CHP16} $K_X$ is assumed to be pseudo-effective, which we don't have, but in the proof of \cite[Theorem 4.1, page 985]{CHP16} only $K_X+\Delta$ being pseudo-effective is used, which we have in our case. Moreover, we replace the use of \cite[Lem. 4.1, Pro. 4.1]{CHP16} by Lemma \ref{lem-non-psef-implies-uniruled} and Theorem \ref{thm-b&b-dlt-adjoint}, respectively.
\end{proof}~\\

Together with \cite[Theorem 1.3]{CHP16}, we obtain the following cone theorem for dlt pair $(X, \D)$ with pseudo-effective $K_X+\D$.

\begin{theorem}
\label{thm-NA-cone-adjoint-pair}
Let $(X, \Delta)$ be a  compact K\"ahler $\mbQ$-factorial $3$-fold dlt pair.  Assume that $K_X+\D$ is pseudo-effective. Then there is a rational number $d>0$ and a countable set of curves $\{\Gamma_i\}_{i\in I}$ such that  \[0< (K_X+\D) \cdot \Gamma_i \< d\]
 and 
  \[\overline{\mathrm{NA}}(X) = \overline{\mathrm{NA}}(X)_{(K_X+\D)\> 0} + \sum_{i\in I} \mathbb{R}^{+} [\Gamma_i].\]
\end{theorem}

\subsection{Contraction theorems}~\\
\label{sect-contr-thm}

We will study the contraction of $(K_X+\D)$-negative extremal rays of $\overline{\mathrm{NA}}(X)$ in this subsection.

\begin{definition}\cite[Def. 7.1, Rmk. 7.2]{HP16}\label{def:extremal-ray}
Let $(X, \Delta)$ be a $\mbQ$-factorial compact K\"ahler $3$-fold dlt pair such that $K_X+\Delta$ is pseudo-effective. 
Let $R$ be a $(K_X+\Delta)$-negative extremal ray of $\NA(X)$. A supporting   class of $R$ is a $(1, 1)$  class $\alpha$ on $X$ such that $R=\{\gamma\in\NA(X): \alpha\cdot\gamma=0\}$.
We say that $R$ is small if every curve $C\subset X$ with $[C]\in R$ is very rigid in the sense of \cite[Definition 4.3]{HP16}. Otherwise we say $R$ is of divisorial-type (See \cite[Remark 7.2]{HP16}).
\end{definition}

\subsubsection{Contraction of small rays} 

\begin{theorem}
\label{thm-contraction-small-ray}
With notations and hypothesis as in Assumption \ref{asmp-base-MRC-dim-2}, let $R$ is a $(K_X+\D)$-negative extremal ray in $\overline{\mathrm{NA}}(X)$. If $R$ is small, then the contraction $c_R:X\to Y$ of $R$ exists.
\end{theorem}

\begin{proof}
	First of all, since $R$ is a $(K_X+\Delta)$-negative extremal ray, there is a supporting nef and big $(1, 1)$  class $\alpha$ of  $R$; note that the existence of $\alpha$ follows similarly as in the proof of \cite[Proposition 7.3]{HP16}, by using Theorem \ref{thm-NA-cone-adjoint-pair}. Also, the bigness of $\alpha$ follows from \cite[Eqn. (10), page 987]{CHP16}. Next we observe that the contraction of $R$ follows exactly as in the proof of \cite[Theorem 4.2]{CHP16} provided we can show that $\alpha^2\cdot S>0$ holds for every irreducible surface $S\subset X$ (\cite[Pro. 4.4]{CHP16}). Finally, this result holds by the following proposition.
\end{proof}~\\

\begin{proposition}
With notations and hypothesis as in Theorem \ref{thm-contraction-small-ray}, if $\alpha$ is a nef supporting class of $R$, then $\alpha^2 \cdot S>0$ for any surface $S\subseteq X$.
\end{proposition}

\begin{proof}
We follow a similar idea as in the proof of \cite[Proposition 4.4]{CHP16}.  By scaling $\alpha$ if necessary we may assume that $\omega=\alpha - (K_X+\D)$ is a K\"ahler class. By contradiction assume that $\alpha^2\cdot S=0$ for some surface $S\subset X$. Now if $\alpha|_S = 0$, then $-(K_X+\D)|_S$ is ample. In particular, $S$ is projective and covered by uncountably many curves.  Since $\alpha|_S =0$, the classes of these curves in $X$ are contained in $R$. This contradicts that $R$ is small.\\

Next assume that $\alpha|_S \neq 0$ and $\alpha^2 \cdot S = 0$. Then we have 
\begin{equation}\label{eqn:contradiction}
(K_X+\D)\cdot \alpha \cdot S=(\alpha - \omega) \cdot \alpha \cdot S =-\omega \cdot \alpha \cdot S < 0,	
\end{equation}
as $\alpha|_S$ is nef and non-zero. Therefore $(K_X+\D)|_S$ is not pseudo-effective. Now let $\hat{S}\to S$ be the minimal resolution of $S$ and $\pi:\hat{S}\to X$ is the induced morphism. Then as in the proof of \cite[Proposition 4.4]{CHP16} we see that in order to conclude,  it is enough to find an effective $\mbQ$-divisor $E$ on $\hat{S}$ such that  \[(K_{\hat{S}}+E)\cdot \pi^*(\alpha) <0.\]
To this end we will consider  two cases related to \eqref{eqn:contradiction}.\\

\textbf{Case I:} Assume that $\D \cdot \alpha \cdot S<0$. Since $\alpha$ is nef, this implies that $S$ is a component of $\D$.  Since $\D$ is a boundary divisor, there is a rational number $\mu\geqslant 0$ such that the multiplicity of $(1+\mu)\D$ along $S$ is $1$. Then \[(K_X+(1+\mu)\D) \cdot \alpha \cdot S \< (K_X+\D) \cdot \alpha \cdot S < 0.\]  By adjunction there is an effective $\mbQ$-divisor $E$ such that \[(K_{\hat{S}}+E) \cdot \pi^* \alpha = \pi^*(K_X+(1+\mu)\D) \cdot \pi^*\alpha = (K_X+(1+\mu)\D) \cdot \alpha \cdot S<0.\]

\textbf{Case II:} Assume that $\D \cdot \alpha \cdot S \geqslant 0$. Then we have  \[K_X\cdot \alpha \cdot S \leqslant (K_X+\D) \cdot  \alpha \cdot S <0.\] Since $(K_X+\D)\cdot \alpha \cdot S <0$,  the surface $S=S_j$ must be a component in the decomposition \eqref{eq-K_X+D=sum S_i}, and we must have $$S \cdot \alpha \cdot S<0.$$ Therefore $(K_X+S)\cdot \alpha \cdot S <0$. Hence by adjunction, there is an effective $\mbQ$-divisor $E$ on $\hat{S}$ such that \[(K_{\hat{S}}+E) \cdot \pi^* \alpha = \pi^*(K_X+S) \cdot \pi^*\alpha = (K_X+S) \cdot \alpha \cdot S<0.\]
\end{proof}~\\

\subsubsection{Contraction of divisorial rays}~\\

With notations and hypothesis as in Assumption \ref{asmp-base-MRC-dim-2}, let $R$ is a $(K_X+\D)$-negative  extremal ray of $\overline{\mathrm{NA}}(X)$ of divisorial type.  Then there is a unique surface $S\subset X$ such that $S$ contains all curves $C\subset X$ such that $[C]\in R$ and also that $S$ is covered by these curves. Let $\nu: \tilde{S} \to S$ be the normalization of $S$. Let $\alpha$ be a  nef supporting class of $R$.  Then the nef reduction of $\nu^*(\alpha|_S)$ gives a fibration \[\tilde{f}:\tilde{S} \to \tilde{T}.\]  We define \[n(\alpha) = \mathrm{dim}\, \tilde{T}.\]
 Since $\alpha|_S$ intersects a covering family of curves in $S$ trivially, it follows that $n(\alpha) \in \{0, 1\}$.

\begin{theorem}
\label{thm-contraction-divisorial-nefdim=0}
With notations and hypothesis as in Assumption \ref{asmp-base-MRC-dim-2}, let $R$ is a $(K_X+\D)$-negative extremal ray of $\overline{\mathrm{NA}}(X)$ of divisorial type. Let $\alpha$ be a nef supporting class of $R$. If $n(\alpha)=0$, then the contraction $c_R:X\to Y$ of $R$ exists.  
\end{theorem}

\begin{proof}
The same proof as in \cite[Corollary 7.7]{HP16} works here using the Cone Theorem \ref{thm-NA-cone-adjoint-pair}.
\end{proof}~\\

In case of $n(\alpha)=1$, we have the following result. 

\begin{proposition}
\label{prop-contraction-divisorial-nefdim=1}

With notations and hypothesis as in Assumption \ref{asmp-base-MRC-dim-2}, let $R$ is a $(K_X+\D)$-negative  extremal ray of $\overline{\mathrm{NA}}(X)$ of divisorial type.  
Let $S\subset X$ be the unique surface which contains all curves $C\subset X$ such that $[C]\in R$ and also that $S$ is covered by these curves. Let $\alpha$ be a nef supporting class of $R$ and assume that $n(\alpha)=1$.  Then the contraction $c_R:X\to Y$ of $R$ exists if one of the following conditions holds:
\begin{enumerate}
\item $S$ has slc singulairties, or
\item $X$ is terminal and $K_X\cdot R <0$.

\end{enumerate}

\end{proposition}

\begin{proof}
If the first condition holds, then the same proof as in \cite[Proposition 4.5]{CHP16} works without any change here. In the second case $R$ is a $K_X$-negative extremal ray with $X$ having terminal singularities, so it follows similarly as in the proof of \cite[Lem. 7.8, Cor. 7.9]{HP16}.

\end{proof} ~\\

\part{Log abundance theorem for log canonical pairs}

\section{Reduction step}
In this section we will prove reduction theorem which will split the abundance problem for klt pairs into two cases.

\begin{definition}\label{def:crepant-map}
        Let $\phi:X\bir X'$ be a bimeromorphic map between two normal $\mbQ$-factorial varieties and $D$ is an $\mbR$-Cartier divisor on $X$. Then $\phi$ is called a $D$-crepant if there is a normal variety $W$ and morphisms $p:W\to X$ and $q:W\to X'$ resolving $\phi$, i.e., $\phi\circ p=q$, such that $p^*D=q^*(\phi_*D)$ holds.
\end{definition}

\begin{theorem}
\label{thm-reduction-2-cases}
Let $(X, \Delta)$ be a klt (resp. lc) pair such that $X$ is a $\mathbb{Q}$-factorial compact K\"ahler $3$-fold with terminal singularities. Assume that $K_X+\D$ is nef.  Then there is a $(K_X+\Delta)$-crepant  bimeromorphic map $\phi:(X,\Delta)\bir (X', \Delta')$ such that exactly one of the following two holds:
\begin{enumerate}
\item There is a $K_{X'}$-Mori fiber space $X'\to Y$ such that $K_{X'}+\Delta'$ is numerically trivial over $Y$.
\item $K_{X'}+(1-\epsilon)\Delta'$  is nef for all $0\<\epsilon\ll 1$.
\end{enumerate}
Moreover,  $X'$ has $\mathbb{Q}$-factorial terminal singularities and $(X',\Delta')$  is  klt  (resp. lc).
\end{theorem}

We need the following lemma to prove this theorem.

\begin{lemma}
\label{lem-existence-Mori-fibration}
Let $(X, \Delta)$ be a klt (resp. lc) pair such that $X$ is a $\mathbb{Q}$-factorial terminal compact K\"ahler $3$-fold. Assume  that  $X$ is uniruled and non-algebraic.  Further assume that $K_X+\Delta$ is nef and it intersects the general fibers of the MRC fibration of $X$ trivially. Then there is a $(K_X+\Delta)$-crepant  bimeromorphic model $\phi:(X, \Delta)\bir (X',\Delta')$ such that 
\begin{enumerate}%[label=(\roman*)]
\item $\Delta':=\phi_*\Delta$,
\item $X'$ has $\mathbb{Q}$-factorial terminal singularities,
\item $(X', \Delta')$ is klt (resp. lc),
\item there is a $K_{X'}$-Mori fiber space $X'\to Y'$, and
\item $(K_{X'}+\Delta')\sim_{\mbQ, Y'} 0$.
\end{enumerate}
\end{lemma}

\begin{proof}
First of all, since $X$ is uniruled and non-algebraic, from Lemma \ref{l-proj} it follows that the base of the MRC fibration of $X$ has 
dimension $2$. We recall that a  K\"ahler class $\omega$ on $X$ is called normalized if $K_X+\omega$ intersects the general fibers of the MRC fibration of $X$ 
trivially (see \cite[Definition 1.2]{HP15}). Set $D:=K_X+\Delta$; then we claim that one of the following holds:

\begin{enumerate}
\item\label{item:scaling-extremal-ray} either there is a $(K_X+\omega)$-negative extremal ray $R$ such that $D\cdot R =0$, 
\item\label{item:non-zero-sup} or there exist a smallest non-negative real number $\lambda\>0$ such that $K_X+\omega+\lambda D$ is nef.
\end{enumerate}
Indeed, if \eqref{item:scaling-extremal-ray} does not hold, then either $K_X+\omega$ is nef or  every $(K_X+\omega)$-negative extremal ray $R$ satisfies $D\cdot 
R>0$. In the first case we set $\lambda=0$, and in the second case the existence of smallest $\lambda$ can be proved as in Proposition \ref{pro:mmp-with-scaling} using the boundedness of the length of extremal rays from \cite[Theorem 3.6]{HP15}.\\
Now if we are in case \eqref{item:scaling-extremal-ray}, then the $K_X$-negative extremal ray $R$  can be contracted as in \cite[Lem. 3.21, Thm. 3.24]{HP15} by a projective morphism; Note that $D$ is crepant with respect to this morphism.  By induction and \cite{HP15}, we can run a $K_X$-MMP, trivial with respect to $D$, and  obtain a $D$-crepant bimeromorphic map $\psi:(X, \Delta)\bir 
(\hat{X},\hat{\Delta})$ such that for any  normalized K\"ahler  class $\hat{\omega}$ on $\hat{X}$ there exist a smallest non-negative real number $\lambda$ such that 
$K_{\hat{X}}+\hat{\omega}+\lambda \hat{D}$ is nef, where $\hat{D}=K_{\hat{X}}+\hat{\Delta}=\phi_*D=\phi_*(K_{X}+\Delta)$ is nef. Moreover, $\hat{X}$ has 
$\mathbb{Q}$-factorial terminal singularities and $\hat{D}$ intersects the general fibers of the MRC fibration of $\hat{X}$ trivially. Therefore $\hat{\omega} + (K_{\hat{X}}+\hat{\omega}+2 \lambda \hat{D})$ is then a normalized K\"ahler  class on $\hat{X}$. Then \[K_{\hat{X}}+ (\hat{\omega} + K_{\hat{X}}+\hat{\omega}+2 \lambda \hat{D}) = 
2 (K_{\hat{X}}+\hat{\omega}+\lambda \hat{D})\quad\mbox{ is nef.}\] 

If we are in case \eqref{item:non-zero-sup} above, then $K_X+\omega+\lambda D$ is nef for the smallest $\lambda\in\mbR^{\>0}$ and $\omega+\lambda D$ is a normalized K\"ahler class, since $D=K_X+\Delta$ intersects the general fibers of the MRC firbation of $X$ trivially. Renaming $X, \omega$ and $D$ by $\hat{X}, \hat{\omega}$ and $\hat{D}=K_{\hat{X}}+\hat{\Delta}$ we see that, in either case above, by \cite[Theorem 1.4]{HP15} there is a projective Fano fibration $\hat{f}: \hat{X}\to \hat{Y}$ onto a surface $\hat{Y}$. 

We remark that $\hat{f}$ is indeed the MRC fibration, as the base of the MRC fibration of $X$, and hence that of $\hat{X}$, has dimension $2$. Note that $\hat{D}$ is nef on $\hat{X}$ and $\hat{D}\cdot F=(K_{\hat{X}}+\hat{\D})\cdot F=0$ for general fibers of $F(\cong\mbP^1)$ of $\hat{f}$. Now since $\hat{f}$ is a projective morphism, by Proposition \ref{pro:relative-projective-mmp} and Section \ref{subsec:mmp-with-scaling} we run a relative $K_{\hat{X}}$-MMP over $\hat{Y}$ with the scaling of $\hat{\D}$. This MMP terminates with a Mori fiber space, say $f':X'\to Y'$ over $Y$ such that $X'$ has $\mbQ$-factorial terminal singularities. By Lemma \ref{lem:directed-relative-mmp} every step of this MMP is $\hat{D}$-trivial. In particular, the induced bimeromorphic map $\phi:\hat{X}\bir X'$ is $\hat{D}$-crepant. Let $\Delta':=\phi_*\hat{\D}$; then $(K_{X'}+\D')\sim_{\mbQ, f'} 0$. The composite of $\phi$ and $\psi$ gives the required map.  
 Finally, since every step of  $\phi\circ\psi$ is $(K_X+\Delta)$-trivial, the discrepancy of $(X',\Delta')$ is same as that of $(X,\Delta)$, and thus $(X', \Delta')$ is klt (resp. lc). 
This shows (3).
\end{proof}~\\

 Now we conclude Theorem \ref{thm-reduction-2-cases}.

\begin{proof}[{Proof of Theorem \ref{thm-reduction-2-cases}}]
If $X$ is projective, then this is just \cite[Lemma 5.3 and Lemma 7.1]{KMM94}. Assume  that  $X$ is non-algebraic and uniruled. If $(X,\D)$ is lc, by taking a dlt model as in Corollary \ref{cor:exist-terminal-dlt-modification}, we may assume that $(X,\D)$ is a $\mbQ$-factorial dlt pair such that $X$ has terminal singularities.  Let $\lambda$ be the  positive rational number such that $K_X+\lambda\D$ has zero intersection with general fibers of the MRC fibration of $X$.  If $\lambda=1$, then we can apply Lemma \ref{lem-existence-Mori-fibration}. So assume from now on that  $\lambda < 1$. Then $(X,\lambda \D)$ is klt.  By Theorem \ref{thm-non-vanishing}, $K_X+\lambda\D$ is $\mathbb{Q}$-effective. If $K_X+\lambda\D$ is nef, then we can take $(X',\D')=(X,\D)$ and we are in the second situation. Assume that $K_X+\lambda\D$ is not nef.  We will show that there exist  a  terminating $(K_X+\lambda\D)$-MMP, which is trivial with respect to $K_X+\D$. 

We first assume that  there is  a $(K_X+\lambda \D)$-negative extremal ray $R$ of $\overline{\mathrm{NA}}(X)$ which is $(K_X+  \D)$-trivial. Then $R$ is also $K_X$-negative. Hence the contraction of $R$ exists by Theorem \ref{thm-contraction-divisorial-nefdim=0} and Proposition \ref{prop-contraction-divisorial-nefdim=1}, if $R$ is divisorial, and by Theorem \ref{thm-contraction-small-ray}, if $R$ is small. Moreover, if $R$ is small, then the corresponding flip exists by \cite[Theorem 4.19]{CHP16}.  Let $h_1: (X,\D) \dashrightarrow (X_1,\D_1)$ be the divisorial contraction or the flip corresponding to the contraction of $R$. Then $h_1$ is $(K_X+\D)$-crepant and $X_1$ has terminal singularities. By induction, we can run a $(K_X+\lambda \D)$-MMP, trivial with respect to $K_X+\D$.  By Theorem \ref{thm:termination} this MMP terminates. Therefore we obtain a $(K_X+\D)$-crepant  bimeromorphic map $\phi:(X, \D)\bir (X',\D')$ such that $K_{X'}+\Delta'$ is nef and $X'$ has $\mathbb{Q}$-factorial terminal singularities, and    
\begin{enumerate}[label=(\roman*)]
\item either $K_{X'}+\lambda\D'$ is nef,
\item or every $(K_{X'}+\lambda\D')$-negative extremal ray  is $(K_{X'}+ \D')$-positive.  
\end{enumerate}
In the first case, $K_{X'}+(1-\epsilon)\D'$  is nef for all $0\< \epsilon \ll 1$. In the second case, as in the proof of Proposition \ref{pro:mmp-with-scaling} using the boundedness of the lengths of  extremal rays of Theorem \ref{thm-NA-cone-uniruled-adjoint-pair} we see that there is a smallest positive rational number $\mu\in\mbQ^+$ such that $(K_{X'}+\D')+\mu(K_{X'} + \lambda \D')$ is nef. This shows that $K_{X'}+(1-\epsilon)\D'$  is nef for all $0\< \epsilon\ll 1$. Since this MMP is $(K_X+\D)$-trivial, we obtain that $(X',\D')$ is also klt (respectively lc).\\

Finally we assume that $X$ is not uniruled. Then we can run a $K_X$-MMP, trivial with respect to $K_X+\D$, by \cite[Theorem 1.1 and Theorem 1.2]{HP16} and the rest of the arguments work as above.
\end{proof}

\section{Construction of bimeromorphic models}

In this section we will construct several bimeromorphic models which will help us in the next section to show that $\nu(X, K_X+\D)=1$ implies $k(X, K_X+\D)>0$.

\begin{lemma}
\label{lem-support-modification}
Let $(X,\D)$ be a  $\mathbb{Q}$-factorial compact K\"ahler $3$-fold klt pair.  Assume that  $K_{X}+(1-\epsilon)\D$  is nef for all $0\< \epsilon \ll 1$. Then there is a boundary $\D'$ in $X$ such that 
\begin{enumerate}
\item $(X,\D')$ is klt,
\item $K_X+(1-\epsilon)\D'$ is nef for all $0\< \epsilon \ll 1$,
\item there is some $D\in |K_X+\D'|_{\mathbb{Q}}$ such that $\Supp D=\Supp \D'$, and
\item  $\nu(K_{X}+\D)=\nu(K_{X }+\D')$ and $\kappa(K_X+\D)=\kappa(K_{X}+\D')$.
\end{enumerate}
\end{lemma}

\begin{proof}
Since $K_X+\lambda\D$ is nef for some $\lambda<1$, by Theorem \ref{thm-non-vanishing-general-setting} there is an effective $\mbQ$-divisor $P \sim_{\mbQ}(K_X+\lambda\Delta)$. Set $Q:=P+(1-\lambda)\D$; then $Q\sim_{\mbQ} K_X+\D$ is nef and its support contains that of $\D$. Choose a rational number $0<\gamma\ll 1$ such that $(X, \D+\gamma Q)$ is klt. Set $\D'=\D + \gamma Q$; then 
\[K_X+(1-\epsilon)\D' = K_X+(1-\epsilon )\D + (1-\epsilon) \gamma Q\]
 is nef for all $0\< \epsilon \ll 1$.    Moreover,  we note that 
 \[K_X+\D' \sim_{\mathbb{Q}} (1+\gamma) (K_X+\D).\]
  This proves property (4). Let $D=\frac{1}{1+\gamma} Q$, then $D\in |K_X+\D'|_{\mathbb{Q}}$ and its support is the same as that of $\D'$.  
\end{proof}

\begin{lemma}
\label{lem-exist-dlt-minimal-model}
Let $(X,B)$ be a $\mathbb{Q}$-factorial compact K\"ahler $3$-fold dlt pair. Assume that $B$ is reduced, and that there is an element $D \in |K_X+B|_{\mathbb{Q}}$ such that $\mathrm{Supp}\, D = B$. Then there is a $(K_X+B)$-MMP which terminates to a dlt pair $(X',B')$ such that $K_{X'}+B'$ is nef. 
\end{lemma}

\begin{proof}
If $X$ is projective, then we can apply the classical algebraic MMP. If $K_X$ is pseudo-effective, then we can apply \cite[Theorem 4.6]{CHP16}. It remains to deal with the case when $X$ is uniruled and non-algebraic.

Assume that $(K_X+B)$ is not nef. 

We will show that a $(K_X+B)$-MMP exists.  By Theorem \ref{thm-NA-cone-adjoint-pair}, there is a $(K_X+B)$-negative extremal ray $R$. If $R$ is small, then there exist a small contraction $c_R:X\to Y$ by Theorem \ref{thm-contraction-small-ray}. The flip  of $c_R$ exists by Theorem \ref{thm:existence-of-flips}.

Assume that  $R$ is divisorial. Let $S$ be the unique surface containing and covered by all the curves $C\subset X$ such that $[C]\in R$. Then $S$ is a component of $B$, since $ B=\Supp D$ and $D\cdot R<0$. Thus $S$ is normal (being a dlt center) and by adjunction $(S, 0)$ has klt singularities. Hence the contraction exists by Theorem \ref{thm-contraction-divisorial-nefdim=0} ans Proposition \ref{prop-contraction-divisorial-nefdim=1}. Termination of the $(K_X+B)$-MMP follows from Theorem \ref{thm:termination}.
\end{proof}

The following result is an analogue of \cite[Lemma 13.2]{Kol92}.
\begin{proposition}
\label{prop-lc-modification}
Let $(X,\D)$ be a klt pair, where $X$ is a $\mathbb{Q}$-factorial compact K\"ahler $3$-fold with terminal singularities.  Suppose that there is a nef $\mbQ$-divisor $D\in |K_{X}+\D|_{\mathbb{Q}}$ such that $\mathrm{Supp}\, \D = \mathrm{Supp}\, D$.
Then there is a $\mathbb{Q}$-factorial   pair $(\hat{X},\hat{B})$  such that 
\begin{enumerate}
\item $(\hat{X},\hat{B})$ is dlt,
\item $\hat{B}$ is reduced,

\item $\hat{X}\backslash \hat{B}$ has terminal singularities,
\item $K_{\hat{X}}+\hat{B}$ is nef,
\item there is some $\hat{D} \in |K_{\hat{X}}+ \hat{B}|_{\mathbb{Q}}$ such that $\mathrm{Supp}\, \hat{D}= \hat{B}$, and 
\item $\nu(K_{X}+\D)=\nu(K_{\hat{X}}+\hat{B})$ and $\kappa(K_X+\D)=\kappa(K_{\hat{X}}+\hat{B})$.

\end{enumerate}
\end{proposition}

\begin{proof} 
Since $X $ has isolated (terminal) singularities, there is a partial log resolution $r: \tilde{X} \to X$  of $(X,\D)$ such that 
\begin{enumerate}
\item[$\bullet$] $r$ is an isomorphism over $X\backslash \Supp\D$,
\item[$\bullet$] $\tilde{X}$ is smooth in a neighborhood of $r^*\D$ and $r^* \D$ has simple normal crossing support.
 \end{enumerate} 
 Since $(X, \D)$ is klt, we can write 
 \[K_{\tilde{X}}+B'=r^*(K_X+\Delta)+E\sim_{\mbQ}r^*D+E,\]
  where $B'=r^{-1}_*\Delta+\Ex(r)$, and $E\>0$ is an effective $\mbQ$-divisor supported on the exceptional locus of $r$.\\
Set $\tilde{B}:=\lru B'\rru=\lru r^{-1}_*\Delta\rru+\Ex(r)$ and $\tilde{D}:=r^*D+E+(\lru r^{-1}_*\Delta\rru-r^{-1}_*\Delta)\>0$. Then $\tilde{D}\in|K_{\tilde{X}}+\tilde{B}|_{\mbQ}$, $\Supp \tilde{D}=\Supp (r^*D+\Ex(r))= \tilde{B}$, since $\Supp D=\Supp \D$.
Moreover, from the construction it also follows that $\tilde{X}\setminus \tilde{B}$ has terminal singularity and $(\tilde{X}, \tilde{B})$ is dlt. We claim that $\tilde{X}$ is $\mbQ$-factorial. Let $F$ be a prime Weil divisor on $\tilde{X}$. If $F$ is $r$-exceptional, then by construction, $\tilde{X}$ is smooth in a neighborhood of $F$ and thus $F$ is Cartier. If $F$ is not $r$-exceptional, then $r_*F$ is a Weil divisor on $X$ and hence it is $\mbQ$-Cartier, since $X$ is $\mbQ$-factorial. Then $r^*r_*F=F+G$, where $G$ is a $r$-exceptional divisor. Since we have just showed that $G$ is $\mbQ$-Cartier, it follows that $F$ is $\mbQ$-Cartier. Finally, it is easy to see that $\omega_{\tilde{X}}^{[k]}$ is a line bundle for some $k\>1$, since the same is true for $X$ and the exceptional divisor of $r$ is $\mbQ$-Cartier.\\
Now by  Lemma \ref{lem-exist-dlt-minimal-model} we can  run a terminating $(K_{\tilde{X}}+\tilde{B})$-MMP to obtain a $\mathbb{Q}$-factorial dlt pair $(\hat{X},\hat{B})$ such that $K_{\hat{X}}+\hat{B}$ is nef. Let $\hat{D}$ be the bimeromorphic transform of $\tilde{D}$ in $\hat{X}$. Then $\hat{D} \in |K_{\hat{X}}+\hat{B}|_{\mathbb{Q}}$ and $\mathrm{Supp}\, \hat{D}= \hat{B}$.  Moreover, since $0\<\tilde{D}\sim_{\mbQ}(K_{\tilde{X}}+\tilde{B})$ and $\Supp\tilde{D}=\Supp B$, all the curves and divisors contracted by the MMP is contained in the $\Supp \tilde{D}=\tilde{B}$. In particular, we have 
\[\tilde{X}\backslash \Supp \tilde{B} \cong \hat{X}\backslash  \Supp \hat{B},\]
 and thus $\hat{X}\backslash  \hat{B}$ has terminal singularities.\\

Now we will show (6). Note that since MMP preserves Kodaira dimension, we have
\[\kappa(K_{\hat{X}}+\hat{B}) = \kappa(K_{\tilde{X}}+\tilde{B}) = \kappa(\tilde{D}).\]
 On the other hand, since $\mathrm{Supp}\, \D = \mathrm{Supp}\, D$, there is a positive number $m$ such that 
\[ \tilde{D} = r^*D+E + (\lceil{r_*^{-1}\D}\rceil -  r_*^{-1} \D)\< r^*D+E+mr^*D.\] 
Now since $E$ is a $r$-exceptional effective divisor, we have 
\[\kappa(D)=\kappa(r^*D)\<\kappa(\tilde{D})\< \kappa ((m+1)r^*D)=\kappa(D).\]
 Therefore $\kappa(K_X+\D)=\kappa(D)=\kappa(\tilde{D})=\kappa(K_{\hat{X}}+\hat{B})$. 

For the equality of numerical dimensions, we take a desingularization $W$ of the graph of $\phi:\tilde{X} \dashrightarrow \hat{X}$, with induced morphisms $p:W \to \tilde{ X}$ and $q:W \to \hat{X}$. Since every step of the MMP $\phi:\tilde{X}\bir \hat{X}$ is $\tilde{D}$-negative and $\hat{D}=\phi_*\tilde{D}$, we obtain that $p^*\tilde{D}$ and $q^*\hat{D}$ have same support. In particular, $p^*(r^*D)$ and $q^*\hat{D}$ have same support. Then the effective divisors 
\[
	(r\circ p)^*D\sim_\mbQ p^*(K_X+\Delta)\mbox{ and } q^*\hat{D}\sim_\mbQ q^*(K_{\hat{X}}+\hat{B})
\]
 are both nef and have same support.

Therefore by \cite[Lemma 11.3.3]{Kol92} (using a K\"ahler class in place of an ample divisor), we obtain that $\nu(K_X+\D)=\nu(K_{\hat{X}}+\hat{B})$.
\end{proof}~\\

The following lemma is an analogue of \cite[Lemma 6.5]{CHP16}.

\begin{lemma}
\label{lem-isolating-trivial-component}
Let $(X,B)$ be a reduced dlt pair such that $X$ is a $\mathbb{Q}$-factorial compact K\"ahler threefold. Assume that
\begin{enumerate}
\item[$\bullet$] there is a nef divisor $D\in |K_X+B|_{\mathbb{Q}}$ such that $\mathrm{Supp}\, D=B$,
\item[$\bullet$] $X\backslash B$ has terminal singularities. 
\end{enumerate}
Choose a rational number $\lambda\in (0, 1)$ such that $D-(1-\lambda)B$ is effective with support equal to $B$. Let $S$ be a reduced Weil divisor contained in the support of $B$ and set $Z:=B-S$. Assume that $(K_X+B)|_T\num 0$ for every component $T$ of $Z$ and that $S\cap Z\neq\emptyset$. Set $X_0:=X$, $S_0:=S$, $Z_0:=Z$ and $B_0:=B$.  Then there is a finite sequence of $(K_X+\lambda S + Z)$-flips and divisorial contractions
 \[\{f_i: X_i \dashrightarrow X_{i+1}\}_{i=0,...,n}\] 
 such that for each $i\> 0$ we have 
\begin{enumerate}
\item $D_{i+1}$, $S_{i+1}$, $Z_{i+1}$ and $B_{i+1}$ are bimeromorphic transform of $D$, $S$, $Z$ and $B$ onto $X_{i+1}$,
\item $f_i$ is $(K_{X_i}+B_i)$-crepant, i.e., for any resolution $p_i:W\to X_i, q_i:W\to X_{i+1}$ of the graph of $f_i$, $p_i^*(K_{X_i}+B_i)=q_i^*(K_{X_{i+1}}+B_{i+1})$,
\item  $K_{X_{i+1}}+B_{i+1}$  has the same numerical dimension and Kodaira dimension as $K_X+B$,
\item  $D_{i+1}\in |K_{X_{i+1}}+B_{i+1}|_{\mathbb{Q}}$ and $\mathrm{Supp}\, D_{i+1}=B_{i+1}$,
\item $D_{i+1}-(1-\lambda)B_{i+1}$ is effective with support equal to $B_{i+1}$,
\item  $({X_{i+1}},\lambda S_{{i+1}} + Z_{{i+1}})$ is $\mbQ$-factorial dlt and $(K_{X_{i+1}}, B_{i+1})$ is lc,
\item   $(K_{X_{i+1}}+B_{i+1})|_{T_{i+1}}\num 0$ for every component $T_{i+1}$ of $Z_{i+1}$, 
\item  $X_{i+1}\backslash B_{i+1}  \cong X\backslash B$, and
\item $S_{n+1}\neq 0$  and  $Z_{n+1}\cap S_{n+1}=\emptyset$.
\end{enumerate}
\end{lemma}

\begin{proof}
 Proceeding by induction assume that we have already constructed $f_{i-1}:X_{i-1}\to X_i$ satisfying the properties $(3)-(8)$. If $S_i\cap Z_{i}=\emptyset$ then we stop and set $i=n+1$. So assume that $S_i\cap Z_i\neq\emptyset$. Then there is some component $T_i$ of $Z_i$ such that $T_i\cap S_i\neq\emptyset$. Then we have that $(K_{X_i}+\lambda S_i + Z_i)|_{T_i} \equiv -(1-\lambda)S_i|_{T_i}$ is not pseudo-effective. We claim that $T_i$ is a Moishezon surface. Indeed, if $X$ is not uniruled, then this follows from \cite[Lemma 4.1]{CHP16}, and otherwise from Lemma \ref{lem-non-psef-implies-uniruled}. Therefore there is a curve $\Gamma_i\subseteq T_i$ such that $(K_{X_i}+\lambda B_i + Z_i)\cdot \Gamma_i<0$. Then from $(7)$ (with $i+1$ replaced by $i$) it follows that $(K_{X_i}+B_i)\cdot \Gamma_i=0$. Now by the cone theorem with respect to $K_{X_i}+\lambda S_i + Z_i$ (see Theorem \ref{thm-NA-cone-adjoint-pair}), we can decompose the class $[\Gamma_i]\in \overline{\mathrm{NA}}(X_i)$ as 
\[[\Gamma_i]=M_i+\sum_{j=1}^{m_i} [C_{ij}],\]
where $M_i$ is $(K_{X_i}+\lambda S_i + Z_i)$-nonnegative and $C_{ij}$ generate $(K_{X_i}+\lambda S_i + Z_i)$-negative extremal rays. Since $K_{X_i}+  B_i$ is nef and $(K_{X_i}+  B_i)\cdot \Gamma_i = 0$, we obtain that $(K_{X_i}+  B_i)\cdot C_{ij} = 0$ for all $j$. Therefore there is a $(K_{X_i}+\lambda S_i + Z_i)$-negative extremal ray $R_i$ which is $(K_{X_i}+B_i)$-trivial.\\
If $R_i$ is small, then the contraction of $R_i$ exists by Theorem \ref{thm-contraction-small-ray}, and the corresponding flip exists by Theorem \ref{thm:existence-of-flips}. If $R_i$ is divisorial, then there is a unique surface $V_i$ which contains and covered by all the curves in  $R_i$. This surface $V_i$ must be a component of $B_i$ as 
$K_{X_i}+\lambda B_i + Z_i \sim_{\mathbb{Q}} D_i-(1-\lambda)B_i\>0$, and the right hand side is supported on $B_i$. Now since $(X_i, \lambda_iS_i+Z_i)$ is a $\mbQ$-factorial dlt pair, $X_i$ has $\mbQ$-factorial klt singularity. Since $(X_i,B_i)$ is  lc, by adjunction $(V_i, 0)$ is slc (see \cite[Remark 1.2(5)]{Fuj00}). Hence the contraction of $V_i$ exists by Theorem \ref{thm-contraction-divisorial-nefdim=0} and Proposition \ref{prop-contraction-divisorial-nefdim=1}. Let $f_i: X_i \dashrightarrow X_{i+1}$ be the corresponding divisorial  contraction or flip. Since $R_i$ is  $(K_{X_i}+B_i)$-trivial, $f_i$ is $(K_{X_i}+B_i)$-crepant. Then by \cite[Lemma 3.2]{CHP16} we also get the properties (3)-(7) at the $(i+1)$-th level. Since the exceptional locus of $f_i$ is contained in $B_i$, we also have $X_{i+1}\backslash B_{i+1} \cong X_i\backslash B_i  \cong X\backslash B$; this is (8).\\

Next we note that, since $(X, \lambda S + Z)$ is a dlt pair, this MMP terminates by Theorem \ref{thm:termination}. At the end of this MMP by our hypothesis we have $Z_{n+1}\cap S_{n+1}=\emptyset$. So it remains to show that $S_{n+1}\neq 0$. To the contrary assume that $S_{n+1}=0$; then all the components of $S$ are contracted by the steps of this MMP. Let $f_i:X_i \to X_{i+1}$ be one of these steps which contracts the last component of $S$, i.e. $S_i\subset X_i$ is irreducible and contracted by $f_i$. We know that the corresponding extremal ray $R_i$ is $(K_{X_i}+B_i)$-trivial and $(K_{X_i}+\lambda S_i + Z_i)$-negative. Thus $S_i\cdot R_i>0$, this is a contradiction.
\end{proof}

\section{$\nu=1$ implies $\kappa > 0 $}

In this section we will prove following Theorem.
\begin{theorem}
\label{thm-nu=1-implies-kappa>0}
Let $(X,\D)$ be a klt pair such that $X$ is a $\mathbb{Q}$-factorial compact K\"ahler $3$-fold with terminal singularities. Assume that $K_{X}+(1-\epsilon)\D$  is nef for all $0\< \epsilon\ll 1$. If $\nu(K_{X}+\D) = 1$, then $\kappa(K_{X}+\D) = 1$.
\end{theorem}

 We need the following  reduction result first which is an analogue of \cite[Lemma 13.2, 13.3.1 and 13.3.2]{Kol92}.

\begin{lemma}
\label{lem-lc-modification-nu=1}
Let $(X,B)$ be a reduced dlt pair such that  $X$ is a $\mathbb{Q}$-factorial  compact K\"ahler $3$-fold with klt singularities. Assume that there is a nef $\mbQ$-divisor $D\in |K_{X}+B|_{\mathbb{Q}}$ such that $\mathrm{Supp}\, D= B$ and  that $\nu(K_X+B)=1$. Moreover, assume that $X\backslash B$ has terminal singularities. Then there is a $(K_X+B)$-crepant bimeromorphic map $\phi:(X',B')\bir (X,B)$ such that 
\begin{enumerate} 
\item $(X',B')$ is a $\mbQ$-factorial reduced lc pair, where $B':=\phi_*B$,
\item $(X', 0)$ is klt and $X'\setminus B'$ has terminal singularities,
\item there is a nef $\mbQ$-divisor $D'\in |K_{X'}+B'|_{\mathbb{Q}}$ such that $\mathrm{Supp}\, D'= B'$,
\item $\nu(K_{X'}+B')=1$ and $\kappa(K_X+B)=\kappa(K_{X'}+B')$, and
\item there is a connected component $S'$ of $B'$ which is irreducible. 
\end{enumerate}
\end{lemma}

\begin{proof}
Since $\nu(K_X+B)=1$, we have $(K_X+B)\cdot D = (K_X+B)^2 \equiv 0$. Since $K_X+B$ is nef, this implies that $(K_X+B)|_T \equiv 0$ for every component $T$ of $B$. Let  $S$ be an irreducible component of $B$ and $Z:=B-S$. Then applying Lemma \ref{lem-isolating-trivial-component} gives us the required bimeromorphic map $\phi:(X, B)\bir (X', B')$.
\end{proof}~\\

\begin{proof}[Proof of Theorem \ref{thm-nu=1-implies-kappa>0}]
We will follow the same strategy as in \cite[\S 8.A]{CHP16} which is same as in \cite[Chapter 13]{Kol92}. First replacing $(X, 0)$ by a terminal model as in Theorem \ref{thm:terminal-model} we may assume that $X$ has terminal singularities. Then combining \ref{lem-support-modification}, Proposition \ref{prop-lc-modification} and Lemma \ref{lem-lc-modification-nu=1} we reduce our problem to proving the following statement:

\begin{em}
	Let $(X, B)$ be a $\mbQ$-factorial compact K\"ahler $3$-fold reduced lc pair and $D\in|K_X+B|$ is a $\mbQ$-divisor satisfying the following properties:
	\begin{enumerate}
		\item $\Supp D=B$, 
		\item $(X, 0)$ is klt and $X\setminus B$ has terminal singularities,
		\item $K_X+B$ is nef and $\nu(K_X+B)=1$, and
		\item there is a connected component $S$ of $B$ such that $S$ is irreducible.
	\end{enumerate}
	
	Then $\kappa(K_X+B)\>1$.
\end{em}~\\

 The rest of the  proof works exactly the same way as in the proof of \cite[Theorem 8.1]{CHP16}.
\end{proof}

\section{Proof of   log abundance for log canonical pairs}
We first prove the klt log abundance for $\nu(X, K_X+\D)=1$.

\begin{theorem}
\label{thm-klt-log-abundance}
Let $(X,\D)$ be a klt pair and $K_X+\Delta$ is nef. If the numerical dimension $\nu(X, K_X+\D)=1$, then $K_X+\Delta$ is semi-ample. 
\end{theorem}

\begin{proof} 
First replacing $(X, \Delta)$ by its terminal model as in Theorem \ref{thm:terminal-model} we may assume that $X$ has $\mbQ$-factorial terminal singularities. 
	 Now by Theorem \ref{thm-reduction-2-cases} there is a $(K_X+\D)$-crepant  bimeromorphic model $(X',\D')$ such that
\begin{enumerate}
\item either there is a Mori fibration $f: X'\to Y$ such that $(K_{X'}+\D')\sim_{\mbQ, f} 0$, or
\item  $K_{X'}+(1-\epsilon)\D'$  is nef for all $0\< \epsilon \ll 1$.
\end{enumerate}
Moreover,  $X'$ has $\mathbb{Q}$-factorial  terminal singularities and $(X',\D')$  is a  klt  pair. Since the model is $(K_X+\D)$-crepant, we only need to prove that $K_{X'}+\D'$ is semi-ample. If we are in the first case, then by Theorem \ref{thm-base-boundary-effective}, there is an effective $\mbQ$-divisor  $B\>0$ on $Y$ such that $(Y,B)$ is klt and 
\[K_{X'}+\D' \sim_{\mathbb{Q}}  f^*(K_Y+B). \]
Now by the log abundance for K\"ahler surfaces as in Theorem \ref{thm:surface-abundance} it follows that $K_Y+B$ is semi-ample. Hence is $K_{X'}+\D'$ semi-ample.\\

If we are in the second case, then by Theorem \ref{thm-nu=1-implies-kappa>0}, $\kappa(X', K_{X'}+\Delta')=1$, and hence $K_{X'}+\Delta'$ is semi-ample by Corollary \ref{cor:lc-abundance}.

\end{proof}

In the remainder of this section we will establish the log abundance theorem for log canonical pairs. First we need the following reduction lemma.

\begin{lemma}\label{lem:lc-reduction}
        Let $(X, \Delta)$ be a log canonical pair and $K_X+\Delta$ is nef, where $X$ is a $\mbQ$-factorial compact K\"ahler $3$-fold with terminal singularities. 
        Then there exists a  $\mbQ$-factorial log canonical pair $(X', \Delta':=S'+B')$, where $S'=\lrd \D'\rrd$ and $X'$ is a compact K\"ahler   $3$-fold with terminal singularities, such that $K_{X'}+\Delta'$ is semi-ample if and only if $K_X+\Delta$ is semi-ample and  that one of the following   conditions is satisfied: 
        \begin{enumerate}
        \item  $K_{X'}+(1-\epsilon)\D'$ is nef for all $ 0\<\epsilon\ll 1$
        \item There is a $K_{X'}$-Mori fiber space $\pi:X'\to Y'$ over a normal compact K\"ahler surface  $Y'$ such that $(K_{X'}+\Delta')\sim_{\mbQ, \pi}0$ and some component of $S'$ dominates $Y'$.
        \item   There is a $K_{X'}$-Mori fiber space $\pi:X'\to Y'$ over a normal compact K\"ahler surface $Y'$ such that $(K_{X'}+\Delta')\sim_{\mbQ, \pi}0$, no component of $S'$ dominates $Y'$ and $K_{X'}+B'+\lambda S'$   klt for all  $0\<\lambda < 1$.
        \end{enumerate} 
Moreover, we have $\kappa(X, K_X+\D)=\kappa(X', K_{X'}+\D')$ and $\nu(X, K_X+\D)=\nu(X', K_{X'}+\D')$.	 
\end{lemma} 

\begin{proof}
	Using Theorem \ref{thm-reduction-2-cases} we may assume that one of the following holds:
\begin{enumerate}
\item[(i)] either there is a $K_X$-Mori fibration $f:X\to Y$ such that $(K_{X}+\D)\sim_{\mbQ, f}0$, or
\item[(ii)] $K_{X}+(1-\epsilon)\D$  is nef for all $0\< \epsilon \ll 1$.
\end{enumerate}

If we are in the case (ii), then set $(X',\D'):=(X,\D)$ and we are done. So now assume that we are in case (i). If there is a component of $\lrd \D \rrd $ which dominates $Y$, then we are again done by setting $(X',\D'):=(X,\D)$. So assume that $\lrd \D\rrd\neq 0$ and none of the components of $\lrd \D \rrd$ dominate $Y$. Let $(X_1,\D_1)$ be a dlt model of $(X,\D)$ as in Corollary \ref{cor:exist-terminal-dlt-modification}.  
Write $\D_1=S_1+ B_1$ with $S_1=\lrd \D_1 \rrd$. Then $K_{X_1}+\D_1$ is nef, $(K_{X_1}+\D_1)\sim_{\mbQ, Y}0$ and $K_X+\D$ is semi-ample if and only if $K_{X_1}+\D_1$ is semi-ample.  

We run a $(K_{X_1}+B_1)$-MMP over $Y$ as in Proposition \ref{pro:relative-projective-mmp} and obtain a pair $(X_2, \D_2= S_2+B_2)$ such that $(K_{X_2}+B_2)$ is nef over $Y$. 
Note that  this MMP is $(K_{X_1}+\D_1)$-trivial. 
Thus $(X_2, \D_2)$ is lc  and $(K_{X_2}+\D_2)\sim_{\mbQ, Y} 0$.

We run  a $K_{X_2}$-MMP over $Y$ with the scaling of $B_2$ and obtain a Mori fiber space  $f':X'\to Y'$. 
Then by Lemma \ref{lem:directed-relative-mmp} every step of this MMP is $(K_{X_2}+B_2)$-trivial. Hence  $(X', B')$ is klt, where $B'$ is the bimeromorphic transform of $B_2$.
 We also remark that  $(K_{X'}+B')\sim_{\mbQ, Y'} 0$ since $X'\to Y'$ is a Mori fiber space. 
Let  $S'$ be the bimeromorphic transform of $S_2$ and $\Delta':=B'+S'$.  
Since this MMP is also $(K_{X_2}+\D_2)$-trivial, we see that  $(X',\D')$ is lc, $K_{X'}+\D'$ is nef  and $(K_{X'}+\D')\sim_{\mbQ, Y'} 0$. Moreover, $K_{X'}+\D'$ is semi-ample if and only if $K_X+\D$ is. 

Now we claim that $(X', B'+\lambda S')$ is klt for all $0\<\lambda<1$. To the contrary assume that there is a $\lambda_0\in (0, 1)$ such that $(X', B'+\lambda_0 S')$ is not klt. 
Then there is an exceptional divisor $E$ over $X'$ such that $a(E, X', B'+\lambda_0 S')\<-1$. Note that since $(X', B')$ has $\mbQ$-factorial klt singularities, it follows that $\Center_{X'}(E)\subset \Supp S'$. Therefore $a(E, X', B'+ S')<-1$, this is a contradiction, since $(X', B'+S')$ has lc singularities. This completes the proof.

\end{proof}~\\

Finally we will prove our main theorem.

\begin{proof}[Proof of Theorem \ref{thm:lc-log-abundance}]
	First replacing $(X, \D)$ by a dlt model as in Corollary \ref{cor:exist-terminal-dlt-modification}, we may assume that $(X, \D)$ is a dlt pair such that $X$ has $\mbQ$-factorial terminal singularities. Now recall that the non-vanishing is proved in Theorem \ref{thm-non-vanishing}. When $\kappa(X, K_X+\Delta)>0$, the semi-ampleness of $K_X+\Delta$ is proved in Corollary \ref{cor:lc-abundance}. If $\nu(X, K_X+\Delta)=0$, i.e. $K_X+\Delta\num 0$, then from the non-vanishing theorem it follows that $K_X+\Delta\sim_{\mbQ} 0$. If $\nu(X, K_X+\Delta)=3$, then $K_X+\Delta$ is big; in particular, $X$ is a compact K\"ahler Moishezon space with rational singularities. Hence $X$ is a projective variety in this case by \cite[Theorem 1.2]{Nam02}, and the semi-ampleness follows from \cite[Theorem 1.1]{KMM94}.\\
 So the only remaining case is $\nu(X, K_X+\Delta)=1$. Using Lemma \ref{lem:lc-reduction} we may assume that $(X,\D)$ satisfies one of the conditions in this lemma. Note that, this in particular implies that $(X, \D)$ is a lc pair such that $X$ has $\mbQ$-factorial terminal singularities.\\
Now suppose that we are in the first  case of Lemma \ref{lem:lc-reduction}.  If $\D=0$, then the theorem follows from \cite[Theorem 8.1]{CHP16}, as $X$ has terminal singularities. We may then assume that $\D\neq 0$. Then $(X, (1-\epsilon)\D)$ is klt and $K_X+(1-\epsilon)\D$ is nef for all $0<\epsilon\ll 1$. By Corollary \ref{cor:lc-abundance} it is enough to show that $\kappa(X, (1-\epsilon)\D)>0$ for some $0<\epsilon<1$, because then we will have $\kappa(X, K_X+\Delta)>0$. To this end, we will first show that $\nu(X, K_X+(1-\epsilon_0)\D)=1$ for some $0<\epsilon_0\ll 1$. Observe that if $K_X+(1-\epsilon)\D\num 0$ for all $0<\epsilon\ll 1$, then taking limit as $\epsilon\to 0$ we see that $K_X+\D\num 0$, this is a contradiction. Thus there is a $0<\epsilon_0\ll 1$ such that $\nu(X, K_X+(1-\epsilon_0)\D)\>1$. Next we will show that $\nu(X, K_X+(1-\epsilon_0)\D)<2$. If $\nu(X, K_X+(1-\epsilon_0)\D)=3$, then $K_X+(1-\epsilon_0)\D$ is big and then so is $K_X+\D$; this is a contradiction, since $\nu(X, K_X+\D)=1$. So by contradiction assume now that $\nu(X, K_X+(1-\epsilon_0)\D)=2$. By non-vanishing Theorem \ref{thm-non-vanishing} there is an effective $\mbQ$-divisor $D\>0$ on $X$ such that $K_X+(1-\epsilon_0)\Delta\sim_{\mbQ} D$. Then $D':=D+\epsilon_0\Delta\sim_{\mbQ}K_X+\Delta$. Since the numerical dimension $\nu(X, D)=\nu(X, K_X+(1-\epsilon_0)\Delta)=2$, there is a K\"ahler class $\omega_0$ on $X$ such that $D^2\cdot \omega_0>0$. However, since $\nu(X, D')=\nu(X, K_X+\Delta)=1$, we have $D'^2\cdot \omega_0=0$. Now 
\begin{align*}
	D'^2\cdot\omega_0 &=(D+\epsilon_0\Delta)^2\cdot\omega_0\\
					&= D\cdot(D+\epsilon_0\Delta)\cdot\omega_0+\epsilon_0(D+\epsilon_0\Delta)\cdot\Delta\cdot\omega_0\\
					 &=\underbrace{D^2\cdot\omega_0}_{>0}+\epsilon_0\underbrace{(D\cdot\omega_0\cdot\Delta)}_{\>0}+\epsilon_0\underbrace{(D'\cdot\omega_0\cdot\Delta)}_{\>0}\\
					 & >0.
\end{align*}
This is a contradiction. Therefore $\nu(X, K_X+(1-\epsilon_0)\D)=1$ for some $0<\epsilon_0\ll 1$. Then by Theorem \ref{thm-klt-log-abundance}, $K_X+(1-\epsilon_0)\D$ is semi-ample, and hence $\kappa(X, K_X+(1-\epsilon_0)\D)=1$.\\

Assume now that we are in the second  case of Lemma \ref{lem:lc-reduction} and $\pi:X\to Y$ is the $K_X$-Mori fiber space as in there, and $\Delta=B+S$, where $S=\lrd \D\rrd$. Suppose that $\msL$ is a line bundles on $Y$ such that $\mcO_X(m(K_{X}+\Delta))\cong \pi^*\msL$ for some $m\in\mbN$ sufficiently divisible. Then it is enough to prove that $\msL$ is semi-ample. Let $T$ be a component of $S$ which dominates $Y$. Then $(T^n, \Delta_{T^n})$ is lc, where $\tau:T^n\to T$ is the normalization and  $K_{T^n}+\Delta_{T^n}\sim_{\mbQ}(K_X+\Delta)|_{T^n}$ is defined by adjunction. Then by Theorem \ref{thm:surface-abundance}, $K_{T^n}+\Delta_{T^n}$ is semi-ample.  Since $(K_{T^n}+\D_{T^n})\sim_{\mbQ}(\pi|_T\circ\tau)^*\msL$, from Lemma \ref{lem:semiample-preserved} it follows that $\msL$ is semi-ample, and hence $K_X+\Delta$ is semi-ample.\\

Now  assume that we are in the third case of Lemma \ref{lem:lc-reduction} and $\pi:X\to Y$ is the $K_X$-Mori fiber space as in there, and $\Delta=B+S$, where $S=\lrd \D\rrd$. If $S=0$, then $(X,\D)$ is klt and we can apply Theorem \ref{thm-klt-log-abundance}. So from now on we assume that $S\neq 0$. We will show that $K_X+\D$ has positive Kodaira dimension. Since $S$ is vertical over $Y$, by Lemma \ref{lem:vertical-pullback} there is an effective non-zero $\mbQ$-Cartier divisor $D\>0$ on $Y$ such that $f^*D=S$. Choose $0<\epsilon<1$ so that $(X, B+(1-\epsilon)S)$ is klt as in the hypothesis. Note that $(K_X+B+(1-\epsilon)S)\sim_{\mbQ, \pi} 0$. Then from Theorem \ref{thm-base-boundary-effective} and its proof it follows that there is an effective $\mbQ$-divisor $\Theta\>0$ on $Y$ such that $K_X+B + (1-\epsilon) S \sim_{\mathbb{Q}} f^*(K_Y+\Theta + (1-\epsilon) D)$ and that $(Y, \Theta+ (1-\epsilon) D)$ is klt. We also have 
\begin{equation}\label{eqn:nef-comparison}
	(K_X+\D)\sim_{\mbQ}\pi^*(K_Y+\Theta+D),
\end{equation}
 and thus $K_Y+\Theta+D$ is nef on $Y$. Note that $K_Y$ is pseudo-effective, since $Y$ is not uniruled. Then from MMP and abundance for compact Kähler surfaces it follows that the Kodaira dimension $\kappa(Y)\>0$. We will show that $\kappa(Y, K_Y+\Theta+(1-\epsilon)D)>0$. To that end, we run a $(K_Y+\Theta+(1-\epsilon)D)$-MMP with scaling of $\epsilon D$. Assume that this MMP terminates with a minimal model $(Y', \Theta'+(1-\epsilon)D')$. Then by Theorem \ref{thm:surface-abundance}, $K_{Y'}+\Theta'+(1-\epsilon)D'$ is semi-ample. Observe that $D'\neq 0$, and $\kappa(Y')\>0$ as $\kappa(Y)\>0$. Now if $\kappa(Y', K_{Y'}+\Theta'+(1-\epsilon)D')=0$, then $-K_{Y'}\sim_{\mbQ} \Theta+(1-\epsilon)D'\>0$. This implies that $K_{Y'}\sim_{\mbQ} 0$ as $\kappa(Y')\>0$. Then $(\Theta'+(1-\epsilon)D')\sim_{\mbQ} 0$, this is a contradiction, since $\Theta'+(1-\epsilon)D'$ is a non-zero effective divisor on a compact K\"ahler variety $Y'$. Thus $\kappa(Y', K_Y+\Theta'+(1-\epsilon)D')>0$, and hence $\kappa(Y, K_Y+\Theta+(1-\epsilon)D)>0$. Then from \eqref{eqn:nef-comparison} it follows that $\kappa(X, K_X+\Delta)>0$, and we are done by Corollary \ref{cor:lc-abundance}.

\end{proof}

\noindent
{\bf Data Availability.} Data sharing is not applicable to this article as no datasets were generated or analyzed during the current study.\\

\noindent
{\bf Conflict of Interest Statement.} On behalf of all authors, the corresponding author states that there is no conflict of interest.

\bibliographystyle{hep}
\bibliography{references}

\end{document}